%% file: main.tex
\documentclass[indent=latex,review=false]{acmhomework}

\input{macros}
\usepackage{quiver}

\title{Extensional concepts in intensional type theory, revisited}

\author{Krzysztof Kapulkin and Yufeng Li}

\begin{abstract}
  Revisiting a classic result from M.~Hofmann's dissertation, we give a direct
  proof of Morita equivalence, in the sense of V.~Isaev, between extensional
  type theory and intensional type theory extended by the principles of
  functional extensionality and of uniqueness of identity proofs.
\end{abstract}

\begin{document}

\maketitle

\input{intro}
\input{background}
\input{plan}
\input{cxlcat-cells.tex}
\input{w-ett}
\input{c-ett}
\input{c-ett-logic}
\input{c-ett-lift}
\input{conclusion}

\printbibliography

\end{document}

%% file: macros.tex
\renewcommand{\ITT}{\textsf{ITT}}
\newcommand{\ETT}{\textsf{ETT}}
\newcommand{\Ext}{\textsf{Ext}}
\newcommand{\ext}{\textsf{ext}}
\newcommand{\UIP}{\textsf{UIP}}

\renewcommand{\CxlCat}{\textsf{CxlCat}}
\renewcommand{\CwA}{\textsf{CwA}}
\DeclareMathOperator{\ft}{ft}
\let\Ty\undefined\DeclareMathOperator{\Ty}{Ty}

\renewcommand{\pair}{\textsf{pair}}
\newcommand{\sigmarec}{\textsf{split}}

\newcommand{\app}{\textsf{app}}

\renewcommand{\Id}{\mathsf{Id}}
\newcommand{\unitrec}{\textsf{unitrec}}

\renewcommand{\Type}{\textsf{Type}}
\renewcommand{\Term}{\textsf{Term}}
\renewcommand{\Ctx}{\textsf{Ctx}}
\newcommand{\tosimeq}{\xrightarrow{\simeq}}

\makeatletter
\newcommand{\nameditem}[1]{%
\item[#1]\protected@edef\@currentlabel{#1}%
}
\makeatother

\DeclareMathOperator{\core}{core}

\newcommand{\Norm}{\mathsf{Norm}}

\newcommand{\Nf}{\mathsf{Nf}}

\renewcommand{\Ty}{\mathsf{Ty}}
\renewcommand{\Id}{\mathsf{Id}}

\newcommand{\NCwA}{\mathsf{NCwA}}

\newcommand{\sz}{\mathsf{sz}}

%% file: intro.tex
\section*{Introduction}

In his Ph.D.~dissertation \cite{hofmann:extensional-concepts,hofmann:extensional-constructs}, Hofmann constructs an interpretation of extensional type theory in intensional type theory, subsequently proving a \emph{conservativity} result of the former over the latter extended by the principles of functional extensionality and of uniqueness of identity proofs (UIP), cf.~\cite[\S3.2]{hofmann:extensional-concepts}.
Interestingly, Hofmann's proof is ``stronger'' than the statement of his theorem, as the requisite language in which to speak of conservativity and equivalence of dependent type theories did not exist at the time.

A major insight came as a result of the development of homotopy type theory.
In particular, Voevodsky's definition, ca.~2009, of when a morphism in a model of dependent type theory is a \emph{weak equivalence} allows one to consider different homotopy-theoretic structures both within such models and the categories thereof.

In \cite{kapulkin-lumsdaine-18}, it was observed that the category of models of a dependent type theory carries the structure of a left semi-model category.
This structure was subsequently used by Isaev \cite{isaev18} to define a \emph{Morita equivalence} of dependent type theories.
In essence, two theories are Morita equivalent if their categories of models are equivalent in a suitable ``up-to-homotopy'' sense.
More precisely, a Morita equivalence between theories is a translation between them that induces a Quillen equivalence (the correct notion of equivalence for left semi-model structures) between their left semi-model categories of models.
Perhaps unsurprisingly, Isaev cites Hofmann's theorem as one of the motivating examples behind his definition, without actually proving it to be one.

In the present work, we give a direct proof of Morita equivalence between extensional type theory and intensional type theory by principles of functional extensionality and of uniqueness of identity proofs.
While Hofmann proves that the initial models of these theories are suitably equivalent, we generalize this result to all possible extensions of the base theories by types and terms, including propositional equalities.
In the homotopy-theoretic terms, these are exactly the \emph{cofibrant} extentions.

Therefore, thanks to proving Morita equivalence, one does not need to prove an analogue of Hofmann's result for any new extension but instead appeal to our result addressing all extensions once and for all.
As new variants and extensions of intensional type theory are constantly proposed, this reduces the burden of proving their expected properties by making what should be formal formal.

\subsection*{Outline}
Our proof follows Hofmann's quite closely but requires a major innovation to account for all possible (cofibrant) extensions.
In \cite{hofmann:extensional-constructs}, Hofmann describes a classs of context morphisms termed \emph{propositional isomorphisms} by inspecting the outer-most type former in (the last type of) each context and collapses these maps to identities, thus obtaining a functor from the syntactic model of intensional type theory to that of extensional type theory.
Our general approach is identical: in \cref{sec:w-ett}, we describe the \emph{extensional kernel} of a cofibrant model of intensional type theory, which we then collapse in \cref{sec:c-ett} to obtain a (cofibrant) model of extensional type theory.
The verification that the construction indeed produces a model of extensional type theory and yields Morita equivalence between the theories is presented in \cref{sec:c-ett-logic} and \cref{sec:c-ett-lift}, respectively.

A major innovation is required in \cref{sec:cofib}.
Because of working with syntactic categories, Hofmann is able to inspect the outer-most constructor in each context.
This property need not hold in every model, but it does hold in the cofibrant ones (which is proven in \cref{sec:cofib}), which is therefore sufficient to follow the remainder of Hofmann's strategy to establish a Morita equivalence.

We give a more detailed overview of the proof in \Cref{sec:outline}, including a detailed comparison with Hofmann's proof, after reviewing the necessary background  in \Cref{sec:background}.

\subsection*{Related work}
While the concept of a Morita equivalence in mathematics goes back a long way, it is a relatively new concept in the context of dependent type theories, cf.~\cite{isaev18}.
The only literature on the topic that we are aware of is \cite{bocquet:coherence}, where Bocquet developed an impressively general theory for proving the existence of Morita equivalence in situations where extensional principles are involved.
Our approach, in contrast, is direct and entirely self-contained, although the result can be deduced from Bocquet's.

\subsubsection*{Acknowledgements}
We thank the anonymous reviewers for their insightful comments that identified
an error in the original submission and their suggestions that helped improve
the quality of presentation of the manuscript.
We would also like to thank Zhixuan Yang for various helpful comments regarding
\Cref{sec:cofib}.

%%% Local Variables:
%%% mode: latex
%%% TeX-master: "./main.tex"
%%% TeX-engine: xetex
%%% End:

%% file: background.tex
\section{Background} \label{sec:background}
This section reviews the required background on
the homotopical structure of the category of models of a given type
theory.
We briefly recall the definition of a contextual category as well as
their underlying homotopy-theoretic properties and give a definition of what it
means for two type theories to be Morita equivalent.
Our presentation of contextual categories follows
\cite{kapulkin-lumsdaine-18,simplicial-uf}, whereas the definition of Morita equivalence is then a direct translation of \cite{isaev18} to our setting.

\subsection{Contextual Categories}
A model of type theory in the present context is taken to be a
\emph{contextual category}, defined as follows.

\begin{definition}[Contextual Category]
  A contextual category $\bC$ consists of a category $\bC$ along with the
  following data:
  \begin{itemize}
    \item A grading on objects
    $\ob\bC = \coprod_{n \in \bN} \ob_n{\bC}$.
    \item An object $1 \in \ob_0{\bC}$.
    \item A \emph{father} function
    $\ft_n \colon \ob_{n+1}\bC \to \ob_n\bC$ along with
    \emph{dependent projection}
    $\pi_\Gamma \colon \Gamma \twoheadrightarrow \ft{\Gamma}$ for each
    $n \in \bN$ and $\Gamma \in \ob_{n+1}\bC$.
    \item For each
    $\Gamma \in \ob_{n+1}\bC$ and $f \colon \Delta \to \ft{\Gamma}$,
    an object $f^*\Gamma$ along
    with the \emph{connecting map}
    $f.\Gamma \colon f^*\Gamma \to \Gamma$.
  \end{itemize}
  subject to the following conditions:
  \begin{itemize}
    \item The empty context $1 \in \ob_0\bC$ is
    terminal and the unique object in $\ob_0\bC$.
    \item For all
    $f \colon \Delta \to \ft\Gamma$, the object $f^*\Gamma$ is such that
    $\ft{f^*\Gamma} = \Delta$ and the following square
    \begin{equation*}
      \begin{tikzcd}
        f^*\Gamma
        \ar[d, two heads, "\pi_{f^*\Gamma}"']
        \ar[r, "f.\Gamma"]
        \ar[rd, phantom, "\lrcorner" {pos=0}]
        & \Gamma \ar[d, two heads, "\pi_\Gamma"] \\
        \Delta \ar[r, "f"] & \ft\Gamma
      \end{tikzcd}
    \end{equation*}
    is a pullback.
    Further, these operations are strictly functorial, i.e.,
    \begin{itemize}
      \item $\id^*\Gamma = \Gamma$ and
      $\id.\Gamma = \id$.
      \item For $g \colon \Theta \to \Delta$
      and $f \colon \Delta \to \ft{\Gamma}$ as above, the following
      equalities hold:
      \begin{align*}
        (fg)^*\Gamma = g^*f^*\Gamma &&
        (fg).\Gamma = g.f^*\Gamma \cdot f.\Gamma
      \end{align*}
    \end{itemize}
  \end{itemize}
\end{definition}
When working with contextual categories, dependent projections will be denoted by the two-headed arrow, e.g., $\pi_\Gamma \colon \Gamma \twoheadrightarrow \ft{\Gamma}$ for $\Gamma \in \ob_{n+1}\bC$.
Moreover, we will often omit the subscript, writing $\pi \colon \Gamma \twoheadrightarrow \ft{\Gamma}$ instead of $\pi_\Gamma$.

For $\Gamma \in \ob_n\bC$, we write $\Ty(\Gamma)$ for the set of objects
$A \in \ob_{n+1}\bC$ such that $\ft{A} = \Gamma$.
We refer to the elements of $\Ty\Gamma$ as
the \emph{types in context $\Gamma$}.
When $\ft{A} = \Gamma$, we write $\Gamma.A$ for $A$.
In this case, for $f \colon \Delta \to \Gamma$, the
canonical substitution pullback is denoted
\begin{equation*}
  \begin{tikzcd}
    \Delta.f^*A
    \ar[d, two heads, "\pi_{f^*A}"']
    \ar[r, "f.A"]
    \ar[rd, phantom, "\lrcorner" {pos=0}]
    & \Gamma.A \ar[d, two heads, "\pi_A"] \\
    \Delta \ar[r, "f"] & \Gamma
  \end{tikzcd}
\end{equation*}
For $A \in \Ty\Gamma$, we often write $\Gamma.A.A$ to mean
$\Gamma.A.\pi^*A$, omitting the explicit weakening substitution
$\pi^*$ for readability.
Furthermore, the set $\Tm_\Gamma A$ consists of sections
$\Gamma \to \Gamma.A$ of the projection
$\Gamma.A \twoheadrightarrow \Gamma$, and is referred to as the
\emph{terms of $A$}.

More generally, if $\Theta \in \ob_{n+m}{\bC}$ is such that
$\ft^m{\Theta} = \Gamma$ then $\Theta$ is written as $\Gamma.\Theta$ for some
$\Theta$ depending on $\Gamma$ so that there is a chain of $m$ composition of
projections $\Gamma.\Theta \twoheadrightarrow \cdots \twoheadrightarrow \Gamma$.
For $f \colon \Delta \to \Gamma$, the $m$-times vertical composition
of the substitution pullbacks gives rise to the canonical pullback of
$\Gamma.\Theta \twoheadrightarrow \Gamma$ along $f$ as
$\Delta.f^*\Theta$ with connecting map $f.\Theta$.

Contextual categories can be equipped with additional structure corresponding to
the logical structure present in type theory, e.g., $\Pi$-types (or
$\Pi_{\Ext}$-types, if we are additionally assuming function extensionality),
$\Sigma$-types, or $\Id$-types, as defined in \cite[Appendix B]{simplicial-uf}.
Denote by $\CxlCat$ the category of contextual categories and $\CxlCat_{\ITT}$
the category consisting of contextual categories equipped with $\Id$-,
$\Pi_{\Ext}$ and $\Sigma$-types as objects and maps $F \colon \bC \to \bD$
between such contextual categories are those maps that preserve all structure of
contextual categories (the grading, the empty context, the context projections
and the substitution pullbacks) and the logical structures.

\subsection{Uniqueness of Identity Proofs and Extensional Type Theories}\label{subsec:uip-ett}
Contextual categories equipped with
$\Id$-type structure also give rise to a way of formalising the
\emph{uniqueness of identity proofs (UIP)} and \emph{equality
  reflection} rules.

For the rest of this section, let $\bC \in \CxlCat_{\ITT}$ be a
contextual category with structures corresponding to $\Id$ types.
We will make extensive use of Garner's identity contexts \cite[Proposition 3.3.1]{garner09}.
For each context extension
$\Gamma.\Theta \in \ob_{n+m}{\bC}$ of
$\Gamma \in \ob_n{\bC}$ there is an \emph{identity context}
$\Gamma.\Theta.\Theta.\Id_\Theta$ with corresponding reflexivity
constructors and elimination rules generalising those of identity
types $\Gamma.A.A.\Id_A$ for types $A \in \Ty\Gamma$.

\begin{definition}[{\cite[Definition 2.6]{kapulkin-lumsdaine-18}}]\label{def:kl-2.6}
  Given $f,g \colon \Gamma \to \Delta \in \bC$, a \emph{type-theoretic
    homotopy} $e \colon f \simeq g$ is a factorisation $e$ as follows:
  \begin{equation*}
    \begin{tikzcd}
     & \Delta.\Delta.\Id_\Delta \ar[dr, two heads] & \\
      \Gamma \ar[rr, "{(f,g)}"]  \ar[ru, "e"]
      & & \Delta.\Delta
    \end{tikzcd}
  \end{equation*}
\end{definition}
\begin{definition}[$\UIP$]\label{def:uip}
  The structure of \emph{uniqueness of identity proofs ($\UIP$)} on
  contextual category $\bC$ consists of an assignment, for each pair
  of sections $s_1,s_2 \colon \Gamma.A.A \to \Gamma.A.A.\Id_A$, a
  choice of a section $\UIP_A(s_1,s_2) \colon s_1 \simeq s_2$, stably in $\Gamma$, i.e., for all $f \colon \Delta \to \Gamma$, we have
  \begin{align*}
    f^*\UIP_A(s_1,s_2) = \UIP_{f^*A}(f^*s_1,f^*s_2)
    \colon f^*s_1 \simeq f^*s_2
  \end{align*}
\end{definition}
The category $\CxlCat_{\ITT+\UIP}$ has as objects contextual categories $\bC \in \CxlCat_\ITT$
along with an $\UIP$-structure on $\bC$ and maps $\bC \to \bD$ are
those functors in $\CxlCat_\ITT$ that additionally preserve the $\UIP$-structure.

\begin{definition}[Equality Reflection]\label{def:eq-reflect}
  The contextual category $\bC$ supports the \emph{equality
    reflection} rule when, for any two sections
  $f, g \colon \Gamma \to \Gamma.A$, if there is a homotopy
  $e \colon f \simeq g$, then $f = g$.
\end{definition}
The category $\CxlCat_{\ETT}$ is the full subcategory of
$\CxlCat_{\ITT}$ spanned by those contextual categories satisfying
equality reflection.

Note that if $\bC$ supports the equality reflection rule, then it also
admits an $\UIP$-structure by taking each $\UIP_A(s_1,s_2)$ in
\Cref{def:uip} as the reflexivity map.
This gives rise to a forgetful functor
$\abs{-} \colon \CxlCat_\ETT \to \CxlCat_{\ITT+\UIP}$ that is identity
on objects and maps.

\subsection{Fibration Categories}
Objects of the categories $\CxlCat_\ITT$, $\CxlCat_{\ITT+\UIP}$ and
$\CxlCat_{\ETT}$ admit structures of a fibration category, as defined
in this section.

\begin{definition}[{\cite[\S{I.1}]{brown73}}]
  A fibration category $\bC$ consists of a category $\bC$ with two
  wide subcategories: $\McF$, whose morphisms are called \emph{fibrations}, and $\McW$, whose morphisms are called
  \emph{weak equivalences}, subject to the following conditions:
  \begin{itemize}
    \item $\McW$ satisfies the 2-of-6 property.
    \item All isomorphisms are trivial fibrations, which are maps that
    are both weak equivalences and fibrations.
    \item Pullbacks along fibrations exist.
    Fibrations and trivial fibrations are closed under pullbacks.
    \item There is a terminal object $1 \in \bC$.
    All objects are fibrant, in the sense that each $\Theta \to 1$ is a fibration.
    \item All maps can be factored as a weak equivalence followed by a
    fibration.
  \end{itemize}
  When $\McW$ is clear from context, we use the same notation as in
  \Cref{def:kl-2.6} and write $w \colon \Gamma_1 \simeq \Gamma_2$ to mean
  $w \in \McW$.
\end{definition}

\begin{theorem}[{\cite[Theorem 3.2.5]{akl15}}]
  Fix $\bC \in \CxlCat_\ITT$ (respectively,
  $\CxlCat_{\ITT+\UIP}, \CxlCat_{\ETT}$).
  Put $\McW$ to be those $w \colon \Gamma \to \Delta \in \bC$ with
  $g_1,g_2 \colon \Delta \to \Gamma$ such that there are homotopies
  $e_1 \colon g_1w \simeq \id$ and $e_2 \colon wg_2 \simeq \id$.
  Put $\McF$ to be those maps isomorphic to compositions of
  projections.
  Then $\bC$ along with $\McW$ and $\McF$ forms a fibration category.
  \qed
\end{theorem}

\begin{definition}\label{def:glue}
  Let $\bC \in \CxlCat_{\ITT+\UIP}$ (respectively, $\CxlCat_{\ETT}$).
  Take $w_\Gamma \colon \Gamma_1 \simeq \Gamma_2$ and
  $w_A \colon \Gamma_1.A_1 \simeq \Gamma_2.A_2$ and
  $w_\Delta \colon \Delta_1 \simeq \Delta_2$ and let there be maps
  $f_i \colon \Delta_i \to \Gamma_i$ for $i=1,2$ as below such that there is a
  homotopy $H \colon w_\Gamma \cdot f_1 \simeq f_2 \cdot w_\Delta$.
  \begin{equation*}
    \begin{tikzcd}[cramped, sep=small]
      & \Delta_2.f_2^*A_2 \ar{rr} \ar[two heads]{dd}
      \ar[phantom]{rrdd}[pos=0em]{\lrcorner}
      & & \Gamma_2.A_2 \\
      \Delta_1.f_1^*A_1
      \ar[rrru, phantom, "\simeq"{pos=0.4}]
      \ar[dashed]{ur}{\gamma}[description, inner sep=1]{\simeq}
      \ar[crossing over]{rr}
      \ar[phantom]{rrdd}[pos=0em]{\lrcorner}
      \ar[two heads]{dd}
      & & \Gamma_1.A_1 \ar[]{ur}[description, inner sep=1]{\simeq}  \\
      & \Delta_2 \ar{rr}[near start]{f_2}
      & & \Gamma_2 \ar[crossing over, twoheadleftarrow]{uu} \\
      \Delta_1 \ar{rr}[swap]{f_1}
      \ar[rrru, phantom, "\simeq"]
      \ar[]{ur}[description, inner sep=1]{\simeq}
      & & \Gamma_1 \ar[crossing over, twoheadleftarrow]{uu}
      \ar[]{ur}[swap]{\simeq}
    \end{tikzcd}
  \end{equation*}
  Let $w_\Gamma,w_A$ and $w_\Delta$ respectively be given by
  \begin{align*}
    (\vec{x}_1:\Gamma_1) &\vdash t_\Gamma(\vec{x}_1) : \Gamma_2 \\
    (\vec{x}_1:\Gamma_1)(a_1:A_1) &\vdash t_\Gamma(\vec{x}_1), t_A(a_1) : \Gamma_2.A_2 \\
    (\vec{y}_1:\Delta_1) &\vdash t_\Delta(\vec{y}_1) : \Delta_2
  \end{align*}
  so that $H \colon w_\Gamma \cdot f_1 \simeq f_2 \cdot w_\Delta$ is given by
  \begin{align*}
    (\vec{y}_1:\Delta_1) \vdash
    t_H(\vec{y}_1)
    :
    \Id_{\Gamma_2}(t_\Gamma f_1(\vec{y}_1),f_2t_\Delta(\vec{y}_1))
  \end{align*}
  By $\transport$, there is an equivalence over $\Delta_1$
  \begin{align*}
    f_1^*A_1 \xrightarrow{f_1^*t_A} f_1^*t_\Gamma^*A_2 \xrightarrow{\transport(H)} t_\Delta^*f_2^*A_2
  \end{align*}
  which gives rise to an equivalence
  $\gamma(w_A,w_\Gamma,w_\Delta; f_1,f_2,H) \colon \Delta_1.f_1^* \to
  \Delta_2.f_2^*A_2$ over $w_\Delta \colon \Delta_1 \to \Delta_2$.
\end{definition}

\begin{lemma}\label{lem:glue}
  The map
  $\gamma(w_A, w_\Gamma, w_\Delta; f_1, f_2, H) \colon
  \Delta_1.f_1^*A_1 \to \Delta_2.f_2^*A_2$ from \Cref{def:glue} is an
  equivalence.
  Moreover, if there is $H' \colon w_\Gamma \cdot f_1 \simeq f_2 \cdot w_\Delta$
  homotopic to $H$ then
  $\gamma(w_A, w_\Gamma, w_\Delta; f_1, f_2, H) \simeq \gamma(w_A, w_\Gamma,
  w_\Delta; f_1, f_2, H')$.
\end{lemma}
\begin{proof}
  By the Gluing Lemma \cite[Application 12.11]{DHK},
  $\gamma(w_A, w_\Gamma, w_\Delta; f_1, f_2, \refl) \colon \Delta_1.f_1^*A_1 \to
  \Delta_2.f_2^*A_2$ from \Cref{def:glue} is an equivalence, so by
  $\MsJ$-elimination,   $\gamma(w_A, w_\Gamma, w_\Delta; f_1, f_2, H) \colon
  \Delta_1.f_1^*A_1 \to \Delta_2.f_2^*A_2$
  equivalence.
  By $\MsJ$-elimination on the homotopy $H' \simeq H$, one obtains a homotopy
  $\gamma(w_A,w_\Gamma,w_\Delta;f_1,f_2,H) \simeq
  \gamma(w_A,w_\Gamma,w_\Delta;f_1,f_2,H)$.
\end{proof}
Often, if $H=\refl$ in \Cref{def:glue} then we write
$\gamma(w_A,w_\Gamma,w_\Delta;f_1,f_2)$ to mean
$\gamma(w_A,w_\Gamma,w_\Delta;f_1,f_2,\refl)$.

\subsection{Homotopical Structure on the Category of Contextual
  Categories}
We now turn our attention towards defining the notion of \emph{Morita equivalence}, which applies for
adjunctions between categories with \emph{left semi-model structure}.
This allows us to formalise the equivalence between $\ITT+\UIP$ and $\ETT$ as a
Morita equivalence.

We begin by recalling the definition of a left semi-model structure.

\begin{definition}[{\cite[Definition 1(I)]{spitzweck:operads-algebras-modules}}]
  Let $\McE$ be a bicomplete category with three classes of maps:
  \emph{weak equivalences} $\McW$, \emph{fibrations} $\McF$, \emph{cofibrations}
  $\McC$.
  The above data form a \emph{left semi-model structure} when:
  \begin{itemize}
    \item All three classes are closed under retracts; $\McW$ satisfies the
    2-of-3 property; $\McF$ and $\McC$ are stable under pullback.
    \item Cofibratons have the left-lifting property against acyclic fibrations;
    acyclic cofibrations with cofibrant source have the left-lifting property
    against fibrations.
    \item Every map can be factored into a cofibration followed by an acyclic
    fibration; and maps with cofibrant source can be factored into an acyclic
    cofibration followed by a fibration.
  \end{itemize}
\end{definition}

We are now ready to define the notion of a \emph{Morita equivalence}.

\begin{definition}
  [{Morita Equivalence \cite[Definition 2.4]{isaev18}}]
  \label{def:morita}
  Let $\bT_1, \bT_2$ be theories whose categories of models $\bM_1, \bM_2$ are
  equipped with left semi-model structures.
  %
  % Denote by
  % $(\bM_1)_\McC,(\bM_2)_\McC$ as their respective categories of cofibrant objects.
  % %
  The theories $\bT_1$ and $\bT_2$ are \emph{Morita equivalent} when there is an
  adjunction
  \begin{equation*}
    \begin{tikzcd}[cramped]
      \bM_1
      \ar[r, {yshift=3}, "L"]
      \ar[r, phantom, "\bot" {font=\tiny}]
      & \bM_2
      \ar[l, {yshift=-3}, "R"]
    \end{tikzcd}
  \end{equation*}
  such that:
  \begin{itemize}
    \item \emph{Qullien adjunction.} $R$ preserves fibrations and
    acyclic fibrations.
    \item \emph{Weak equivalence.} The unit of the adjunction is a
    weak equivalence at each cofibrant object.
  \end{itemize}
\end{definition}

Following \cite{kapulkin-lumsdaine-18}, $\CxlCat_{\ITT}$,
$\CxlCat_{\ITT+\UIP}$, and $\CxlCat_{\ETT}$ admit left semi-model structure, which we now recall.
For the rest of this section, fix $F \colon \bC \to \bD \in \CxlCat_{\ITT}$.

\begin{definition}\label{def:ctx-wequiv}
  A contextual functor $F$ is a \emph{weak equivalence} when it has the \emph{weak
    type lifting} and \emph{weak term lifting} properties.
  \begin{itemize}
    \item \emph{Weak type lifting.}
    If $\Gamma \in \bC$ and $\Delta \in \bD$ is such that
    $\ft{\Delta} = F\Gamma$ then there is $\Gamma.A$ an extension of
    $\Gamma$ such that there is some
    $w \colon F(\Gamma.A) \simeq \Delta$ in $\bD$ over $F\Gamma$.
    \item \emph{Weak term lifting.}
    If $\Gamma \in \bC$ and $A \in \Ty\Gamma$ then for each section
    $a \colon F\Gamma \to F(\Gamma.A) \in \bD$ there exists
    $\overline{a} \colon \Gamma \to \Gamma.A \in \bC$ along with a
    homotopy $e \colon F\overline{a} \simeq a$.
  \end{itemize}
\end{definition}

\begin{definition}\label{def:ctx-fib}
  A contextual functor $F$ is a \emph{fibration} when it has the \emph{homotopy
    lifting properties} for terms and types:
  \begin{itemize}
    \item \emph{Homotopy lifting property for types.}
    Suppose $\Gamma \in \bC$ and $A \in \Ty\Gamma$.
    Fix $w \colon F(\Gamma.A) \simeq \Delta$ over $F\Gamma$ in $\bD$.
    Then there is $\overline{w} \colon \Gamma.A \simeq \Gamma.A'$ over
    $\Gamma$ in $\bC$ such that $F\overline{w} = w$ in $\bD$.
    \item \emph{Homotopy lifting property for terms.}
    Suppose $\Gamma \in \bC$ and $A \in \Ty\Gamma$.
    Fix sections $a \colon \Gamma \to \Gamma.A$ and
    $a' \colon F\Gamma \to F(\Gamma.A)$ along with a homotopy
    $e \colon Fa \simeq a'$.
    Then there is a section $\overline{a'} \colon \Gamma \to \Gamma.A$
    along with $\overline{e} \colon a \simeq \overline{a'}$ in $\bC$
    such that $F\overline{e} = e$ in $\bD$.
  \end{itemize}
\end{definition}

\begin{definition}\label{def:ctx-triv-fib}
  A contextual functor $F$ is a \emph{trivial fibration} when it has the
  \emph{strict type lifting} and \emph{strict term lifting}
  properties:
  \begin{itemize}
    \item \emph{Strict type lifting.} If $\Gamma \in \bC$ and
    $\Delta \in \bD$ is such that $\ft{\Delta} = F\Gamma$ then there
    is $\Gamma.A$ an extension of $\Gamma$ such that
    $F(\Gamma.A) = \Delta$.
    \item \emph{Strict term lifting.} If $\Gamma \in \bC$ and
    $A \in \Ty\Gamma$ then each section
    $a \colon F\Gamma \to F(\Gamma.A) \in \bD$ lifts to a section
    $\overline{a} \colon \Gamma \to \Gamma.A \in \bC$.
  \end{itemize}
\end{definition}

\begin{lemma}[{\cite[Proposition 3.14]{kapulkin-lumsdaine-18}}]
  A contextual functor $F$ is a trivial fibration if and only if it is both a weak
  equivalence and a fibration.
\end{lemma}

\begin{definition}\label{def:ctx-cofib}
  For each $n \in \bN$, we define, for $\bT$ one of $\ITT,\ITT+\UIP,\ETT$:
  \begin{itemize}
    \item $\vvbr{\Ctx_n}^\CxlCat_\bT \in \CxlCat_\bT$ to be the contextual
    category with the structure of $\bT$ freely generated by a generic context
    of length $n$.
    \item $\vvbr{\Ctx_n \vdash \Type}^\CxlCat_\bT \in \CxlCat_\bT$ to be the
    contextual category with the structure of $\bT$ freely generated by a
    generic type in a context of length $n$.
    \item $\vvbr{\Ctx_n \vdash \Term : \Type}^\CxlCat_\bT \in \CxlCat_\bT$ to be
    the contextual category with the structure of $\bT$ freely generated by a
    generic (well-typed) term in in a context of length $n$.
  \end{itemize}
  A contextual functor $F$ is a \emph{cofibration} in $\CxlCat_\bT$ when it is
  small with respect to the set $\McI^\CxlCat_\bT$ of maps consisting of
  inclusions
  $\vvbr{\Ctx_n}^\CxlCat_\bT \hookrightarrow \vvbr{\Ctx_n \vdash
    \Type}^\CxlCat_\bT$ and
  $\vvbr{\Ctx_n \vdash \Type}^\CxlCat_\bT \hookrightarrow \vvbr{\Ctx_n \vdash
    \Term : \Type}^\CxlCat_\bT$.
\end{definition}

With the above structure, $\CxlCat_{\bT}$ admits a \emph{left semi-model
  structure} for $\bT = \ITT,\ITT+\UIP,\ETT$.

\begin{theorem}[{\cite[Theorem 6.9]{kapulkin-lumsdaine-18}}]
  Each of: $\CxlCat_{\ITT}$, $\CxlCat_{\ITT+\UIP}$ and $\CxlCat_{\ETT}$ carries
  a left semi-model structure given by weak equivalences, fibrations, and
  cofibrations defined above.
  \qed
\end{theorem}

%%% Local Variables:
%%% mode: latex
%%% TeX-master: "./main.tex"
%%% TeX-engine: xetex
%%% End:

%% file: plan.tex
\section{Outline of Argument} \label{sec:outline}
Recall that there is a forgetful functor
$\abs{-} \colon \CxlCat_\ETT \to \CxlCat_{\ITT+\UIP}$ as described in
\Cref{subsec:uip-ett}.
Our main goal is to show that this forgetful functor constitutes the
right adjoint of a Morita equivalence.
Denote its left adjoint as
\begin{align*}
  \vbr{-}_\ETT^\circ \colon \CxlCat_{\ITT+\UIP}\to \CxlCat_\ETT
\end{align*}
There are two main requirements for this left adjoint:
\begin{itemize}
  \item Obviously, $\vbr{\bC}_\ETT^\circ$ must be a model of an extensional type
  theory for each cofibrant $\bC \in \CxlCat_{\ITT+\UIP}$.
  \item In view of the weak equivalence condition of \Cref{def:morita} and the
  definition of a weak equivalence given in \Cref{def:ctx-wequiv}, we require
  that the unit $\bC \to \abs{\vbr{\bC}_\ETT^\circ}$ at $\bC$ cofibrant has the weak
  type and term lifting properties.
\end{itemize}
We satisfy the above requirements by providing an explicit construction of its
left adjoint at $\McI^\CxlCat_{\ITT+\UIP}$-cellular objects.
Suppose we have such a cellular object $\bC \in \CxlCat_{\ITT+\UIP}$ and we
would like to construct a model of extensional type theory $\vbr{\bC}_\ETT$
using $\bC$.

Because the main difference between $\ITT$ and $\ETT$ is the equality reflection
rule, the intuition is that whenever two terms are propositionally equal in
$\bC$ then they should be formally identified in $\vbr{\bC}_\ETT$.
However, because in dependent type theory, types depend on terms, we need to
also identify the types that depend on the terms we identified.
For example, suppose terms $a_1~a_2:A$ are propositionally equal so that
with equality reflection, they are judgementally equal.
Now suppose we have a dependent type $(x:A) \vdash B(x)~\Type$.
By $\transport$ and the proof $H \colon a_1 \simeq a_2$, we have an equivalence
$\gamma(H) \colon B(a_1) \tosimeq B(a_2)$.
But under the presence of the equality reflection rule, $a_1 \simeq a_2$ implies
that $a_1 = a_2$ and so $B(a_1) = B(a_2)$.
Furthermore, by $\UIP$, if $a_1 = a_2$ then a proof that
$H \colon a_1 \simeq a_2$ is homotopic to $\refl$.
But then $\transport$ at $\refl$ is exactly the identity so the equivalence
$\gamma(H) \colon B(a_1) \simeq B(a_2)$ when $a_1 = a_2$ judgementally is homotopic
to the identity.
So $\gamma(H) \colon B(a_1) \tosimeq B(a_2)$ should be identified with the identity.
But then let us now suppose that there are terms $b_1 : B(a_1)$ and
$b_2 : B(a_2)$ for which we have a proof of propositional equality
$t(p) \cdot b_1 \simeq b_2 : B(a_2)$.
Once again, we see need to formally identify terms $t(p) \cdot b_1$ and $b_2$.

Because in a model $\bC$ of $\ITT$, terms and term-in-type substitution
correspond to maps and pullbacks respectively, we are led to the iterative
process of repeatedly identifying homotopic maps and then identifying pullbacks
along these homotopic maps.
In particular, we identify pullbacks along homotopic maps by setting the map
obtained by $\transport$ to identity.
In other words, we are iteratively specifying a class of maps $\McW_\ETT$ in
$\bC$, obtained in a fashion similar to that in \Cref{lem:glue}, to formally
collapse to identity maps to obtain $\vbr{\bC}_\ETT$ via a quotient
construction.
Under this quotient construction, the type- and term-lifting properties now
manifest themselves as the existence of a choice of representatives for each
equivalence class.

In performing this quotient construction to obtain $\vbr{\bC}_\ETT$, we need to
ensure it is indeed a contextual category supporting the requisite logical
structure.
Roughly, this means that if we have identified
$\Gamma_1 \tosimeq \Gamma_2 \in \McW_\ETT$ and
$\Gamma_1.A_1 \tosimeq \Gamma_2.A_2 \in \McW_\ETT$ and
$\Gamma_1.A_1.B_1 \tosimeq \Gamma_2.A_2.B_2 \in \McW_\ETT$ then
we also need to identify
$\Gamma_1.\Pi(A_1,B_1) \tosimeq \Gamma_2.\Pi(A_2,B_2) \in \McW_\ETT$.
Likewise, we need to identify
$\Gamma_1.\Sigma(A_1,B_1) \tosimeq \Gamma_2.\Sigma(A_2,B_2) \in \McW_\ETT$ and
$\Gamma_1.A_1.A_1.\Id_{A_1} \tosimeq \Gamma_2.A_2.A_2.\Id_{A_2} \in \McW_\ETT$.
Moreover, in order to ensure that result of this formal identification is a
category to begin with, we need to ensure that identity maps in $\vbr{\bC}_\ETT$
are unique.
As we shall see, the two goals above rely on the property that if
$w_1,w_2 \colon \Gamma \tosimeq \Delta \in \McW_\ETT$ then $w_1 \simeq w_2$.
This follows from $\UIP$ along with a classification of types property enjoyed
by cellular models.
This classification of types property is also why in \Cref{def:morita}, we
restrict our attention to the cellular (and thus cofibrant) models.
In addition, here, we see an interesting parallel: in this quotient
construction, $\UIP$ manifests itself as saying that the identity map is unique.

To summarise, we first show in \Cref{sec:cofib} that cellullar models are
generalisations of the syntactic category of $\ITT+\UIP$ by allowing for
freely-added base types and terms, but not definitional equalities between
those.
Then, throughout \Cref{sec:w-ett,sec:c-ett,sec:c-ett-logic}, we fix a cellular
object $\bC \in \CxlCat_{\ITT+\UIP}$ and proceed as follows:
\begin{itemize}
  \item \Cref{sec:w-ett} describes a wide subcategory $\McW_\ETT$ in $\bC$ of
  the weak equivalences $\McW$ called the \emph{extensional kernel} constructed
  in the aforementioned manner.
  \item \Cref{sec:c-ett} describes a contextual category
  $\vbr{\bC}_\ETT$ obtained from $\bC$ by formally collapsing all maps
  in $\McW_\ETT$ to identities via a quotient construction.
  This category is called the \emph{free extensional type theory}
  generated by $\bC$ and as a result of this quotient construction,
  there is a quotient map $[-] \colon \bC \to \vbr{\bC}_\ETT$.
  \item \Cref{sec:c-ett-logic} verifies that $\vbr{\bC}_\ETT$
  inherits the requisite logical structure from $\bC$.
\end{itemize}
Finally, in \Cref{sec:c-ett-lift}, we check that the assignment of cellular
objects $\bC \mapsto \vbr{\bC}_\ETT$ defines a functor that agrees with the
functor
\begin{align*}
  \vbr{-}_\ETT^\circ \colon \CxlCat_{\ITT+\UIP} \to \CxlCat_\ETT
\end{align*}
that is left adjoint to the forgetful functor
$\abs{-} \colon \CxlCat_{\ETT} \to \CxlCat_{\ITT+\UIP}$ on cellular objects and
that this adjunction forms a Morita equivalence.

\subsection{Comparison with Hofmann's Proof}
Our approach constructing the class of maps $\McW_\ETT$ in $\bC$ and collapsing
them to identities is reminiscent of Hofmann's approach in
\cite{hofmann:extensional-constructs}.

In \cite{hofmann:extensional-constructs}, Hofmann works explicitly with the
syntactic category with attributes of $\ITT+\UIP$ extended with the natural
numbers as a base type.
His proof proceeded as follows:
\begin{itemize}
  \item In \S 3.2.5.2, Hofmann describes a class of partial homotopy
  equivalences (referred to in \cite{hofmann:extensional-constructs} as
  propositional isomorphisms)
  $\mathsf{co}_{\Gamma,\Delta} \colon \Gamma \tosimeq \Delta$ constructed by
  inspecting the outer-most type former of the type in $\Gamma$ and $\Delta$.
  The maps $\mathsf{co}$ corresponds to the class $\McW_\ETT$ in our setting,
  and the inspection of outer-most type former corresponds to
  \Cref{prop:cofib-ty-class} due to cofibrancy in our setting.
  \item Then, in \S 3.2.5.3, Hofmann describes a CwA $\bQ$ obtained from the
  syntactic CwA via a quotient construction by collapsing the maps
  $\mathrm{co}_{\Gamma,\Delta}$ to identities.
  This corresponds the category $\vbr{\bS}_\ETT$ in our setting, where $\bS$ is
  the syntactic contextual category of $\ITT+\UIP$ extended with a base type of
  natural numbers.
\end{itemize}

A major difference in our approach from that of Hofmann's is that the definition
of the class $\McW_\ETT$ itself does not rely on inspecting the outer-most type
former of domains and codomains of its maps.
Instead, we rely on cofibrancy to inspect the outer-most type former when we
prove that $w_1,w_2 \colon \Gamma \tosimeq \Delta \in \McW_\ETT$ implies
$w_1 \simeq w_2$.
Furthermore, Hofmann works with the syntactic model of $\ITT+\UIP$, which has no
base types or terms.
In contrast, our setting of cofibrant models is more general in that they allow
for freely-added types or terms.
As a result, the obvious property in the syntactic model that that types can be
distinguished using their outer-most type former requires justification, which is given in \Cref{sec:cofib}.
Note that this part of the proof does not have an analogue in Hofmann's proof as Hofmann does not consider extensions of the base theories.
In each of \Cref{sec:w-ett,sec:c-ett,sec:c-ett-logic,sec:c-ett-lift}, which
constitute the remaining portion of the proof, we have provided an overview of
each section and cross-referenced various lemmas and propositions in the section
in question with the corresponding results in
\cite{hofmann:extensional-constructs} for easy comparison.

%%% Local Variables:
%%% mode: latex
%%% TeX-master: "./main.tex"
%%% TeX-engine: xetex
%%% End:

%% file: cxlcat-cells.tex
\section{Classification of Types in Cellular Models}\label{sec:cofib}
A key step in Hoffman's original proof in the syntactic setting was the ability
to define maps between contexts based on the outer-most type former of the last
type in context.
Although the definition of Morita equivalence in \Cref{def:morita} as a Quillen adjunction between their
categories of models whose units weak equivalences when evaluated at cofibrant models
is more general than Hofmann's syntactic setting, cofibrancy
nevertheless allows one to recover the certain syntax-like properties.
Specifically, as cofibrant models are obtained by freely adjoining types and terms,
but no judgemental equalities, to the initial model, they have a syntax-like
property that allows one to classify types into the mutually exclusive cases of
$\Pi$-types, $\Id$-types, $\Sigma$-types, or base types.
This syntax-like property is the key that allows Hoffman's proof to be
reproduced in our more general setting.

Intuitively, this result holds because, by \Cref{def:ctx-cofib}, the cofibrant models are exactly those obtained by repeatedly freely adjoining base types and
terms, but no definitional equalities, starting from the initial model, which consists of just the empty context.
In each step when a base type is freely adjoined, $\Pi$-, $\Id$- and
$\Sigma$-types that can be obtained using combinations of this new base type are
also formally added.
In what follows, we formalise this intuition by giving an explicit construction
of a model of $\ITT+\UIP$ obtained by freely adjoining a type or term to a given
model of $\ITT+\UIP$
This construction is most naturally viewed not directly as construction on
contextual categories, but rather constructions happening in the slightly more
general world of \emph{categories with attributes} (CwAs).
Therefore, we first recall some basic definitions regarding to CwAs and their
relation to contextual categories in \Cref{subsec:cwa} and then in  \Cref{subsec:cofib-constr}, we proceed to
prove the classification of types for cofibrant models in \Cref{prop:cofib-ty-class}.

\subsection{Categories with Attributes}\label{subsec:cwa}
\begin{definition}[Category with Attributes]\label{def:cwa}
  A category with attribute $\bC$ consists of a category $\bC$ along with the following data:
  \begin{itemize}
    \item A chosen terminal object $1 \in \ob\bC$.
    \item A functor $\Ty_\bC\relax \colon \bC^\op \to \Set$.
    \item An assignment to each $A \in \Ty_\bC\Gamma$ for $\Gamma \in \bC$, an
    object $\Gamma.A \in \bC$ and a map
    $\pi^\bC_A \colon \Gamma.A \twoheadrightarrow \Gamma$.
    \item For each $A \in \Ty_\bC\Gamma$ and $f \colon \Delta \to \Gamma$, a map
    $f.A \colon \Delta.f^*A \to \Gamma.A$, where $f^*A \coloneqq (\Ty_\bC f)A$.
  \end{itemize}
  subject to the condition that for each $A \in \Ty_\bC\Gamma$ and
  $f \colon \Delta \to \Gamma$ as above, the following square
  \begin{equation*}
    \begin{tikzcd}
      \Delta.f^*A & \Gamma.A \\
      \Delta & \Gamma
      \ar["f.A", from=1-1, to=1-2]
      \ar["\pi_{f^*A}"', two heads, from=1-1, to=2-1]
      \ar["\pi_A", two heads, from=1-2, to=2-2]
      \ar["f"', from=2-1, to=2-2]
      \ar["\lrcorner"{pos=0}, phantom, from=1-1, to=2-2]
    \end{tikzcd}
  \end{equation*}
  is a pullback.
  Further, these operations are strictly functorial, i.e.,
  \begin{itemize}
    \item ${\id\relax}.A = \id_{\Gamma.A}$
    \item For $g \colon \Theta \to \Delta$ and $f \colon \Delta \to \Gamma$, one has
    \begin{equation*}
      (fg).A = g.f^*A \cdot f.A
    \end{equation*}
  \end{itemize}
\end{definition}
Just like in the case of contextual categories, we denote dependent projections
by the two-headed arrow, e.g. $\pi_A \colon \Gamma.A \twoheadrightarrow \Gamma$
for $\Gamma \in \bC$ and $A \in \Ty_\bC\Gamma$.
Moreover, we will often omit the subscript, writing
$\pi \colon \Gamma.A \twoheadrightarrow \Gamma$ instead of $\pi_A$.
Likewise, we often write $\Gamma.A.A$ to mean $\Gamma.A.\pi^*A$.
Furthermore, the set $\Tm_\bC^\Gamma{A}$ consists of sections of
$\Gamma \to \Gamma.A$ of the projection $\Gamma.A \twoheadrightarrow \Gamma$,
and is referred to as the \emph{terms of $A$}.

Denote by $\CwA$ the category of categories with attributes and maps thereof
preserving the type-theoretic structure (i.e. the empty context, the context
projections and the substitution pullbacks).
Again, much like in the case of contextual categories, categories with
attributes can be equipped with additional structure corresponding to the
logical structure such as $\Id$-, $\Pi$- (or $\Pi_\Ext$-), $\Sigma$ and
$\UIP$-structure present in type theory.
The category consisting of categories with attributes equipped with $\Id$-,
$\Pi_\Ext$, $\Sigma$-types as objects and maps $F \colon \bC \to \bD$ between
such $\CwA$s given by those maps that preserve all type-theoretic and logical
structures is denoted by $\CwA_\ITT$.
Accordingly, $\CwA_{\ITT+\UIP}$ has as objects $\bC \in \CwA_\ITT$ along with an
$\UIP$ structure on $\bC$ and maps $\bC \to \bD$ are those maps in $\CwA_\ITT$
that additionally preserve the $\UIP$ structure.

It is clear that any contextual category $\bC$ can be regarded as a $\CwA$ by
taking the presheaf of types as the sections of the father map.
Hence, there is an evident full and faithful functor
$\CxlCat \hookrightarrow \CwA$ (respectively,
$\CxlCat_\ITT \hookrightarrow \CwA_\ITT$ and
$\CxlCat_{\ITT+\UIP} \hookrightarrow \CwA_{\ITT+\UIP}$) exhibiting $\CxlCat$
(respectively, $\CxlCat_\ITT$ and $\CxlCat_{\ITT+\UIP}$) as the full subcategory
consisting of $\CwA$s equipped with a suitable grading on objects, since such a
grading is unique if it exists, and is automatically preserved by any $\CwA$
map.
Conversely, the only difference between a $\CwA$ and a contextual category is
that objects in a contextual category must all be successively built-up by
repeated context comprehension.

\begin{definition}[Fibrant Slice]\label{def:fib-slice}
  Let $\bC \in \CwA$ and $\Gamma \in \bC$.
  A \emph{context over $\Gamma$} is a sequence $\vec{A} = (A_1, \ldots, A_n)$
  where $A_1 \in \Ty_\bC\Gamma$ and each $A_i \in \Ty_\bC\Gamma.A_0...A_{i-1}$.
  Any context $\vec{A}$ over $\Gamma$ induces an evident context extension
  $\Gamma.\vec{A} \coloneqq \Gamma.A_1...A_n$, with projection map
  $\pi_{\vec{A}} \coloneqq \Gamma.A_1...A_{n-1}.A_n \twoheadrightarrow
  \Gamma.A_1...A_{n-1} \twoheadrightarrow \ldots \twoheadrightarrow \Gamma$.

  The \emph{fibrant slice over $\Gamma$} is the contextual category
  $\bC\sslash\Gamma$ defined as follows:
  \begin{itemize}
    \item Objects of length $n$ are contexts $\vec{A} = (A_1, \ldots ,A_n)$ of
    length $n$ over $\Gamma$.
    \item
    $\sfrac{\bC}[\sslash]{\Gamma}(\vec{A}, \vec{B}) \coloneqq
    \sfrac{\bC}{\Gamma}(\Gamma.\vec{A}, \Gamma.\vec{B})$, and the categorical
    structure is inherited from the slice $\bC/\Gamma$.
    \item Reindexing and the connecting maps are inherited directly from $\bC$
    using $\Ty_\bC\relax$.
  \end{itemize}

  Moreover, any logical structure on $\bC$ induces one on $\bC \sslash \Gamma$.
  A map $f \colon \Gamma' \to \Gamma$ induces a contextual functor
  $f^* \colon \bC \sslash \Gamma \to \bC \sslash \Gamma'$, strictly functorially
  in $f$, and preserving all logical structure under consideration.
  This gives rise to a presheaf $\vec{\Ty\relax}_\bC \colon \bC^\op \to \Set$, where
  $\vec{\Ty\relax}_\bC$ consists of $\vec{A} \in \bC\sslash\Gamma$ and the
  functorial action on $f \colon \Gamma' \to \Gamma$ is given by the action of
  $f^* \colon \bC \sslash \Gamma \to \bC \sslash \Gamma'$ on objects.

  Similarly, for any $\CwA$ map $F \colon \bC \to \bD$ and $\Gamma \in \bC$,
  there is an induced slice functor
  $F \sslash \Gamma \colon \bC \sslash \Gamma \to \bD \sslash F\Gamma$,
  preserving any logical structure that F does.

  In particular, $\bC \sslash 1$, the fibrant slice over $1 \in \bC$, is the
  \emph{contextual core of $\bC$}.
  For each $\bC \in \CwA$ (respectively, $\CwA_\ITT$ and $\CwA_{\ITT+\UIP}$),
  denote by $\core\bC$ (respectively, $\core_\ITT\bC$ and $\core_{\ITT+\UIP}\bC$)
  the contextual core of $\bC$.
\end{definition}

\iffalse
%
\begin{proposition}[{\cite[Proposition 4.4]{kapulkin-lumsdaine-18}}]
  %
  A assignment on objects $\bC \mapsto \core\bC$ (respectively,
  $\bC \mapsto \core_\ITT\bC$ and $\bC \mapsto \core_{\ITT+\UIP}\bC$) extends to
  a functor $\core\relax \colon \CwA \to \CxlCat$ (respectively,
  $\core_\ITT\relax \colon \CwA_\ITT \to \CxlCat_\ITT$ and
  $\core_{\ITT+\UIP}\relax \colon \CwA_{\ITT+\UIP} \to \CxlCat_{\ITT+\UIP}$)
  that is right adjoint to the inclusion
  %
  \begin{equation*}
    \begin{tikzcd}[column sep=huge]
      \CxlCat_{\ITT+\UIP} & \CwA_{\ITT+\UIP} \\
      \CxlCat_\ITT & \CwA_\ITT \\
      \CxlCat & \CwA
      %
      \ar[from=1-1, to=2-1]
      \ar[from=1-2, to=2-2]
      \ar[from=2-1, to=3-1]
      \ar[from=2-2, to=3-2]
      \ar[phantom, "\bot"{description}, from=1-1, to=1-2]
      \ar[phantom, "\bot"{description}, from=2-1, to=2-2]
      \ar[phantom, "\bot"{description}, from=3-1, to=3-2]
      \ar[hook, from=1-1, to=1-2, bend left=10]
      \ar[hook, from=2-1, to=2-2, bend left=10]
      \ar[hook, from=3-1, to=3-2, bend left=10]
      \ar["\core_{\ITT+\UIP}\relax"{description}, from=1-2, to=1-1, bend left=10]
      \ar["\core_{\ITT}\relax"{description}, from=2-2, to=2-1, bend left=10]
      \ar["\core\relax"{description}, from=3-2, to=3-1, bend left=10]
    \end{tikzcd}
  \end{equation*}
  %
  \qed
  %
\end{proposition}
%
\fi

\begin{definition}\label{def:cwa-cofib}
  Define the following freely-generated models in $\CwA$, where $\bT$ is either $\ITT$ or $\ITT+\UIP$.
  \begin{itemize}
    \item Let $\vvbr{\Ctx}^\CwA$ (respectively, $\CwA_\bT$) be the $\CwA$
    (respectively, $\CwA$ equipped with $\bT$-structure) freely generated by a
    generic context $\Ctx$.
    \item Let $\vvbr{\Ctx \vdash \Type}^\CwA$ (respectively,
    $\vvbr{\Ctx \vdash \Type}^\CwA_{\bT}$) be the $\CwA$ (respectively, $\CwA$
    equipped $\bT$-structure) freely generated by a generic context $\Ctx$ and a
    generic type $\Type$ in this generic context $\Ctx$.
    \item Let $\vvbr{\Ctx \vdash \Term : \Type}^\CwA$ (respectively,
    $\vvbr{\Ctx \vdash \Term : \Type}^\CwA_\bT$) be the $\CwA$ (respectively,
    $\CwA$ equipped $\bT$-structure) freely generated by a generic context
    $\Ctx$, a generic type $\Type$ in this generic context $\Ctx$ and a generic
    term $\Term$ whose type is the generic type $\Type$ in this generic context
    $\Ctx$.
  \end{itemize}
  Put
  \begin{align*}
    \McI^\CwA &\coloneqq \set{\vvbr{\Ctx}^\CwA \hookrightarrow \vvbr{\Ctx \vdash \Type}^\CwA,
                \vvbr{\Ctx \vdash \Type}^\CwA \hookrightarrow \vvbr{\Ctx \vdash \Term : \Type}^\CwA} \\
    \McI^\CwA_{\bT} &\coloneqq \set{\vvbr{\Ctx}^\CwA_{\bT} \hookrightarrow
                            \vvbr{\Ctx \vdash \Type}^\CwA_{\bT},
                     \vvbr{\Ctx \vdash \Type}^\CwA_{\bT} \hookrightarrow
                            \vvbr{\Ctx \vdash \Term : \Type}^\CwA_{\bT}}
  \end{align*}
\end{definition}

\begin{proposition}[{\cite[Proposition 4.13.4]{kapulkin-lumsdaine-18}}]\label{prop:kl-4.13.4}
  The inclusion of $\CxlCat_{\ITT+\UIP}$ into $\CwA_{\ITT+\UIP}$ sends
  $\McI^\CxlCat_{\ITT+\UIP}$-cell complexes to $\McI^\CwA_{\ITT+\UIP}$-cell
  complexes.
  \def\endingmark{\qedsymbol}
\end{proposition}

\subsection{Classification of Types in Cellular Models}\label{subsec:cofib-constr}
We now shift our attention to proving the classification of types property for
$\McI^\CwA_\bT$-cell complexes, where $\bT$ is either $\ITT$ or $\ITT+\UIP$.

\begin{definition}\label{def:ncwa}
  The category of $\NCwA$ consists of objects $\CwA_\bT$-structures
  $(\bC, \Ty_\bC, \pi^\bC, 1_\bC,\Pi^\bC,\Sigma^\bC,\Id^\bC)$ and a subset
  $G \subseteq \coprod_{\Theta \in \bC}{\Ty_\bC\Theta}$ of types in $\bC$.
  Maps between $\NCwA$s $(\bC,G) \to (\bD,H)$ consists of maps of $\CwA$s
  $F \colon \bC \to \bD$ taking types in $G$ to types in $H$.
  Explicitly: $(F\Theta, F^{\Ty\relax}_\Theta T) \in H$ for each $(\Theta,T) \in G$.
\end{definition}

\begin{definition}\label{def:is-Pi-Sigma-Id}
  Given $\bC \in \CwA_\bT$ along with $\Gamma \in \bC$ and
  $A \in \Ty_\bC\Gamma$, we say:
  \begin{itemize}
    \item \emph{$A$ is a $\Pi$-type} when $A = \Pi(X,Y)$ for some
    $X \in \Ty_\bC \Gamma$ and $Y \in \Ty_\bC\Gamma.X$.
    \item \emph{$A$ is a $\Sigma$-type} when $A = \Sigma(X,Y)$ for some
    $X \in \Ty_\bC \Gamma$ and $Y \in \Ty_\bC\Gamma.X$.
    \item \emph{$A$ is an $\Id$-type} when $A = f^*\Id_X$ for some context
    $\Delta \in \bC$, type $X \in \Ty_\bC \Delta$ and map
    $f \colon \Gamma \to \Delta.X.X$.
  \end{itemize}
\end{definition}

\begin{construction}\label{constr:norm-cwa}
  Let $(\bC, G) \in \NCwA_\bT$.
  Define $\Norm(\bC,G) = (\bC, \Ty_{\Norm(G)}, \pi^{\Norm(G)}, 1_\bC) \in \CwA$,
  which we also denote by $\Norm(G)$ for brevity, with underlying category $\bC$
  and the same choice of the terminal context as in $\bC$.
  The types over each $\Gamma \in \bC$ in $\Norm(G)$ is the disjoint union of
  all the $\Pi,\Sigma,\Id$-types in $\Gamma$ along with the comma category
  $\Gamma \downarrow (G \hookrightarrow \coprod_{\Theta \in \bC}\Ty_\bC\Theta
  \to \bC)$, or more explicitly,
  \begin{align*}
    \Ty_{\Norm(G)}\Gamma \coloneqq \set{(\Gamma \xrightarrow{g} \Theta, T \in \Ty_\bC\Theta) ~|~ (\Theta,T) \in G}
    \sqcup \set{A \in \Ty_\bC\Gamma ~|~ A \text{ is a $\Pi,\Sigma,\Id$-type}}
  \end{align*}
  Given a map $f \colon \Gamma' \to \Gamma \in \bC$ the functorial action
  $f^* \colon \Ty_{\Norm(G)}\Gamma \to \Ty_{\Norm(G)}\Gamma'$ is given by
  \begin{equation*}
    % https://q.uiver.app/#q=WzAsNixbMCwwLCIoZyxUKSJdLFswLDEsIlxcVHlfe1xcTm9ybShHKX1cXEdhbW1hIl0sWzAsMiwiXFx0ZXh0eyRcXFBpLFxcU2lnbWEsXFxJZCQtdHlwZSB9IEEiXSxbMSwxLCJcXFR5X3tcXE5vcm0oRyl9XFxHYW1tYSciXSxbMSwyLCJcXHRleHR7JFxcUGksXFxTaWdtYSxcXElkJC10eXBlIH0gZl4qQSJdLFsxLDAsIihmZyxUKSJdLFswLDEsIiIsMCx7InN0eWxlIjp7InRhaWwiOnsibmFtZSI6Im1hcHMgdG8ifX19XSxbMiwxLCIiLDIseyJzdHlsZSI6eyJ0YWlsIjp7Im5hbWUiOiJtYXBzIHRvIn19fV0sWzQsMywiIiwyLHsic3R5bGUiOnsidGFpbCI6eyJuYW1lIjoibWFwcyB0byJ9fX1dLFsxLDNdLFs1LDMsIiIsMCx7InN0eWxlIjp7InRhaWwiOnsibmFtZSI6Im1hcHMgdG8ifX19XSxbMCw1LCIiLDEseyJzdHlsZSI6eyJ0YWlsIjp7Im5hbWUiOiJtYXBzIHRvIn19fV0sWzIsNCwiIiwxLHsic3R5bGUiOnsidGFpbCI6eyJuYW1lIjoibWFwcyB0byJ9fX1dXQ==
    \begin{tikzcd}
      {(g,T)} & {(f^*g,T)} \\
      {\Ty_{\Norm(G)}\Gamma} & {\Ty_{\Norm(G)}\Gamma'} \\
      {\text{$\Pi,\Sigma,\Id$-type } A} & {\text{$\Pi,\Sigma,\Id$-type } f^*A}
      \arrow[maps to, from=1-1, to=1-2]
      \arrow[maps to, from=1-1, to=2-1]
      \arrow[maps to, from=1-2, to=2-2]
      \arrow[from=2-1, to=2-2]
      \arrow[maps to, from=3-1, to=2-1]
      \arrow[maps to, from=3-1, to=3-2]
      \arrow[maps to, from=3-2, to=2-2]
    \end{tikzcd}
  \end{equation*}

  Context comprehension for $(g,T) \in \Ty_{\Norm(G)}\Gamma$ is defined as
  $\pi^{\Norm(G)}_{(g,T)} \coloneqq \pi^{\bC}_{g^*T}$ and
  $\pi^{\Norm(G)}_A \coloneqq \pi^\bC_A$ for $A \in \Ty_\bC\Gamma$ some
  $\Pi,\Sigma,\Id$-type.
  For a map $f \colon \Gamma' \to \Gamma \in \bC$, the connecting map is defined
  as $f.(g,T) \coloneqq f.g^*T$ using the $\CwA$-structure of $\bC$ and
  $f.A \coloneqq f.A$ is the connecting map in $\bC$ for
  $A \in \Ty_{\Norm(G)}\Gamma$ some $\Pi,\Sigma,\Id$-type.
  \begin{equation*}
    \begin{tikzcd}[cramped]
      \Gamma'.(\Gamma' \xrightarrow{f} \Gamma \xrightarrow{g} \Theta, T \in \Ty_\bC\Theta)
      \ar[d, two heads, "{\pi^{\Norm(G)}_{f^*(g,T)}}"']
      \ar[r, "{f.(g,T)}"]
      \ar[rd, "\lrcorner"{pos=0}, phantom]
      &
      \Gamma.(\Gamma \xrightarrow{g} \Theta, T \in \Ty_\bC\Theta)
      \ar[d, two heads, "{\pi^{\Norm(G)}_{(g,T)}}"]
      \\
      \Gamma'
      \ar[r, "{f}"']
      &
      \Gamma
    \end{tikzcd}
    \coloneqq
    \begin{tikzcd}[cramped]
      \Gamma'.f^*g^*T
      \ar[d, two heads, "{\pi^\bC_{f^*g^*T}}"']
      \ar[r, "{f.g^*T}"]
      \ar[rd, "\lrcorner"{pos=0}, phantom]
      &
      \Gamma.g^*T
      \ar[d, two heads, "{\pi^\bC_{g^*T}}"]
      \\
      \Gamma'
      \ar[r, "{f}"']
      &
      \Gamma
    \end{tikzcd}
    \in \bC
  \end{equation*}

  % Given $X_1 \in \Ty_{\Norm(G)}\Gamma$ and
  % $ \in \Ty_{\Norm(G)}\Gamma.g_1^*T_1$, the $\Pi,\Sigma,\Id$-types
  % are defined as
  % %
  % \begin{align*}
  %   \Pi^{\Norm(G)}((f_1,T_1),(f_2,T_2)) \coloneqq \Pi^\bC(f_1^*T_1, f_2^*T_2)
  %   && \Pi^{\Norm(G)}((f_1,T_1),(f_2,T_2)) \coloneqq \Sigma^\bC(f_1^*T_1, f_2^*T_2)
  % \end{align*}
  % \begin{align*}
  %   \Id^{\Norm(G)}_{(f_1,T_1)} \coloneqq \Id^{\bC}_{\Id(f_1^*T_1)}
  % \end{align*}
  %
  Define
  $\ev(G) = (\bC \xrightarrow{=} \bC, \Ty_{\Norm(G)} \xrightarrow{\ev(G)^{\Ty\relax}}
  \Ty_\bC) \colon \Norm(G) \to \bC \in \CwA$ that is the identity map on
  underlying categories and the natural transformation $\ev(G)^{\Ty\relax}$ at
  component $\Gamma \in \bC$ is defined as
  \begin{align*}
    \ev(G)^{\Ty\relax}_\Gamma(\Gamma \xrightarrow{g} \Theta, T \in \Ty_{\bC}\Theta) \coloneqq g^*T \in \Ty_\bC\Gamma &&
    \ev(G)^{\Ty\relax}_\Gamma(\text{$\Pi,\Sigma,\Id$-type } A \in \Ty_\bC\Gamma) \coloneqq A \in \Ty_\bC\Gamma
  \end{align*}
  Naturality of $\ev(G)^{\Ty\relax}$ follows from naturality of precomposition and
  naturality of $\Ty_\bC$.
  And by definition of context comprehension of $\Norm(G)$, one has that
  $\pi_{(g,T)}^{\Norm(G)} = \pi^\bC_{g^*T} = \pi_{\ev(K)^{\Ty\relax}(g,T)}^\bC$, so
  $\ev(K)$ is indeed a map of $\CwA$s.

  Put $\Pi,\Sigma,\Id$-type structures on $\Norm(G)$ by taking, for each
  $X \in \Ty_{\Norm(G)}\Gamma$ and
  $Y \in \Ty_{\Norm(G)}(\Gamma.X) = \Ty_{\Norm(G)}(\Gamma.\ev(G)^{\Ty\relax} X)$,
  \begin{align*}
    \Pi^{\Norm(G)}(X,Y) \coloneqq \Pi^\bC(\ev(G)^{\Ty\relax} X,\ev(G)^{\Ty\relax} Y)
    && \Sigma^{\Norm(G)}(X,Y) \coloneqq \Sigma^\bC(\ev(G)^{\Ty\relax} X,\ev(G)^{\Ty\relax} Y)
  \end{align*}
  \begin{align*}
    \Id^{\Norm(G)}_X \coloneqq \Id^\bC_{\ev(G)^{\Ty\relax} X}
  \end{align*}
  Because logical structure in $\bC$ is stable and $\ev(G)$ is functorial, this
  definition of $\Pi,\Sigma,\Id$-types in $\Norm(G)$ is stable as well.
  The necessary maps for $\Pi,\Sigma,\Id$-types are inherited from $\bC$, as
  $\Norm(G)$ has the underlying category the same as $\bC$.
  This puts a $\CwA_\bT$ structure on $\Norm(G)$ and by construction,
  $\ev(G) \colon \Norm(G) \to \bC$ is a map of $\CwA_\bT$-structures.

  Furthermore, the construction $\Norm \colon \NCwA_\bT \to \CwA_\bT$ is
  functorial.
  Given $F \colon (\bC,G) \to (\bD,H) \in \NCwA_\bT$, the map
  \begin{align*}
    \Norm(F) = (F, \Norm(F)^{\Ty\relax}) \colon \Norm(\bC,G) \to \Norm(\bD,H)
  \end{align*}
  has action on underlying categories given by $F$.
  The action on types at each $\Gamma \in \bC$ is
  $\Norm(F)^{\Ty\relax}_\Gamma \colon \Ty_{\Norm(G)}\Gamma \to \Ty_{\Norm(H)}F\Gamma$
  defined as
  \begin{align*}
    \Norm(F)^{\Ty\relax}_\Gamma(\Gamma \xrightarrow{g} \Theta, T) \coloneqq
    (F\Gamma \xrightarrow{Fg} F\Theta, F^{\Ty\relax}_\Theta T)
    &&
       \Norm(F)^{\Ty\relax}_\Gamma(\text{$\Pi,\Sigma,\Id$-type } A) \coloneqq
       F^{\Ty\relax}_\Gamma A
  \end{align*}
  Because $(F\Theta,F_\Theta^{\Ty\relax} T) \in H$ for each $(\Theta,T)$ and $F$
  preserves $\Pi,\Sigma,\Id$-types, it is indeed that
  $\Norm(F)^{\Ty\relax}_\Gamma(g,T) \in \Norm(H)_{F\Gamma}$ and $\Norm(F)_\Gamma^{\Ty\relax} A$
  remains a $\Pi,\Sigma,\Id$-type when $A$ is a $\Pi,\Sigma,\Id$-type.
  Naturality of $\Norm(G)^{\Ty\relax}$ follows from naturality of precomposition and
  functoriality of $F$.
  And $\Norm(F)$ respects context comprehension because
  $F\pi_{(g,T)}^{\Norm(G)} = F\pi_{g^*T}^\bC
  = \pi_{F^{\Ty\relax}_\Gamma g^*T}^\bD
  = \pi_{(Fg)^*(F^{\Ty\relax}_\Theta T)}^\bD
  = \pi_{(Fg,F^{\Ty\relax}_\Theta T)}^{\Norm(H)}
  $, so $\Norm(F)$ is indeed a map of $\CwA$s.
  Finally, $\Norm(F)$ preserves the logical structures $\Sigma,\Pi,\Id$ because
  $F$ preserves these logical structures.
  Functoriality of $\Norm \colon \NCwA_\bT \to \CwA_\bT$ follows by functoriality of
  application.
\end{construction}

\begin{lemma}\label{lem:ev-nat}
  Let $F \colon (\bC,G) \to (\bD,H) \colon \NCwA_\bT$.
  Then,
  \begin{equation*}
    \begin{tikzcd}[cramped, column sep=large]
      {\Norm(G)} & {\Norm(H)} \\
      \bC & \bD
      \arrow["{\Norm(F)}", from=1-1, to=1-2]
      \arrow["{\ev(G)}"', from=1-1, to=2-1]
      \arrow["{\ev(H)}", from=1-2, to=2-2]
      \arrow["{F}"', from=2-1, to=2-2]
    \end{tikzcd}
  \end{equation*}
\end{lemma}
\begin{proof}
  Because the action on underlying categories of $\ev(G),\ev(H)$ are both
  identity while the underlying action on underlying categories of $\Norm(F)$
  is $F$, the diagram on the right commutes on the level of underlying
  categories.
  Now, fix $\Gamma \in \bC$ along with $(\Gamma \xrightarrow{g} \Theta, T)$ with
  $(\Theta,T) \in G$.
  Then, following the top-right path gives
  $(g,T) \mapsto (Fg, F_\Theta^{\Ty\relax} T) \mapsto (Fg)^*(F_\Theta^{\Ty\relax} T)$ while the
  bottom-left path gives $(g,T) \mapsto g^*T \mapsto F^{\Ty\relax}_\Gamma(g^*T)$ and
  $(Fg)^*(F_\Theta^{\Ty\relax} T) = F^{\Ty\relax}_\Gamma(g^*T)$ by naturality of $F^{\Ty\relax}$.
  Finally, if $A \in \Ty_\bC\Gamma \hookrightarrow \Ty_{\Norm(G)}\Gamma$ is some
  $\Pi,\Sigma,\Id$-type then the top-right and bottom-left path both give
  $F^{\Ty\relax} A$.
\end{proof}

\begin{lemma}\label{lem:ev-tm-lift}
  For $(\bC,G) \in \NCwA_\bT$, the map
  $\ev(G) \colon \Norm(\bC,G) \to \bC \in \CwA_\bT$ has the strict term lifting
  property.
  \begin{equation*}
    \begin{tikzcd}
      \vvbr{\Ctx \vdash \Type}
      \ar[d, hook]
      \ar[r]
      &
      \Norm(G)
      \ar[d, "{\ev(G)}"]
      \\
      \vvbr{\Ctx \vdash \Term : \Type}
      \ar[r]
      \ar[ur, dashed]
      &
      \bC
    \end{tikzcd}
  \end{equation*}
\end{lemma}
\begin{proof}
  A map $\vvbr{\Ctx \vdash \Type} \to \Norm(G)$ is a choice of context
  $\Gamma \in \bC$ along with either:
  \begin{itemize}
    \item Some type $A$ in $\Gamma$ that is some $\Pi,\Sigma,\Id$-type; or
    \item Some $\Theta \in \bC$ and $T \in \Ty_\bC\Theta$ and
    $g \colon \Gamma \to \Theta \in \bC$ where $(\Theta,T) \in G$.
  \end{itemize}

  In the first case, $\ev(G)$ sends $A$ to the type $A$ itself, and the bottom
  map $\vvbr{\Ctx \vdash \Term : \Type} \to \bC$ selects a term of type $A$ in
  $\bC$, and because maps in $\Norm(G)$ are exactly maps in $\bC$, one can
  simply use the term selected by the bottom map as the lift.

  In the second case, composing with $\ev(K)$ selects the type
  $g^*T \in \Ty_\bC\Gamma$.
  Suppose one has a term of said type, which is a section
  $t \colon \Gamma \to \Gamma.g^*T$ of
  $\pi_{g^*T}^\bC \colon \Gamma.f^*T$.
  Then, by definition, $t$ is a also a section of
  $\pi_{(g,T)}^{\Norm(G)} = \pi_{g^*T}^\bC$, so we may take $t$ to be the lift
  because the action on underlying categories of $\ev(G)$ is the identity.
\end{proof}

% \begin{lemma}\label{lem:ev-triv-fib}
%   %
%   The map
%   $\ev^\bC \colon (\bC,\Ty_{\Norm\bC}, \pi^{\Norm\bC}) \to
%   (\bC,\Ty_\bC,\pi^\bC)$ from \Cref{constr:norm-cwa} is a trivial fibration.
%   %
% \end{lemma}
% %
% \begin{proof}
%   %
%   One must verify the strict type and term lifting properties.

%   We first check strict type lifting.
%   %
%   Fix $\Gamma \in \bC$ and $T \in \Ty_\bC\Gamma$.
%   %
%   Then,
%   $(\Gamma \xrightarrow{=} \Gamma, T \in \Ty_\bC\Gamma) \in
%   \Ty_{\Norm\bC}\Gamma$ and by definition,
%   $\Ty_{\ev^\bC}(\Gamma \xrightarrow{=} \Gamma, T \in \Ty_\bC\Gamma) =
%   \id\relax^*T = T$, so types strict lift along $\ev^\bC$.

%   Next, we check strict term lifting.
%   %
%   Fix $\Gamma \in \bC$ and a type
%   $(\Gamma \xrightarrow{g} \Delta, T \in \Ty_\bC\Delta) \in
%   \Ty_{\Norm\bC}\Gamma$.
%   %
%   Let $t$ be a term of type $Ty_{\ev^\bC}(g,T) = g^*T$, so that
%   $t \colon \Gamma \to \Gamma.g^*T$ is a section of
%   $\pi^\bC_{g^*T} \colon \Gamma.g^*T \twoheadrightarrow \Gamma$.
%   %
%   But then by definition, $\pi_{(g,T)}^{\Norm\bC} = \pi_{g^*T}\bC$, so $t$ is
%   also a term of type $(g,T)$, and $\ev^\bC(t) = t$ because $\ev^\bC$ is
%   identity on the underlying category.
%   %
% \end{proof}

\begin{definition}\label{def:base-decompose}
  Fix a map $F \colon \bD \to \Norm(\bC,G) \in \CwA_\bT$ for
  $(\bC,G) \in \NCwA_\bT$ and $\bD \in \NCwA$.
  Let $\Delta \in \bD$ and $X \in \Ty_\bD\Delta$.
  If
  $F^{\Ty\relax}_\Delta X \in \coprod_{(\Theta,T) \in G} \bC(F\Delta,\Theta)
  \hookrightarrow \Ty_{\Norm(G)}F\Delta$ is not a $\Pi$-, $\Sigma$-, or
  $\Id$-type then we write $F_\Delta^\Ctx,F_\Delta^\Nf,F_\Delta^\Sub$ for the
  various projections
  \begin{align*}
    F^{\Ty\relax}_\Delta X = (F\Delta \xrightarrow{F_\Delta^\Sub X} F_\Delta^\Ctx X, F_\Delta^\Nf X)
  \end{align*}
\end{definition}

\begin{lemma}\label{lem:into-norm}
  Fix a map $F \colon \bD \to \Norm(\bC,G) \in \CwA_\bT$ for
  $(\bC,G) \in \NCwA_\bT$ and $\bD \in \CwA$ along with some $\Delta \in \bD$
  and $X \in \Ty_\bD\Delta$.
  Suppose that $F^{\Ty\relax}_\Delta X$ is not a $\Pi,\Sigma,\Id$-type.
  Then, for each $f \colon \Delta' \to \Delta$, neither is $F^{\Ty\relax}_\Delta(f^*X)$.
  Moreover,
  \begin{center}
    \begin{minipage}{0.3\linewidth}
      \begin{equation*}
        F^\Ctx_\Delta X = F^\Ctx_{\Delta'}(f^*X)
      \end{equation*}
    \end{minipage}
    \begin{minipage}{0.3\linewidth}
      \begin{equation*}
        \begin{tikzcd}
          F\Delta' & F\Delta & F^\Ctx_\Delta X
          \ar[from=1-1, to=1-2, "{Ff}"]
          \ar[from=1-2, to=1-3, "{F_\Delta^\Sub X}"]
          \ar[from=1-1, to=1-3, bend right, "{F_\Delta^\Sub(f^*X)}"']
        \end{tikzcd}
      \end{equation*}
    \end{minipage}
    \begin{minipage}{0.3\linewidth}
      \begin{equation*}
        F_\Delta^\Nf X = F_{\Delta'}^\Nf(f^*X)
      \end{equation*}
    \end{minipage}
  \end{center}
\end{lemma}
\begin{proof}
  Write $F_\Delta^{\Ty\relax} X \in \Ty_{\Norm(G)}F\Delta$ as
  \begin{align*}
    F^{\Ty\relax}_\Delta X = (F\Delta \xrightarrow{F_\Delta^\Sub X} F_\Delta^\Ctx X, F_\Delta^\Nf X)
  \end{align*}
  Then,
  Given such $F$ and family of triples
  $((F_\Delta^\Ctx,F_\Delta^\Nf,F_\Delta^\Sub))_\Delta$ as above, one may define
  a map $\bD \to \Norm(G) \in \CwA$ whose action on underlying categories is
  $F \colon \bD \to \bC$.
  The action on types is given as
  \begin{align*}
    \Ty_\bD\Delta \ni A \mapsto (F\Delta \xrightarrow{F_\Delta^\Sub A} F_\Delta^\Ctx A, F_\Delta^\Nf A)
    \in \Ty_{\Norm(G)}(F\Delta)
  \end{align*}
  so that $F^{\Ty\relax}(f^*X) = (Ff)^*F^{\Ty\relax} X = ((Ff)^*F^\Sub X, F^\Nf X)$.
  This means in particular that $F^{\Ty\relax}(f^*X)$ is not a $\Pi,\Sigma,\Id$-type.
  By the definition of the functorial action of $\Ty_{\Norm(G)}$, one has
  $(Ff)^*F^\Sub X = F^\Sub(f^*X)$ and $F^\Nf(f^*X) = F^\Nf X$.
\end{proof}

\begin{lemma}\label{lem:sec-ev-char}
  Suppose a map $S \colon \bC \to \Norm(G)$ for $(\bC,G) \in \NCwA_\bT$ is a
  section of $\ev(G)$.
  Then, its action on underlying categories is the identity and for each
  $(\Gamma \in \bC, X \in \Ty_\bC\Gamma) \in G$ where $X$ is not a
  $\Pi,\Sigma,\Id$-type, neither is $S^{\Ty\relax}_\Gamma A$ and one has
  \begin{align*}
    (S^\Sub X)^*(S^\Nf X) = X
  \end{align*}
\end{lemma}
\begin{proof}
  Because the action on underlying categories of $\ev(G)$ is the identity, the
  action on underlying categories of $S$ is forced to be the identity.
  By definition, $S$ being a map of $\CwA_\bT$ means that it must preserve
  $\Pi,\Sigma,\Id$-types.
  Because $S$ is a section of $\ev(G)$ and the action of $\ev(G)$ on
  $\Pi,\Sigma,\Id$-types is the identity, if $X$ is not a $\Pi,\Sigma,\Id$-type then it
  cannot be mapped to a $\Pi,\Sigma,\Id$-type.
  Hence, one may write $S^{\Ty\relax} X$ as $S^{\Ty\relax} X = (S^\Sub X, S^\Nf X)$.
  And again because $S$ is a section, $\ev(K)^{\Ty\relax} S^{\Ty\relax} X = X$ for each
  $(\Gamma \in \bC, X \in \Ty_\bC\Gamma)$, one has
  $\ev(K)^{\Ty\relax} S^{\Ty\relax} X = \ev(K) (S^\Sub X, S^\Nf X) = (S^\Sub X)^*(S^\Nf X)$, so
  $(S^\Sub X)^*(S^\Nf X) = X$, as claimed.
\end{proof}

\begin{definition}\label{def:type-normalise}
  Given $(\bC,G) \in \NCwA$ and section $S \colon \bC \to \Norm(G)$ of
  $\ev(G) \colon \Norm(G) \to \bC$, a pair
  $(\Gamma \in \bC, X \in \Ty_\bC\Gamma)$ is an \emph{$S$-normal form} when
  $S^{\Ty\relax} X$ is not a $\Pi,\Sigma,\Id$-type and $S^\Sub X = \id_\Gamma$.
  In particular, $S^\Nf X = X$.
  \begin{align*}
    (\Gamma \xrightarrow{S^\Sub X} S^\Ctx X, S^\Nf X)
    =
    (\Gamma \xrightarrow{=} \Gamma, X)
  \end{align*}
  This section $S$ is a \emph{type-normalisation structure relative to
    $(\bC,G)$} when each $(\Theta,T) \in G$ is an $S$-normal form.
  A \emph{type-normalisation structure} consists of $(\bC,G) \in \NCwA$ along
  with a type-normalisation structure $S$ relative to $(\bC,G)$.
\end{definition}

\begin{definition}\label{def:types-sep-faith}
  Let $\bC$ be a $\CwA$.
  Two pairs of context-type pairs
  $(\Theta_i,T_i) \in \coprod_{\Gamma \in \bC}\Ty_\bC\Gamma$ for $i=1,2$ are
  \emph{separated} if for all $\Gamma \in \bC$ and
  $f_i \colon \Gamma \to \Theta_i$ one has $f_1^*T_1 \neq f_2^*T_2$.
  A context-type pair $(\Theta,T) \in \coprod_{\Gamma \in \bC}\Ty_\bC\Gamma$ is
  \emph{faithful} if $f_1^*T = f_2^*T$ implies $f_1 = f_2$ for all
  $f_1, f_2 \colon \Gamma \rightrightarrows \Theta$.

  A set $G \subseteq \coprod_{\Gamma \in \bC} \Ty_\bC\Gamma$ is \emph{faithfully
    separated} when each of its element is faithful and each of its pairwise
  distinct elements are separated.
\end{definition}
In other words, two context-type pairs are separated when they are in different
connected components of the category of elements of the presheaf of types, and a
faithful type is such that its substitution along different maps can never be
the same type.

In particular, maps between $\CwA$s are maps between the categories of elements
of the presheaf of types, which must preserve connected components.
As a result, all maps between $\CwA$s reflect separation, in the following
sense.
\begin{lemma}\label{lem:sep-reflect}
  Let $\bC$ be a $\CwA$ with context-type pairs
  $(\Theta_i \in \bC, T_i \in \Ty_\bC{\Theta_i})$ for $i=1,2$.
  Suppose one has a map $\bC \to \bD \in \CwA$ where $(F\Theta_1, F^{\Ty\relax} T_1)$
  and $(F\Theta_2, F^{\Ty\relax} T_2)$ are separated in $\bD$.
  Then $(\Theta_1,T_1)$ and $(\Theta_2,T_2)$ must also be separated in $\bC$.
\end{lemma}
\begin{proof}
  Take maps $f_i \colon \Gamma \to \Theta_i$ for $i=~1,2$ so that by the
  assumption that $F^{\Ty\relax} T_1$ is separated from $F^{\Ty\relax} T_2$ and the fact that
  $F$ respects substitution,
  $F^{\Ty\relax}(f_1^*T_1) = (Ff_1)^*(F^{\Ty\relax} T_1) \neq (Ff_2)^*(F^{\Ty\relax} T_2) =
  F^{\Ty\relax}(f_2^*T_2)$.
  In particular, if $f_1^*T_1 = f_2^*T_2$ then
  $F^{\Ty\relax}(f_1^*T_1) = F^{\Ty\relax}(f_2^*T_2)$, which we have concluded is impossible.
\end{proof}

Reflection of faithfulness, on the other hand, requires an additional assumption
that the map of underlying categories is a faithful functor.
\begin{lemma}\label{lem:faith-reflect}
  Suppose $F \colon \bC \to \bD \in \CwA$ has the underlying functor between
  underlying categories faithful.
  Then, $(\Theta \in \bC, T \in \Ty_\bC\Theta)$ is faithful in $\bC$ if
  $(F\Theta, F^{\Ty\relax} T)$ is faithful in $\bD$.
\end{lemma}
\begin{proof}
  By faithfulness of $F$ and $F^{\Ty\relax} T$, for all
  $f_1,f_2 \colon \Theta \rightrightarrows \Gamma$, if $f_1^*T = f_2^T$ then
  $(Ff_1)^*(F^{\Ty\relax} T) = F^{\Ty\relax}(f_1^*T) = F^{\Ty\relax}(f_2^*T) = (Ff_2)^*(F^{\Ty\relax} T)$ so
  $Ff_1 = Ff_2$ and so $f_1 = f_2$.
\end{proof}

\begin{lemma}\label{lem:norm-faith-sep}
  Let $(\bC,G)$ be a $\NCwA_\bT$-structure.
  Then, the set of types $\set{(\id_\Theta, T)}_{(\Theta,T) \in G}$ is
  faithfully separated in $\Norm(\bC,G) \in \CwA_\bT \hookrightarrow \CwA$.
\end{lemma}
\begin{proof}
  Fix $(\Theta_1,T_1) \neq (\Theta_2,T_2) \in G$.
  Then, for each $f_i \colon \Gamma \to \Theta_i$, the action on substitution in
  $\Norm(\bC,G)$ gives $(f_i, T_i)$ for $i=1,2$, so
  $f_1^*(\Theta_1,T_1) \neq f_2^*(\Theta_2,T_2)$ because either the codomain of
  the map in the first projection or the type in the second projection is
  different.
  Similarly, for each $(\Theta,T) \in G$, one has $f^*(\id_\Theta,T) = (f,T)$,
  so faithfulness follows.
\end{proof}

\begin{lemma}\label{lem:norm-struct-extend}
  Let $F \colon (\bC,G) \to (\bD,H) \in \NCwA_\bT$ and suppose $S$ and $L$ as below
  are sections of $\ev(\bC,G)$ and $\ev(\bD,G)$.
  \begin{equation*}
    \begin{tikzcd}
      \bC \ar[r, "{F}"] \ar[d, "{S}"']
      &
      \bD \ar[d, "{L}"]
      \\
      \Norm(\bC,G)
      \ar[r, "{\Norm(F)}"']
      &
      \Norm(\bD,H)
    \end{tikzcd}
  \end{equation*}
  Then, for each $\Gamma \in \bC$ and $X \in \Ty_\bC\Gamma$ such that $S^{\Ty\relax} X$
  is not a $\Pi,\Sigma,\Id$-type, neither is $F^{\Ty\relax} X$ and one has
  \begin{align*}
    (F(S^\Sub X), F^{\Ty\relax}(S^\Nf X)) = (L^\Sub(F^{\Ty\relax} X), L^\Nf(F^{\Ty\relax} X))
  \end{align*}
  In particular, $F$ sends $S$-normal forms to $L$-normal forms.
  Moreover when $H \subseteq \set{(F\Theta, F^{\Ty\relax} T)}_{(\Theta,T) \in G}$, if
  $S$ is a type-normalisation structure relative to $(\bC,G)$ then $L$ is a
  type-normalisation structure relative to $(\bD,H)$.
\end{lemma}
\begin{proof}
  For each $\Gamma \in \bC$ and $X \in \Ty_\bC\Gamma$, the bottom-left path
  gives
  $(\Norm(F) \cdot S)^{\Ty\relax} X = \Norm(F)^{\Ty\relax}(S^\Sub A, S^\Nf A) = (F(S^\Sub X),
  F^{\Ty\relax}(S^\Nf X))$ and the top-right path gives $(LF)^{\Ty\relax} X$.
  Commutativity says that they are equal
  \begin{align*}
    (F(S^\Sub X), F^{\Ty\relax}(S^\Nf X)) = (LF)^{\Ty\relax} X = (L^\Sub(F^{\Ty\relax} X), L^\Nf(F^{\Ty\relax} X))
  \end{align*}
  If $X$ is $S$-normal then $S^\Sub X = \id\relax$, so
  $L^\Sub(F^{\Ty\relax} X) = F(S^\Sub X) = \id\relax$, which says that $F$ sends
  $S$-normal forms to $L$-normal forms.
  So if everything in $H$ is the image under $F$ of a type in $G$ and all types
  in $G$ are $S$-normal they are also $L$-normal.
\end{proof}

\begin{proposition}\label{prop:add-tm-norm}
  Fix a cell as follows.
  \begin{equation*}
    % https://q.uiver.app/#q=WzAsNCxbMCwwLCJcXHZ2YnJ7XFxDdHggXFx2ZGFzaCBcXFR5cGV9Il0sWzEsMCwiXFx2dmJye1xcQ3R4IFxcdmRhc2ggXFxUZXJtIDogXFxUeXBlfSJdLFswLDEsIlxcYkMiXSxbMSwxLCJcXGJDIFxcY3VwX3tcXFRoZXRhLkF9IFxcVGVybSJdLFswLDIsIlxcVGhldGFfMC5UXzAiLDJdLFswLDEsIiIsMCx7InN0eWxlIjp7InRhaWwiOnsibmFtZSI6Imhvb2siLCJzaWRlIjoidG9wIn19fV0sWzIsMywiRiIsMl0sWzEsMywiXFxUZXJtX3tcXFRoZXRhXzAuVF8wfSJdXQ==
    \begin{tikzcd}
      {\vvbr{\Ctx \vdash \Type}} & {\vvbr{\Ctx \vdash \Term : \Type}} \\
      \bC & {\bC \cup_{\Theta.A} \Term}
      \arrow[hook, from=1-1, to=1-2]
      \arrow["{\Theta_0.T_0}"', from=1-1, to=2-1]
      \arrow["{\Term_{\Theta_0.T_0}}", from=1-2, to=2-2]
      \arrow["F"', from=2-1, to=2-2]
      \ar[from=2-2, to=1-1, "{\ulcorner}"{pos=0}, phantom]
    \end{tikzcd} \in \CwA_\bT
  \end{equation*}
  Suppose one has a type-normalisation structure $S \colon \bC \to \Norm(G)$
  with respect to $G \subseteq \coprod_{\Gamma \in \bC}\Ty_\bC\Gamma$.
  Then, one has a type-normalisation structure
  $S \cup_{\Theta_0.T_0} \Term \colon \bC \cup_{\Theta_0.T_0} \Term \to \Norm(FG)$ with
  respect to
  $FG \coloneqq \set{(F\Theta, F^{\Ty\relax} T)}_{(\Theta,T) \in G}$ such that
  \begin{equation*}
    % https://q.uiver.app/#q=WzAsNCxbMCwwLCJcXGJDIl0sWzEsMCwiXFxiQyBcXGN1cF97XFxUaGV0YV8wLlRfMH0gXFxUZXJtIl0sWzAsMSwiXFxOb3JtKEcpIl0sWzEsMSwiXFxOb3JtKEZHKSJdLFswLDEsIkgiXSxbMCwyLCJTIiwyXSxbMiwzLCJcXE5vcm0oRikiLDJdLFsxLDMsIlMgXFxjdXBfe1xcVGhldGFfMC5UXzB9IFxcVGVybSIsMCx7InN0eWxlIjp7ImJvZHkiOnsibmFtZSI6ImRhc2hlZCJ9fX1dXQ==
    \begin{tikzcd}
      \bC & {\bC \cup_{\Theta_0.T_0} \Term} \\
      {\Norm(G)} & {\Norm(FG)}
      \arrow["F", from=1-1, to=1-2]
      \arrow["S"', from=1-1, to=2-1]
      \arrow["{S \cup_{\Theta_0.T_0} \Term}", dashed, from=1-2, to=2-2]
      \arrow["{\Norm(F)}"', from=2-1, to=2-2]
    \end{tikzcd}
  \end{equation*}
  Moreover, if $G$ is faithfully separated in $\bC$ then $FG$ is faithfully
  separated in $\bC \cup_{\Theta_0.T_0} \Term$.
\end{proposition}
\begin{proof}
  By \Cref{lem:ev-tm-lift}, one has a lift $\widetilde{\Term_{\Theta_0.T_0}}$ as
  below:
  \begin{equation*}
    % https://q.uiver.app/#q=WzAsNixbMCwwLCJcXHZ2YnJ7XFxDdHggXFx2ZGFzaCBcXFR5cGV9Il0sWzAsMSwiXFx2dmJye1xcQ3R4IFxcdmRhc2ggXFxUZXJtIDogXFxUeXBlfSJdLFsxLDAsIlxcYkMiXSxbMiwwLCJcXE5vcm0oRykiXSxbNCwwLCJcXE5vcm0oRkcpIl0sWzQsMSwiXFxiQyBcXGN1cF97XFxUaGV0YS5BfSBcXFRlcm0iXSxbMCwyLCJcXFRoZXRhXzAuVF8wIl0sWzAsMSwiIiwwLHsic3R5bGUiOnsidGFpbCI6eyJuYW1lIjoiaG9vayIsInNpZGUiOiJ0b3AifX19XSxbMiwzLCJTIiwxXSxbMyw0LCJcXE5vcm0oRikiXSxbMSw0LCJcXHdpZGV0aWxkZXtcXFRlcm1fe1xcVGhldGFfMC5UXzB9fSIsMSx7InN0eWxlIjp7ImJvZHkiOnsibmFtZSI6ImRhc2hlZCJ9fX1dLFs0LDUsIlxcZXYoRkcpIl0sWzEsNSwiXFxUZXJtX3tcXFRoZXRhXzAuVF8wfSIsMl1d
    \begin{tikzcd}
      {\vvbr{\Ctx \vdash \Type}} & \bC & {\Norm(G)} && {\Norm(FG)} \\
      {\vvbr{\Ctx \vdash \Term : \Type}} &&&& {\bC \cup_{\Theta_0.T_0} \Term}
      \arrow["{\Theta_0.T_0}", from=1-1, to=1-2]
      \arrow[hook, from=1-1, to=2-1]
      \arrow["S"{}, from=1-2, to=1-3]
      \arrow["{\Norm(F)}", from=1-3, to=1-5]
      \arrow["{\ev(FG)}", from=1-5, to=2-5]
      \arrow["{\widetilde{\Term_{\Theta_0.T_0}}}"{description}, dashed, from=2-1, to=1-5]
      \arrow["{\Term_{\Theta_0.T_0}}"', from=2-1, to=2-5]
    \end{tikzcd}
  \end{equation*}
  whose existence gives rise to a map $S \cup_{\Theta_0.T_0} \Term$ as below.
  \begin{equation*}
    % https://q.uiver.app/#q=WzAsOCxbMCwwLCJcXHZ2YnJ7XFxDdHggXFx2ZGFzaCBcXFR5cGV9Il0sWzEsMCwiXFx2dmJye1xcQ3R4IFxcdmRhc2ggXFxUZXJtIDogXFxUeXBlfSJdLFswLDEsIlxcYkMiXSxbMSwxLCJcXGJDIFxcY3VwX3tcXFRoZXRhLkF9IFxcVGVybSJdLFsxLDIsIlxcTm9ybShHKSJdLFsyLDIsIlxcTm9ybShGRykiXSxbMiwzLCJcXGJDIl0sWzMsMywiXFxiQyBcXGN1cF97XFxUaGV0YV8wLlRfMH0gXFxUZXJtIl0sWzAsMiwiXFxUaGV0YV8wLlRfMCIsMl0sWzAsMSwiIiwwLHsic3R5bGUiOnsidGFpbCI6eyJuYW1lIjoiaG9vayIsInNpZGUiOiJ0b3AifX19XSxbMiwzLCJGIiwxXSxbMSwzLCJcXFRlcm1fe1xcVGhldGFfMC5UXzB9IiwxXSxbMiw0LCJTIiwxXSxbNCw1LCJcXE5vcm0oRikiLDFdLFszLDUsIlMgXFxjdXBfe1xcVGhldGFfMC5UXzB9IFxcVGVybSIsMSx7InN0eWxlIjp7ImJvZHkiOnsibmFtZSI6ImRhc2hlZCJ9fX1dLFsxLDUsIlxcd2lkZXRpbGRle1xcVGVybV97XFxUaGV0YV8wLlRfMH19IiwxLHsic3R5bGUiOnsiYm9keSI6eyJuYW1lIjoiZGFzaGVkIn19fV0sWzQsNiwiXFxldihHKSIsMV0sWzYsNywiRiIsMl0sWzUsNywiXFxldihGRykiLDFdLFsyLDYsIj0iLDIseyJjdXJ2ZSI6NH1dLFsxLDcsIlxcVGVybV97XFxUaGV0YV8wLlRfMH0iLDAseyJjdXJ2ZSI6LTR9XV0=
    \begin{tikzcd}
      {\vvbr{\Ctx \vdash \Type}} & {\vvbr{\Ctx \vdash \Term : \Type}} \\
      \bC & {\bC \cup_{\Theta_0.T_0} \Term} \\
      & {\Norm(G)} & {\Norm(FG)} \\
      && \bC & {\bC \cup_{\Theta_0.T_0} \Term}
      \arrow[hook, from=1-1, to=1-2]
      \arrow["{\Theta_0.T_0}"', from=1-1, to=2-1]
      \arrow["{\Term_{\Theta_0.T_0}}"{description}, from=1-2, to=2-2]
      \arrow["{\widetilde{\Term_{\Theta_0.T_0}}}"{description}, dashed, from=1-2, to=3-3]
      \arrow["{\Term_{\Theta_0.T_0}}", curve={height=-24pt}, from=1-2, to=4-4]
      \arrow["F"{description}, from=2-1, to=2-2]
      \arrow["S"{description}, from=2-1, to=3-2]
      \arrow["{=}"', curve={height=24pt}, from=2-1, to=4-3]
      \arrow["{S \cup_{\Theta_0.T_0} \Term}"{description}, dashed, from=2-2, to=3-3]
      \arrow["{\Norm(F)}"{description}, from=3-2, to=3-3]
      \arrow["{\ev(G)}"{description}, from=3-2, to=4-3]
      \arrow["{\ev(FG)}"{description}, from=3-3, to=4-4]
      \arrow["F"', from=4-3, to=4-4]
      \ar[from=2-2, to=1-1, "{\ulcorner}"{pos=0}, phantom]
    \end{tikzcd}
  \end{equation*}
  The bottom-most parallelogram
  $F \cdot \ev(G) = \ev(FG) \cdot \Norm(F)$ above commmutes by
  \Cref{lem:ev-nat}, so by the universal property of the pushout,
  $S \cup_{\Theta_0.T_0} \Term$ is a section of $\ev(FG)$.
  Finally, by \Cref{lem:norm-struct-extend} and the fact that
  $\Norm(F) \cdot S = (S \cup_{\Theta_0.T_0} \Term) \cdot F$, because $S$ is a
  type-normalisation structure relative to $G$ and
  $FG \subseteq \set{(F\Theta, F^{\Ty\relax} T)}_{(\Theta,T) \in G}$ (by definition the
  inclusion is in fact an equality), it follows that
  $S \cup_{\Theta_0.T_0} \Term$ is a type-normalisation structure relative to
  $FG$.

  Finally, assuming that $G$ is faithfully separated in $\bC$, we must check
  that $FG$ is faithfully separated in $\bC \cup_{\Theta_0.T_0} \Term$.
  Because $S \cup_{\Theta_0.T_0} \Term$ is a section of $\ev(FG)$, its action on
  underlying categories is the identity, so it is in particular faithful.
  By \Cref{lem:faith-reflect,lem:sep-reflect}, it reflects separation and
  faithfulness, so one may just check that the image of $FG$ under
  $S \cup_{\Theta_0.T_0} \Term$ is faithfully separated in $\Norm(FG)$.
  But then $S \cup_{\Theta_0.T_0} \Term$ is a type-normalisation structure with
  respect to $FG$, which means by definition that
  $(S \cup_{\Theta_0.T_0} \Term)^{\Ty\relax}(F^{\Ty\relax} T) = (\id_\Theta, F^{\Ty\relax} T)$ for each
  $(\Theta,T) \in G$.
  By \Cref{lem:norm-faith-sep}, the set
  $\set{(S \cup_{\Theta_0.T_0} \Term)^{\Ty\relax}(F_\Theta^{\Ty\relax} T)}_{(\Theta,T) \in G} =
  \set{(\id_\Theta, F^{\Ty\relax} T)}_{(\Theta,T) \in G}$ is faithfully separated in
  $\Norm(FG)$, and the result follows.
\end{proof}

\begin{proposition}\label{prop:add-ty-norm}
  Fix a cell as follows.
  \begin{equation*}
    % https://q.uiver.app/#q=WzAsNCxbMCwwLCJcXHZ2YnJ7XFxDdHh9Il0sWzEsMCwiXFx2dmJye1xcQ3R4IFxcdmRhc2ggXFxUeXBlfSJdLFswLDEsIlxcYkMiXSxbMSwxLCJcXGJDIFxcY3VwX3tcXFRoZXRhXzB9IFxcVHlwZSJdLFswLDIsIlxcVGhldGFfMCIsMl0sWzAsMSwiIiwwLHsic3R5bGUiOnsidGFpbCI6eyJuYW1lIjoiaG9vayIsInNpZGUiOiJ0b3AifX19XSxbMiwzLCJGIiwyXSxbMSwzLCJGXFxUaGV0YV8wLlxcVHlwZV97XFxUaGV0YV8wfSIsMV1d
    \begin{tikzcd}
      {\vvbr{\Ctx}} & {\vvbr{\Ctx \vdash \Type}} \\
      \bC & {\bC \cup_{\Theta_0} \Type}
      \arrow[hook, from=1-1, to=1-2]
      \arrow["{\Theta_0}"', from=1-1, to=2-1]
      \arrow["{F\Theta_0.\Type_{F\Theta_0}}"{}, from=1-2, to=2-2]
      \arrow["F"', from=2-1, to=2-2]
      \ar[from=2-2, to=1-1, "{\ulcorner}"{pos=0}, phantom]
    \end{tikzcd} \in \CwA_\bT
  \end{equation*}
  Suppose one has a type-normalisation structure $S \colon \bC \to \Norm(G)$
  with respect to $G \subseteq \coprod_{\Gamma \in \bC}\Ty_\bC\Gamma$.
  Then, one has a type-normalisation structure
  $S \cup_{\Theta_0} \Type \colon \bC \cup \Type \to \Norm((FG)^+)$ with respect to
  \begin{align*}
    (FG)^+ \coloneqq \set{(F\Theta, F^{\Ty\relax} T)}_{(\Theta,T) \in G} \sqcup
    \set{(F\Theta_0, \Type_{F\Theta_0})} \subseteq \coprod_{\Gamma \in \bC \cup \Type}
    \Ty_{\bC \cup \Type}\Gamma
  \end{align*}
  such that
  \begin{equation*}
    % https://q.uiver.app/#q=WzAsNCxbMCwwLCJcXGJDIl0sWzEsMCwiXFxiQyBcXGN1cF97XFxUaGV0YV8wfSBcXFR5cGUiXSxbMCwxLCJcXE5vcm0oRykiXSxbMSwxLCJcXE5vcm0oKEZHKV4rKSJdLFswLDEsIkgiXSxbMCwyLCJTIiwyXSxbMiwzLCJcXE5vcm0oRikiLDJdLFsxLDMsIlMgXFxjdXBfe1xcVGhldGFfMH0gXFx5cGUiLDAseyJzdHlsZSI6eyJib2R5Ijp7Im5hbWUiOiJkYXNoZWQifX19XV0=
    \begin{tikzcd}
      \bC & {\bC \cup_{\Theta_0} \Type} \\
      {\Norm(G)} & {\Norm((FG)^+)}
      \arrow["F", from=1-1, to=1-2]
      \arrow["S"', from=1-1, to=2-1]
      \arrow["{S \cup_{\Theta_0} \Type}", dashed, from=1-2, to=2-2]
      \arrow["{\Norm(F)}"', from=2-1, to=2-2]
    \end{tikzcd}
  \end{equation*}
  Moreover, if $G$ is faithfully separated in $\bC$ then $(FG)^+$ is faithfully
  separated in $\bC \cup_{\Theta_0} \Type$.
\end{proposition}
\begin{proof}
  As in \Cref{prop:add-tm-norm}, we construct a lift
  $\widetilde{F\Theta_0.\Type}$.
  However, because \Cref{lem:ev-tm-lift} is only for strict term-lifting and not
  type-lifting, we construct $\widetilde{F\Theta_0.\Type}$ by the universal
  property of $\vvbr{\Ctx \vdash \Type}$.
  Namely, $\widetilde{F\Theta_0.\Type}$ selects the context $F\Theta_0$ and the
  type
  $(F\Theta_0 \xrightarrow{=} F\Theta_0, \Type_{F\Theta_0}) \in
  \Ty_{\Norm((FG)^+)} F\Theta$.
  \begin{equation*}
    % https://q.uiver.app/#q=WzAsNixbMCwwLCJcXHZ2YnJ7XFxDdHh9Il0sWzAsMSwiXFx2dmJye1xcQ3R4IFxcdmRhc2ggXFxUeXBlfSJdLFsxLDAsIlxcYkMiXSxbMiwwLCJcXE5vcm0oRykiXSxbNCwwLCJcXE5vcm0oKEZHKV4rKSJdLFs0LDEsIlxcYkMgXFxjdXBfe1xcVGhldGFfMH0gXFxUZXJtIl0sWzAsMiwiXFxUaGV0YV8wIl0sWzAsMSwiIiwwLHsic3R5bGUiOnsidGFpbCI6eyJuYW1lIjoiaG9vayIsInNpZGUiOiJ0b3AifX19XSxbMiwzLCJTIl0sWzMsNCwiXFxOb3JtKEYpIl0sWzEsNCwiXFx3aWRldGlsZGV7RlxcVGhldGFfMC5cXFR5cGVfe0ZcXFRoZXRhXzB9fSIsMSx7InN0eWxlIjp7ImJvZHkiOnsibmFtZSI6ImRhc2hlZCJ9fX1dLFs0LDUsIlxcZXYoKEZHKV4rKSJdLFsxLDUsIkZcXFRoZXRhXzAuXFxUeXBlX3tGXFxUaGV0YV8wfSIsMl1d
    \begin{tikzcd}
      {\vvbr{\Ctx}} & \bC & {\Norm(G)} && {\Norm((FG)^+)} \\
      {\vvbr{\Ctx \vdash \Type}} &&&& {\bC \cup_{\Theta_0} \Term}
      \arrow["{\Theta_0}", from=1-1, to=1-2]
      \arrow[hook, from=1-1, to=2-1]
      \arrow["S", from=1-2, to=1-3]
      \arrow["{\Norm(F)}", from=1-3, to=1-5]
      \arrow["{\ev((FG)^+)}", from=1-5, to=2-5]
      \arrow["{\widetilde{F\Theta_0.\Type_{F\Theta_0}}}"{description}, dashed, from=2-1, to=1-5]
      \arrow["{F\Theta_0.\Type_{F\Theta_0}}"', from=2-1, to=2-5]
    \end{tikzcd}
  \end{equation*}
  whose existence gives rise to a section $S \cup_{\Theta_0} \Type$ as below.
  \begin{equation*}
    % https://q.uiver.app/#q=WzAsOCxbMCwwLCJcXHZ2YnJ7XFxDdHh9Il0sWzEsMCwiXFx2dmJye1xcQ3R4IFxcdmRhc2ggXFxUeXBlfSJdLFswLDEsIlxcYkMiXSxbMSwxLCJcXGJDIFxcY3VwX3tcXFRoZXRhXzB9IFxcVHlwZSJdLFsxLDIsIlxcTm9ybShHKSJdLFsyLDIsIlxcTm9ybSgoRkcpXispIl0sWzIsMywiXFxiQyJdLFszLDMsIlxcYkMgXFxjdXBfe1xcVGhldGFfMH0gXFxUeXBlIl0sWzAsMSwiIiwwLHsic3R5bGUiOnsidGFpbCI6eyJuYW1lIjoiaG9vayIsInNpZGUiOiJ0b3AifX19XSxbMCwyLCJ7XFxUaGV0YV8wfSIsMl0sWzEsMywie0ZcXFRoZXRhXzAuXFxUeXBlX3tGXFxUaGV0YV8wfX0iLDFdLFsxLDUsIntcXHdpZGV0aWxkZXtGXFxUaGV0YV8wLlxcVHlwZV97RlxcVGhldGFfMH19fSIsMSx7ImN1cnZlIjotMiwic3R5bGUiOnsiYm9keSI6eyJuYW1lIjoiZGFzaGVkIn19fV0sWzEsNywie0ZcXFRoZXRhXzAuXFxUeXBlX3tGXFxUaGV0YV8wfX0iLDAseyJjdXJ2ZSI6LTR9XSxbMiwzLCJGIiwxXSxbMiw0LCJTIiwxXSxbMiw2LCJ7PX0iLDIseyJjdXJ2ZSI6NH1dLFszLDUsIntTIFxcY3VwX3tcXFRoZXRhXzB9IFxcVHlwZX0iLDEseyJzdHlsZSI6eyJib2R5Ijp7Im5hbWUiOiJkYXNoZWQifX19XSxbNCw1LCJ7XFxOb3JtKEYpfSIsMV0sWzQsNiwie1xcZXYoRyl9IiwxXSxbNSw3LCJ7XFxldigoRkcpXispfSIsMV0sWzYsNywiRiIsMl0sWzMsMCwie1xcdWxjb3JuZXJ9IiwxLHsibGFiZWxfcG9zaXRpb24iOjAsInN0eWxlIjp7ImJvZHkiOnsibmFtZSI6Im5vbmUifSwiaGVhZCI6eyJuYW1lIjoibm9uZSJ9fX1dXQ==
    \begin{tikzcd}[column sep=large]
      {\vvbr{\Ctx}} & {\vvbr{\Ctx \vdash \Type}} \\
      \bC & {\bC \cup_{\Theta_0} \Type} \\
      & {\Norm(G)} & {\Norm((FG)^+)} \\
      && \bC & {\bC \cup_{\Theta_0} \Type}
      \arrow[hook, from=1-1, to=1-2]
      \arrow["{{\Theta_0}}"', from=1-1, to=2-1]
      \arrow["{{F\Theta_0.\Type_{F\Theta_0}}}"{description}, from=1-2, to=2-2]
      \arrow["{{\widetilde{F\Theta_0.\Type_{F\Theta_0}}}}"{description}, curve={height=-12pt}, dashed, from=1-2, to=3-3]
      \arrow["{{F\Theta_0.\Type_{F\Theta_0}}}", curve={height=-24pt}, from=1-2, to=4-4]
      \arrow["F"{description}, from=2-1, to=2-2]
      \arrow["S"{description}, from=2-1, to=3-2]
      \arrow["{{=}}"', curve={height=24pt}, from=2-1, to=4-3]
      \arrow["{{\ulcorner}}"{description, pos=0}, draw=none, from=2-2, to=1-1]
      \arrow["{{S \cup_{\Theta_0} \Type}}"{description}, dashed, from=2-2, to=3-3]
      \arrow["{{\Norm(F)}}"{description}, from=3-2, to=3-3]
      \arrow["{{\ev(G)}}"{description}, from=3-2, to=4-3]
      \arrow["{{\ev((FG)^+)}}"{description}, from=3-3, to=4-4]
      \arrow["F"', from=4-3, to=4-4]
    \end{tikzcd}
  \end{equation*}
  It remains to check that all types in
  $(FG)^+ = FG \sqcup \set{(F\Theta_0, \Type_{F\Theta_0})}$ are
  $(S \cup_{\Theta_0} \Type)$-normal.
  Because $\Norm(F) \cdot S = (S \cup \Type) \cdot F$, all types in $FG$ are
  $(S \cup_{\Theta_0} \Type)$-normal by \Cref{lem:norm-struct-extend}.
  The type $(F\Theta_0, \Type_{F\Theta_0})$ is
  $(S \cup_{\Theta_0} \Type)$-normal because
  $(S \cup \Type) \cdot (F\Theta.\Type_{F\Theta_0}) =
  \widetilde{F\Theta_0.\Type_{F\Theta_0}} \colon \vvbr{\Ctx \vdash \Type} \to
  \Norm((FG)^+)$, and $\widetilde{F\Theta_0.\Type_{F\Theta_0}}$ selects the type
  $(\id_{F\Theta_0}, \Type_{F\Theta_0})$.

  Finally, assuming that $G$ is faithfully separated in $\bC$, we must check
  that $(FG)^+$ is faithfully separated in $\bC \cup_{\Theta_0} \Type$.
  Following reasoning similar to \Cref{prop:add-tm-norm}, because
  $S \cup_{\Theta_0} \Type$ is a section of $\ev((FG)^+)$, its map on underlying
  categories is faithful, so by \Cref{lem:sep-reflect,lem:faith-reflect}, it
  suffices to check that it sends
  $(FG)^+ = FG \sqcup \set{(F\Theta_0, \Type_{F\Theta_0})}$ to a faithfully
  separated set of types in $\Norm((FG)^+)$.
  Using \Cref{lem:norm-faith-sep} like in \Cref{prop:add-tm-norm}, one sees that
  $FG$ is sent to a faithfully separated set of types in $\Norm((FG)^+)$.
  Moreover,
  $(S \cup_{\Theta_0} \Type)^{\Ty\relax}(F\Theta_0, \Type_{F\Theta_0}) =
  (\id_{F\Theta_0}, \Type_{F\Theta_0})$ by construction, which is faithful in
  $\Norm((FG)^+)$ by \Cref{lem:norm-faith-sep}.
  Finally, $(F\Theta_0, \Type_{F\Theta_0})$ is separated from $(F\Theta, T)$ in
  $\Norm((FG)^+)$ for each $(\Theta,T) \in G$ because by definition
  $\Type_{F\Theta_0}$ is a fresh type distinct from any previous type in $T$.
\end{proof}

\begin{corollary}\label{cor:cell-sep}
  Every $\McI_\bT^\CwA$-cell complex $0 \hookrightarrow \bC$ has a
  type-normalisation structure $S \colon \bC \to \Norm(\bC,G)$ with respect to a
  faithfully separating set of types
  $G \subseteq \coprod_{\Gamma \in \bC} \Ty_\bC\Gamma$.
\end{corollary}
\begin{proof}
  The $\CwA$ $0$ with just the empty context and no types has such a
  type-normalisation structure by taking $G = \emptyset$ and
  $S \colon 0 \to \Norm(0,0)$ to be the unique map.
  The result then follows by applying \Cref{prop:add-tm-norm,prop:add-ty-norm}
  on the filtration of $\bC$.
\end{proof}

\begin{corollary}\label{thm:cell-syntax}
  Suppose $0 \hookrightarrow \bC$ is an $\McI_\bT^\CwA$-cell complex.
  Then, there is a set
  $\set{(\Theta_i,T_i)}_{i \in I} \subseteq \coprod_{i \in I}\Ty_\bC\Theta_i$
  such that for all $\Gamma \in \bC$ and $X \in \Ty_\bC\Gamma$ where $X$ is not
  a $\Pi,\Sigma,\Id$-type then there is some $i \in I$ such that the
  $\Hom\relax$-set $\Hom_{\oint \Ty_\bC}((\Gamma,X), (\Theta_i,T_i))$ is a
  singleton and the $\Hom\relax$-set
  $\Hom_{\oint \Ty_\bC}((\Gamma,X), (\Theta_j,T_j))$ is empty for all
  $j \neq i$.
\end{corollary}
\begin{proof}
  By \Cref{cor:cell-sep}, there is a faithfully separating set of types
  $G \subseteq \coprod_{\Gamma \in \bC} \Ty_\bC \Gamma$ and a
  type-normalisation structure $S \colon \bC \to \Norm(\bC,G)$.
  Because $S$ is a section of $\ev(G)$ and $X$ is not a $\Pi,\Sigma,\Id$-type, \Cref{lem:sec-ev-char}
  says that
  \begin{align*}
    S^{\Ty\relax} X = (\Gamma \xrightarrow{S^\Sub X} S^\Ctx X, S^\Nf X)
  \end{align*}
  such that $(S^\Ctx X, S^\Nf X) \in G$ and $(S^\Sub X)^*(S^\Nf X) = X$.
  Hence, $S^\Sub X \in \Hom_{\oint \Ty_\bC}((\Gamma,X), (S^\Sub X, S^\Nf X)) \neq \emptyset$.
  And because $(S^\Ctx X, S^\Nf X) \in G$ is faithful, for all
  $f \in \Hom_{\oint \Ty_\bC}((\Gamma,X), (S^\Sub X, S^\Nf X))$ one has
  $f^*(S^\Nf X) = X = (S^\Sub X)^*(S^\Nf X)$ so $f = S^\Sub X$, and so
  $\Hom_{\oint \Ty_\bC}((\Gamma,X), (S^\Sub X, S^\Nf X)) = \set{S^\Sub X}$ is a
  singleton.
  Now, if $\Hom_{\oint \Ty_\bC}((\Gamma,X), (\Theta, T)) \neq \emptyset$ for any
  $(\Theta,T) \in G$ different from $(S^\Sub X, S^\Nf X)$ then by definition of
  the category of elements, there is some map $f \colon \Gamma \to \Theta$ such
  that $(S^\Sub X)^*(S^\Nf X) = X = f^*T$.
  But $G$ is separated, so this cannot be possible.
\end{proof}

\begin{theorem}\label{prop:cwa-cells-gen-base}
  Suppose $0 \hookrightarrow \bC$ is an $\McI^\CwA_\bT$-cell complex.
  Then there is a set
  $\set{(\Theta_i \in \bC, T_i \in \Ty_\bC\Theta_i)}_{i \in I}$ such that for
  every $\Gamma \in \bC$ and $X \in \Ty_\bC\Gamma$, exactly one of the following
  cases is true:
  \begin{enumerate}[ref={Case (\theenumi)}]
    \item \label{itm:gen-base-Pi} $X = \Pi(A,B)$ for some unique
    $A \in \Ty_\bC\Gamma$ and $B \in \Ty_\bC{\Gamma.B}$.
    \item \label{itm:gen-base-Sigma} $X = \Sigma(A,B)$ for some unique
    $A \in \Ty_\bC\Gamma$ and $B \in \Ty_\bC{\Gamma.A}$.
    \item \label{itm:gen-base-Id} $X = f^*\Id_A$ for some $\Delta \in \bC$ and
    type $A \in \Ty_\bC \Delta$ and map $f \colon \Gamma \to \Delta.A.A$.
    Moreover, the choice of $A$ is unique, in that if $X = g^*\Id_B$ for some
    other $B \in \Ty_\bC\Delta$ and $g \colon \Gamma \to \Delta.B.B$ then
    $A = B$.
    \item \label{itm:gen-base-base} There is some $i \in I$ such that in
    $\oint_\bC{\Ty_\bC\relax}$, the $\Hom\relax$-set
    $\Hom((\Gamma,T), (\Theta_i,T_i))$ is a singleton and the $\Hom\relax$-set
    $\Hom((\Gamma,T), (\Theta_j,T_j))$ is empty for all $j \neq i$.
  \end{enumerate}
  Moreover, for each $\Gamma \in \bC$, there is a function
  $\sz \colon \Ty_\bC\Gamma \to \bN$ such that:
  \begin{itemize}
    \item For all $\Gamma \in \bC$ and $A.B \in \Gamma$, one has that
    $\sz_\Gamma(A), \sz_{\Gamma.A}(B) < \sz_\Gamma(\Pi(A,B)), \sz_\Gamma(\Sigma(A,B))$
    \item For all $\Delta \in \bC$ and $A \in \Ty_\bC\Delta$ and map
    $f \colon \Gamma \to \Delta.A.A$, one has $\sz_\Delta(A) < \sz_\Gamma(f^*\Id_A)$.
    %
    % \item Each $\sz_{\Gamma}(f^*T_i) = 0$ for each
    % $f \colon \Gamma \to \Theta_i$, where $(\Theta_i,T_i)$ is as above.
    %
  \end{itemize}
\end{theorem}
\begin{proof}
  By \Cref{thm:cell-syntax}, \Cref{itm:gen-base-base} holds when none of
  \Cref{itm:gen-base-Pi,itm:gen-base-Sigma,itm:gen-base-Id} holds.
  To shows that \Cref{itm:gen-base-Pi,itm:gen-base-Sigma,itm:gen-base-Id} are
  disjoint is to show that $\Pi,\Sigma,\Id$-types are separated in $\bC$.
  By \Cref{lem:sep-reflect}, it suffices to exhibit a map
  $F \colon \bC \to \bD \in \CwA_\bT$ where $\Pi,\Sigma,\Id$-types are separated in $\bD$.
  We take $\bD$ to be the syntactic $\CwA_\bT$ containing only one base type
  $\MsT$ and one distinct formal term for each type, so that types in each
  context is either some weakening of $\MsT$ or the $\Pi,\Sigma,\Id$-type
  constructors applied to some weakening of $\MsT$ and all types are inhabited.
  In particular, weakenings of $\MsT$ and $\Pi,\Sigma,\Id$-types are all
  mutually syntactically separated in $\bD$.
  Then, applying induction on the cellular filtration of $\bC$, one obtains such
  a map $\bC \to \bD$, and by construction, $\Pi,\Sigma,\Id$-types are separated
  in $\bD$.
  Likewise, we can define the size function on types by putting $\sz_\Gamma(X)$
  to be the syntactic size of the type $F^{\Ty\relax} X$ in the syntactic category
  $\bD$.

  It remains to show the uniqueness parts of
  \Cref{itm:gen-base-Pi,itm:gen-base-Sigma,itm:gen-base-Id}.
  The uniqueness parts follow by exhibiting a $\CwA_\bT$ map
  $G \colon \bC \to \bE$ such that the uniqueness parts of
  \Cref{itm:gen-base-Pi,itm:gen-base-Sigma,itm:gen-base-Id} hold in $\bE$ and
  that the action on types of $G$ is faithful (i.e. is pointwise injective).
  This is because having such a $G$ implies that if
  $\Pi(A,B) = \Pi(A',B') \in \Ty_\bC\Gamma$ then
  $\Pi(G^{\Ty\relax}_\Gamma(A), G^{\Ty\relax}_{\Gamma.A}(B)) = G^{\Ty\relax}_\Gamma(\Pi(A,B)) =
  G^{\Ty\relax}_\Gamma(\Pi(A',B')) = \Pi(G^{\Ty\relax}_\Gamma(A'), G^{\Ty\relax}_{\Gamma.A'}(B'))$, so
  $G^{\Ty\relax}_\Gamma(A) = G^{\Ty\relax}_\Gamma(A')$, which implies $A = A'$ and
  $G^{\Ty\relax}_{\Gamma.A}(B) = G^{\Ty\relax}_{\Gamma.A'}(B')$, which implies $B = B'$.
  A similar argument applies for the $\Sigma$-type case.
  For the $\Id$-type case, if $f_1^*\Id_{A_1} = f_2^*\Id_{A_2}$ for some
  $\Delta$ and types $A_i \colon \Gamma \to \Delta.A_i.A_i$ and maps
  $f_i \colon \Gamma \to \Delta.A_i.A_i$ then
  $G_\Gamma^{\Ty\relax}(f_1^*\Id_{A_1}) = G(f_1)^*\Id_{G^{\Ty\relax}_\Delta(A_1)} =
  G(f_2)^*\Id_{G^{\Ty\relax}_\Delta(A_1)} = G_\Gamma^{\Ty\relax}(f_2^*\Id_{A_2})$, so
  $G^{\Ty\relax}_\Delta(A_1) = G^{\Ty\relax}_\Delta(A_2)$, from which $A_1 = A_2$ follows.

  We take $\bE$ to be the syntactic $\CwA_\bT$ consisting of $\alpha$ formal
  base types, where $\alpha$ is the number of steps in the filtration of $\bC$
  along with one formal term for each formal base type, so that again all types
  in $\bE$ are inhabited and by the filtration of $\bC$ there is a map
  $G \colon \bC \to \bE$ mapping each $T_i$ a distinct choice of the formal $\alpha$ base types.
  In particular, $G^{\Ty\relax}_{\Theta_i}T_i \neq G^{\Ty\relax}_{\Theta_j}T_j$ whenever
  $i \neq j$.
  Because $\bE$ is syntactic, the uniqueness parts of
  \Cref{itm:gen-base-Pi,itm:gen-base-Sigma,itm:gen-base-Id} hold.
  It remains to check that the map $G$ is such that for all
  $X,X' \in \Ty_\bC\Gamma$, if $G^{\Ty\relax}_\Gamma(X) = G^{\Ty\relax}_\Gamma(X')$ then
  $X = X'$.
  We proceed by nested induction on $\sz_\Gamma(X)$ and $\sz_\Gamma(X')$.
  Because $\Pi,\Sigma,\Id$- and base-types are distinct (i.e. cases
  \Cref{itm:gen-base-Id,itm:gen-base-Pi,itm:gen-base-Sigma,itm:gen-base-base}
  are mutually exclusive) in both $\bC$ and $\bE$ and $G$ preserves the logical
  structure, one may assume both $X$ and $X'$ are either both $\Pi$-types, both
  $\Sigma$-types, both $\Id$-types or both base types.
  The base type case is immediate because by construction
  $G^{\Ty\relax}_{\Theta_i}T_i \neq G^{\Ty\relax}_{\Theta_j}T_j$ whenever $i \neq j$.
  We next consider the case of $\Pi$-types as the cases of $\Id,\Sigma$-types
  proceed similarly.
  Suppose $X = \Pi(A,B)$ and $X' = \Pi(A',B')$ for
  $A.B, A'.B' \in \vec{\Ty}_\bC\Gamma$.
  In this case,
  $\Pi(G^{\Ty\relax}_\Gamma(A), G^{\Ty\relax}_{\Gamma.A}(B)) = G^{\Ty\relax}_\Gamma(X) =
  G^{\Ty\relax}_\Gamma(X') = \Pi(G^{\Ty\relax}_\Gamma(A'), G^{\Ty\relax}_{\Gamma.A'}(B'))$ in $\bE$, so
  $G^{\Ty\relax}_\Gamma(A) = G^{\Ty\relax}_\Gamma(A')$ and
  $G^{\Ty\relax}_{\Gamma.A}(B) = G^{\Ty\relax}_{\Gamma.A'}(B')$.
  Because $\sz_\Gamma(A) < \sz_\Gamma(X)$ and
  $\sz_\Gamma(A') < \sz_\Gamma(X')$, induction gives $A = A'$ so that
  $\Gamma.A = \Gamma.A'$ and again induction gives $B = B'$.
\end{proof}

Moving back to the contextual world, we can now prove \Cref{prop:cofib-ty-class}.
\begin{proposition}\label{prop:cofib-ty-class}
  If $\bC \in \CxlCat_{\ITT+\UIP}$ is $\McI^\CxlCat_{\ITT+\UIP}$-cellular then
  there is a set $\set{(\Theta_i \in \bC, A_i \in \Ty_\bC\Theta_i)}_{i \in I}$
  such that each object $\Gamma \in \ob_{n+1}\bC$ for $n \in \bN$ is exactly one
  of the following mutually exclusive cases:
  \begin{itemize}
    \item \emph{$\Pi$-types.} $\Gamma = \Delta.\Pi(X,Y)$ for unique
    $\Delta \in \bC$ and $X.Y \in \Ty_\bC\Delta$.
    \item \emph{$\Sigma$-types.} $\Gamma = \Delta.\Sigma(X,Y)$ for unique
    $\Delta \in \bC$ and $X.Y \in \Ty_\bC\Delta$.
    \item \emph{$\Id$-types.} $\Gamma = \Delta.f^*\Id_X$ for some
    $\Delta,\Delta' \in \bC$ and $X \in \Ty_\bC\Delta'$ and
    $f \colon \Delta \to \Delta'.X.X$.
    Moreover, the choice of $X$ is unique, in that if $\Gamma = \Delta.g^*\Id_Y$
    for some $Y \in \Ty_\bC\Delta'$ and $g \colon \Delta \to \Delta'.Y.Y$ then
    $Y = X$.
    \item \emph{Base types.} $\Gamma = \Delta.f^*A_i$ for unique
    $\Delta \in \bC$ and $f \colon \Delta \to \Theta_i$ and moreover
    $\Gamma \neq \Delta.g^*A_i$ for any $g \colon \Delta \to \Theta_j$ where
    $j \neq i$.
  \end{itemize}
  Moreover, there is a well-founded order $<$ on $\ob\bC$ such that for all
  $\Gamma \in \bC$ and $X.Y \in \vec{\Ty_\bC\relax}_\bC\Gamma$,
  \begin{align*}
    \Gamma.X.Y < \Gamma.\Pi(X,Y), \Gamma.\Sigma(X,Y)
    &&
       \Gamma < \Gamma.X
       %\qedhere
  \end{align*}
\end{proposition}
\begin{proof}
  By \Cref{prop:kl-4.13.4}, $\McI^\CxlCat_{\ITT+\UIP}$-cell complex in
  $\CxlCat_{\ITT+\UIP}$ are also $\McI^\CwA_{\ITT+\UIP}$-cell complexes under
  the inclusion $\CxlCat_{\ITT+\UIP} \hookrightarrow \CwA_{\ITT+\UIP}$.
  The result that $\Pi$-,$\Sigma$,$\Id$-types are separated follows from
  \Cref{prop:cwa-cells-gen-base}.
  The ordering on contexts is provided by the functions
  $(\sz_\Gamma \colon \Ty_\bC\Gamma \to \bN)_{\Gamma \in \bC}$.
\end{proof}

%%% Local Variables:
%%% mode: latex
%%% TeX-master: "./main.tex"
%%% TeX-engine: xetex
%%% End:

%% file: w-ett.tex
\section{The Extensional Kernel}\label{sec:w-ett}
We now describe the wide subcategory $\McW_\ETT$ of $\McW$ in an
$\McI^\CxlCat_{\ITT+\UIP}$-cell complex $\bC$ which we intend to collapse to
identities.
The organisation of this section is as follows:
\begin{itemize}
  \item In \Cref{def:id-pi,lem:id-pi} , we show that if
  $w_\Gamma \colon \Gamma_1 \tosimeq \Gamma_2$ and
  $w_A \colon \Gamma_1.A_1 \tosimeq \Gamma_2.A_2$ and
  $w_{AB} \colon \Gamma_1.A_1.B_1 \tosimeq \Gamma_2.A_2.B_2$ then one has an
  equivalence
  $w_\Pi(w_{AB},w_A,w_\Gamma) \colon \Gamma_1.\Pi(A_1,B_1) \tosimeq
  \Gamma_2.\Pi(A_2,B_2)$ over $w_\Gamma \colon \Gamma_1 \tosimeq \Gamma_2$ that
  preserves the necessary logical constructors.
  \Cref{def:id-id,lem:id-id} is the corresponding construction for $\Id$-types
  while \Cref{def:id-sigma,lem:id-sigma} is the corresponding construction for
  $\Sigma$-types.
  These definitions and lemmas correspond to \cite[Lemma 3.2.7]{hofmann:extensional-constructs}.
  \item In \Cref{lem:id-pi-gamma}, we show that a map
  $\gamma \colon \Delta_1.\Pi(f_1^*A_1, f_1^*B_1) \tosimeq \Delta_2.\Pi(f_2^*A_2,
  f_2^*B_2)$ obtained from \Cref{def:glue} is homotopic to a map $w_\Pi$
  obtained from \Cref{def:id-pi}.
  In other words, $w_\Pi$ is a ``canonical'' way of writing equivalences between $\Pi$-types.
  \Cref{lem:id-id-gamma,lem:id-sigma-gamma} are corresponding versions of this
  result for $\Id$- and $\Sigma$-types.
  Because Hofmann constructed his maps $\textsf{co}$ by way of inspecting the
  outer-most type former, there is no corresponding result in
  \cite{hofmann:extensional-constructs}.
  \item We then define $\McW_\ETT$ via an inductive process in \Cref{def:w-ett}.
  As previously mentioned, the maps in $\McW_\ETT$ corresponds to the maps
  $\textsf{co}$ in \cite[\S 3.2.5.2]{hofmann:extensional-constructs}.
  \item \Cref{lem:w-ett-free-ty} shows that maps in $\McW_\ETT$ between two
  freely-added base types are canonically expressed (up to homotopy) as a $\gamma$-map as in \Cref{def:glue}.
  Likewise, \Cref{lem:w-ett-canon} shows that maps in $\McW_\ETT$ between
  two $\Pi$-, $\Id$- or $\Sigma$-types are canonically expressed (up to
  homotopy) as some $w_\Pi$, $w_\Id$ or $w_\Sigma$ respectively from
  \Cref{def:id-pi,def:id-id,def:id-sigma}.
  Essentially, \Cref{lem:w-ett-canon} is proved by inductively applying
  \Cref{lem:id-pi-gamma,lem:id-id-gamma,lem:id-sigma-gamma}
  with the help of the classification of types property given by cofibrancy in
  \Cref{prop:cofib-ty-class}.
  \item Making use of the canonical form of weak equivalences in $\McW_\ETT$
  given by \Cref{lem:w-ett-free-ty,lem:w-ett-canon}, we repeatedly use $\UIP$ to
  show \Cref{prop:w-ett-thin-subgpd} , which states that objects identified by
  $\McW_\ETT$ have exactly one homotopy class of identification: if
  $w,w' \colon \Theta_1 \tosimeq \Theta_2 \in \McW_\ETT$ then $w \simeq w'$.
  We also show in \Cref{lem:w-ett-grading-proj} that the identification respects
  the grading in a contextual category.
  These results correspond to \cite[Lemma
  3.2.8]{hofmann:extensional-constructs}.
  \item In \Cref{def:w-ett-maps}, we define a quotient relation on the objects
  and maps of $\bC$ that in essence defines the objects and maps in the category
  obtained by formally collapsing all maps in $\McW_\ETT$ to identities.
  This construction corresponds to \cite[Definition 3.2.12]{hofmann:extensional-constructs}.
  In \Cref{lem:approx-char}, we give a useful characterisation of the
  equivalence relation on maps from \Cref{def:w-ett-maps} that is crucially
  reliant on the fact that objects identified by $\McW_\ETT$ have exactly one
  homotopy class of identification as in \Cref{prop:w-ett-thin-subgpd}, which
  itself is a consequence of $\UIP$ and cofibrancy.
  \item In \Cref{lem:ext-ker}, we show that all maps from $\bC$ to a model of
  extensional type theory must send all maps in $\McW_\ETT$ to identities.
  This corresponds to \cite[Lemma 3.2.11]{hofmann:extensional-constructs}.
\end{itemize}

\begin{definition}\label{def:id-pi}
  Let $w_{AB} \colon \Gamma_1.A_1.B_1 \tosimeq \Gamma_2.A_2.B_2$ and
  $w_A \colon \Gamma_1.A_1 \tosimeq \Gamma_2.A_2$ and
  $w_\Gamma \colon \Gamma_1 \tosimeq \Gamma_2$ be such that
  such that
  \begin{equation*}
    \begin{tikzcd}[cramped]
      \Gamma_1.A_1.B_1 \ar{r}{w_{AB}} \ar[two heads]{d} & \Gamma_2.A_2.B_2 \ar[two heads]{d} \\
      \Gamma_1.A_1 \ar{r}{w_A} \ar[two heads]{d} & \Gamma_2.A_2 \ar[two heads]{d} \\
      \Gamma_1 \ar{r}{w_\Gamma} & \Gamma_2
    \end{tikzcd}
  \end{equation*}
  Suppose $w_\Gamma$ is
  $(\vec{x}_1:\Gamma_1) \vdash t_\Gamma(\vec{x}_1) : \Gamma_2$.
  Write $w_A$ as $(\vec{x}_1:\Gamma_1)(a_1:A_1(\vec{x}_1))
  \vdash t_\Gamma(\vec{x}_1), t_A(\vec{x}_1,a_1)
  : \Gamma_2.A_2$ so that $w_{AB}$ is
  \begin{align*}
    (\vec{x}_1:\Gamma_1)(a_1:A_1(\vec{x}_1))(b_1:B_1(a_1,\vec{x}_1))
    \vdash t_\Gamma(\vec{x}_1), t_A(\vec{x}_1, a_1), t_B(\vec{x}_1, a_1,b_1) :
    \Gamma_2.A_2.B_2
  \end{align*}
  Take left inverse
  $(\vec{x}_2:\Gamma_2)(a_2:A_2(\vec{x}_2)(b_2:B_2(\vec{x}_2,a_2))
  \vdash \ell_\Gamma(\vec{x}_2), \ell_A(\vec{x}_2,a_2), \ell_B(\vec{x}_2,a_2,b_2) :
  \Gamma_1.A_1.B_1$.
  Define $w_\Pi$ as
  \begin{align*}
    (\vec{x}_1:\Gamma_1)(f_1:\Pi(A_1,B_1)) \vdash
    \begin{pmatrix}
      t_\Gamma(\vec{x}_1) \\
      \lambda(a_2:A_2).t_B(\ell_A(a_2), \app(f_1, \ell_A(a_2)))
    \end{pmatrix}
    :
    \Gamma_2.\Pi(A_2,B_2)
  \end{align*}
\end{definition}

\begin{lemma}\label{lem:id-pi}
  The map
  $w_\Pi(w_{AB},w_A,w_\Gamma) \colon \Gamma_1.\Pi(A_1,B_1) \tosimeq \Gamma_2.\Pi(A_2,B_2)$
  from \Cref{def:id-pi} is an equivalence over
  $w_\Gamma \colon \Gamma_1 \tosimeq \Gamma_2$.

  Moreover, if one has $w_{AB}' \colon \Gamma_1.A_1.B_1 \tosimeq \Gamma_2.A_2.B_2$
  over $w_A' \colon \Gamma_1.A_1 \tosimeq \Gamma_2.A_2$, all over
  $w_\Gamma' \colon \Gamma_1 \tosimeq \Gamma_2$ such that $w_{AB} \simeq w_{AB}'$
  and $w_A \simeq w_A'$ and $w_\Gamma \simeq w_\Gamma'$ then
  $w_\Pi(w_{AB},w_A,w_\Gamma) \simeq w_\Pi(w_{AB}',w_A',w_\Gamma') \colon \Gamma_1.\Pi(A_1,B_1)
  \tosimeq \Gamma_2.\Pi(A_2,B_2)$.
\end{lemma}
\begin{proof}
  For ease of understanding, we omit the $(\vec{\gamma_1}:\Gamma_1)$ and
  $(\vec{\gamma_2}:\Gamma_2)$ arguments.
  The left homotopy inverse $\overline{w_\Pi}^\ell$ is given by
  \begin{align*}
    (f_2:\Pi(a_2:A_2).B_2(a_2)) \vdash
    \lambda(a_1:A_1).\ell_B(t_A(a_1), \app(f_2, t_A(a_1))) : \Pi(A_1,B_1)
  \end{align*}
  and the right homotopy inverse $\overline{w_\Pi}^r$ is given by
  \begin{align*}
    (f_2:\Pi(a_2:A_2).B_2(a_2)) \vdash
    \lambda(a_1:A_1).r_B(t_A(a_1), \app(f_2, t_A(a_1))) : \Pi(a_1:A_1).B_1(a_1)
  \end{align*}
  where
  \begin{align*}
    (\vec{x}_2:\Gamma_2)(a_2:A_2(\vec{x}_2))(b_2:B_2(\vec{x}_2,a_2))
    \vdash
    r_\Gamma(\vec{x}_2),
    r_A(\vec{\gamma_2},a_2), r_B(\vec{\gamma_2},a_2,b_2) :
    (\vec{x}_1:\Gamma_1)(a_1:A_1(\vec{x}_1))(b_1:B_1(\vec{x}_1,a_1))
  \end{align*}
  is a right inverse of $w_{WB}$.

  Then, noting that each
  $t_A(\ell_A(a_2)) \simeq t_A(\ell_A(t_A(r_A(a_2)))) \simeq
  t_A(r_A(a_2)) \simeq a_2$ and making use of functional
  extensionality, we have that
  $(w_\Pi \cdot \overline{w_\Pi}^r)(f_2)$, for each
  $f_2:\Pi(a_2:A_2).B_2(a_2)$, is
  \begin{align*}
    & \lambda(a_2:A_2).t_B(\ell_A(a_2), r_B(\underbrace{t_A(\ell_A(a_2))}_{{} \simeq a_2},
      \app(f_2, \underbrace{t_A(\ell_A(a_2))}_{{} \simeq a_2}))) \\
    {}\simeq{} & \lambda(a_2:A_2).t_B(\ell_A(\underbrace{a_2}_{{} \simeq t_A(r_A(a_2))}),
                 r_B(a_2, \app(f_2,a_2))) \\
    {}\simeq{} & \lambda(a_2:A_2).
                 \underbrace{t_B(r_A(a_2),r_B(a_2, \app(f_2,a_2)))}_{{} \simeq \app(f_2,a_2)} \\
    {}\simeq{} & f_2
  \end{align*}
  Likewise, for the left inverse, for each
  $f_1:\Pi(a_1:A_1).B_1(a_1)$, we have
  $(\overline{w_\Pi}^\ell \cdot w_\Pi)(f_1)$ is
  \begin{align*}
    & \lambda(a_1:A_1).\ell_B(t_A(a_1), t_B(\underbrace{\ell_A(t_A(a_1))}_{{} \simeq a_1},
      \app(f_1, \underbrace{\ell_A(t_A(a_1))}_{{} \simeq a_1}))) \\
    {}\simeq{} & \lambda(a_1:A_1).
                 \underbrace{\ell_B(t_A(a_1),t_B(a_1, \app(f_1,a_1)))}_{{} \simeq \app(f_1,a_1)} \\
    {}\simeq{} & f_1
  \end{align*}

  Now, for another pair of equivalences
  $w_{AB}' \colon \Gamma_1.A_1.B_1 \tosimeq \Gamma_2.A_2.B_2$ over
  $w_A' \colon \Gamma_1.A_1 \tosimeq \Gamma_2.A_2$ both over $w_\Gamma' \colon \Gamma_1 \tosimeq \Gamma_2$,
  write $w_A'$ as $(a_1:A_1) \vdash t_A'(a_1) : A_2$ so that
  $w_{AB}'$ is
  \begin{align*}
    (a_1:A_1)(b_1:B_1(a_1)) \vdash t_A'(a_1), t_B'(a_1,b_2) : (a_2:A_2)(b_2:B_2(a_2))
  \end{align*}
  with left and right inverses respectively
  \begin{align*}
    (a_2:A_2)(b_2:B_2(a_2)) &\vdash
                              \ell_A'(a_2), \ell_B'(a_2,b_2) : (a_1:A_1)(b_1:B_1(a_1))
    \\
    (a_2:A_2)(b_2:B_2(a_2)) &\vdash
                              r_A'(a_2), r_B'(a_2,b_2) : (a_1:A_1)(b_1:B_1(a_1))
  \end{align*}
  If $w_A \simeq w_A'$ and $w_{AB} \simeq w_{AB}'$ then one has terms
  \begin{align*}
    (a_1:A_1)(b_1:B_1(a_1)) &\vdash H_1(a_1,b_1)
                              :
                              \Id((t_A(a_1), t_B(a_1,b_1)),
                              (t_A'(a_1), t_B'(a_1,b_1))) \\
    (a_2:A_2)(b_2:B_2(a_2)) &\vdash H_2(a_2,b_2)
                              :
                              \Id((\ell_A(a_2), \ell_B(a_2,b_2)),
                              (\ell_A'(a_2), \ell_B'(a_2,b_2)))
  \end{align*}
  So by repeated $\MsJ$-elimination on $H_1$ and $H_2$ along with functional
  extensionality, it follows that there is a term witnessing
  $w_\Pi(w_{AB},w_A,w_\Gamma) \simeq w_\Pi(w_{AB}',w_A',w_\Gamma')$.
\end{proof}

\begin{lemma}\label{lem:id-pi-gamma}
  For
  $\gamma = \gamma(w_\Pi(w_{AB},w_A,w_\Gamma), w_\Gamma, w_\Delta;
  f_1,f_2,\refl) \colon \Delta_1.\Pi(f_1^*A_1,f_1^*B_1) \tosimeq
  \Delta_2.\Pi(f_1^*A_2,f_2^*B_2)$ as in \Cref{def:glue}
  \begin{equation*}
    \begin{tikzcd}[cramped, sep=small]
      &
      \Delta_2.\Pi(f_2^*A_2,f_2^*B_2) \ar{rr} \ar[two heads]{dd}
      \ar[phantom]{rrdd}[pos=0em]{\lrcorner}
      & & \Gamma_2.\Pi(A_2,B_2) \\
      \Delta_1.\Pi(f_1^*A_1,f_1^*B_1)
      \ar[ur, "\gamma", "\simeq"']
      \ar[crossing over]{rr}
      \ar[phantom]{rrdd}[pos=0em]{\lrcorner}
      \ar[two heads]{dd}
      & & \Gamma_1.\Pi(A_1,B_1) \ar[ur, "\simeq"', "w_{\Gamma.\Pi}"]  \\
      & \Delta_2 \ar{rr}[near start]{f_2}
      & & \Gamma_2 \ar[crossing over, twoheadleftarrow]{uu} \\
      \Delta_1 \ar{rr}[swap]{f_1}
      \ar[ur, "w_\Delta", "\simeq"']
      & & \Gamma_1 \ar[crossing over, twoheadleftarrow]{uu}
      \ar[ur, "\simeq"', "w_\Gamma"]
    \end{tikzcd}
  \end{equation*}
  one has $\gamma \simeq w_\Pi(w_{AB}',w_A',w_\Gamma)$ for
  $w_A' \colon \Delta_1.f_1^*A_1 \tosimeq \Delta_2.f_2^*A_2$ as
  $w_A' \coloneqq \gamma(w_A,w_\Gamma,w_\Delta;f_1,f_2,\refl)$ and
  \begin{align*}
    w_{AB}' \coloneqq \gamma(w_{AB},w_A,w_A';f_1.A_1,f_2.A_2,\refl) \colon
    \Delta_1.f_1^*A_1.f_1^*B_1 \tosimeq \Delta_2.f_2^*A_2.f_2^*B_2
  \end{align*}
\end{lemma}
\begin{proof}
  This follows from the fact that $w_\Pi$ is stable under substitution as
  $\lambda$-abstractions are stable under substitution.
  Write $w_{AB}$ as
  \begin{align*}
    (\vec{x}_1:\Delta_1)(a_1:A_1(\vec{x}_1))(b_1:B_1(a_1,\vec{x}_1))
    \vdash w_\Gamma(\vec{x}_1), w_A(\vec{x}_1, a_1), w_{AB}(\vec{x}_1, a_1,b_1) :
    \Gamma_2.A_2.B_2
  \end{align*}
  with left inverse $\ell \colon \Gamma_2.A_2 \tosimeq \Gamma_1.A_1$ of $w_A$ given by
  $(\vec{x}_2:\Gamma_2)(a_2:A_2(\vec{x}_2))
  \vdash \ell_\Gamma(\vec{x}_2), \ell_A(\vec{x}_2,a_2) :
  \Gamma_1.A_1$
  so that $w_{\Gamma.\Pi}$ is, by \Cref{def:id-pi}, is given as
  \begin{align*}
    (\vec{x}_1:\Gamma_1)(k_1:\Pi(a_1:A_1(\vec{x}_1)).B_1(\vec{x}_1,a_1))
    \vdash
    \begin{pmatrix}
      w_\Gamma(\vec{x}_1) \\
      \lambda(a_2:A_2).w_{AB}(\ell_A(a_2), \app(k_1, \ell_A(a_2)))
    \end{pmatrix}
    :
    \Gamma_2.\Pi(A_2,B_2)
  \end{align*}
  And therefore,
  $\gamma \colon \Delta_1.\Pi(f_1^*A_1,f_1^*B_1) \tosimeq
  \Gamma_2.\Pi(f_2^*A_2,f_2^*B_2)$ is given by
  \begin{align*}
    (\vec{y}_1:\Delta_1)(h_1:\Pi(a_1:A_1(f_1(\vec{y}_1))).B_1(f_1(\vec{y}_1),a_1))
    \vdash
    \begin{pmatrix}
      w_\Delta(\vec{y}_1) \\
      w_{\Gamma.\Pi}^*(f_1\vec{y}_1,h_1)
    \end{pmatrix}
    :
    \Delta_2.\Pi(f_2^*A_2,f_2^*B_2)
  \end{align*}
  where $w_{\Gamma.\Pi}^*(f_1\vec{y}_1,h_1)$ is given by
  \begin{align*}
    \lambda(a_2:A_2(f_1\vec{y}_1)).
    w_{AB}(w_\Delta(\vec{y}_1),
    \ell_A(w_\Gamma(f_1\vec{y}_1), a_2),
    \app(h_1, \ell_A(w_\Gamma(f_1\vec{y}_1), a_2)))
  \end{align*}

  On the other hand, note that $w_{AB}' = \gamma(w_{AB},w_A,w_A';f_1.A_1,f_2.A_2)$ is
  \begin{align*}
    (\vec{y}_1:\Delta_1)(a_1:f_1^*A_1)(b_1:f_1^*B_1(a_1))
    \vdash
    w_\Delta(\vec{y}_1), w_A(w_\Delta(\vec{y}_1), a_1), w_{AB}(w_\Delta(\vec{y}_1),a_1,b_1)
    :
    \Delta_2.f_2^*A_2.f_2^*B_2
  \end{align*}
  furthermore, for $\ell_\Delta \colon \Delta_2 \to \Delta_1$ a left inverse of
  $w_\Delta \colon \Delta_1 \tosimeq \Delta_2$, one has that
  $\ell_A' \coloneqq \gamma(\ell_A,\ell_\Gamma,\ell_\Delta; f_2,f_1) \colon
  \Delta_2.f_2^*A_2 \to \Delta_1.f_1^*A_1$ is a left inverse of
  $w_A' = \gamma(w_A,w_\Gamma,w_\Delta;f_1,f_2)$ and is given by
  \begin{align*}
    (\vec{y}_2:\Delta_2)(a_2:A_2(f_2\vec{y}_2))
    \vdash
    \ell_\Delta(\vec{y}_2),
    \ell_A(f_2\vec{y}_2, a_2)
    :
    \Delta_1.f_1^*A_1
  \end{align*}
  Hence, $w_\Pi(w_{AB}',w_A',w_\Gamma)$ is given by
  \begin{align*}
    \lambda(a_2:A_2(f_1\vec{y}_1)).
    w_{AB}(w_\Delta(\vec{y}_1),
    \ell_A(f_2w_\Delta(\vec{y}_1), a_2),
    \app(h_1, \ell_A(f_2w_\Delta(\vec{y}_1), a_2)))
  \end{align*}
  But $f_2w_\Delta = w_\Gamma f_1$.
  And for any other choice of $\ell_A''$ left inverse of $w_A'$, one has
  $\ell_A'' \simeq \ell_A'$ constructed as above, so
  $\gamma \simeq w_\Pi(w_{AB}',w_A',w_\Gamma)$, as claimed.
\end{proof}

\begin{definition}\label{def:id-id}
  Suppose $w_A \colon \Gamma_1.A_1 \tosimeq \Gamma_1.A_2$ over
  $w_\Gamma \colon \Gamma_1 \tosimeq \Gamma_2$.
  Suppose $w_\Gamma$ is
  $(\vec{x}_1:\Gamma_1) \vdash t_\Gamma(\vec{x}_1):\Gamma_2$.
  Write $w_A$ as
  $(\vec{x}_1:\Gamma_1)(a_1:A_1(\vec{x}_1)) \vdash t_\Gamma(\vec{x}_1), t_A(\vec{x}_1, a_1) :
  \Gamma_2.A_2$.
  Define $w_{\Id}(w_A,w_\Gamma) \colon \Gamma_1.A_1.A_1.\Id_{A_1} \to
  \Gamma_2.A_2.A_2.\Id_{A_2}$ as the map given by
  \begin{align*}
    (\vec{x}_1:\Gamma_1)(a_1~a_1':A_1(\vec{x}_1))(p_1 : \Id_{A_1(\vec{x}_1)}(a_1,a_1'))
    \vdash
    t_\Gamma(\vec{x}_1),
    t_A(\vec{x}_1,a_1),
    t_A(\vec{x}_1,a_1'),
    \transport(p_1)
    : \Gamma_2.A_2.A_2.\Id_{A_2}
  \end{align*}
  where $\transport$ is the standard transport map.
\end{definition}

\begin{lemma}\label{lem:id-id}
  The map
  $w_{\Id}(w_A,w_\Gamma) \colon \Gamma_1.A_1.A_1.\Id_{A_1} \tosimeq
  \Gamma_1.A_2.A_2.\Id_{A_2}$ from \Cref{def:id-id} is an equivalence over
  $w_\Gamma \colon \Gamma_1 \tosimeq \Gamma_2$ mapping constructors to
  constructors:
  \begin{equation*}
    \begin{tikzcd}[cramped]
      \Gamma_1.A_1 \ar{r}{\refl_1} \ar{d}{\simeq}[swap]{w_A}
      & \Gamma_1.A_1.A_1.\Id_{A_1} \ar{d}{w_\Id}[swap]{\simeq} \ar[two heads]{r}
      & \Gamma_1.A_1.A_1 \ar{d}{\gamma}[swap]{\simeq}
      \\
      \Gamma_2.A_2 \ar{r}[swap]{\refl_2} & \Gamma_2.A_2.A_2.\Id_{A_2} \ar[two heads]{r}
      & \Gamma_2.A_2.A_2
    \end{tikzcd}
  \end{equation*}
  where $\gamma = \gamma(w_A,w_\Gamma,w_A; \pi,\pi)$ is as from \Cref{def:glue}.

  Moreover, if there is another $w_A' \colon \Gamma_1.A_1 \tosimeq \Gamma_2.A_2$
  over $w_\Gamma' \colon \Gamma_1 \tosimeq \Gamma_2$ such that
  $w_\Gamma \simeq w_\Gamma'$ and $w_A \simeq w_A'$ then
  $w_\Id(w_A,w_\Gamma) \simeq w_\Id(w_A',w_\Gamma')$.
\end{lemma}
\begin{proof}
  The map $w_{\Id}$ is an equivalence by $\UIP$ and weak functoriality of
  transport obtained via $\MsJ$-elimination.
  Since transport computes at reflexivity, we have
  $w_{\Id} \cdot \refl_1 = \transport(\refl) = \refl_2 \cdot w_A$.
  This proves the left square commutes.
  Furthermore, note that $\gamma(w_A,w_\Gamma,w_A; \pi,\pi)$ is also
  $(\vec{x}_1:\Gamma_1)(a_1~a_1':A_1(\vec{x})_1) \vdash
  t_\Gamma(\vec{x}_1), t_A(\vec{x}_1,a_1), t_A(\vec{x}_1,a_1') : \Gamma_2.A_2.A_2$.
  This proves the right square commutes.

  Now, for another
  $w_A' \colon \Gamma_1.A_1 \tosimeq \Gamma_2.A_2$  given by
  $(\vec{x}_1:\Gamma_1)(a_1:A_1(\vec{x}_1)) \vdash t_\Gamma'(\vec{x}_1),
  t_A(\vec{x}_1,a_1) \vdash \Gamma_2.A_2$,
  if $w_A' \simeq w_A$ then there is
  \begin{align*}
    (\vec{x}_1:\Gamma_1)(a_1:A_1(\vec{x}_1)) \vdash H(\vec{x}_1,a_1) :
    \Id_{\Gamma_2.A_2}((w_\Gamma(\vec{x}_1), w_A(\vec{x}_1,a_1)),
    (w_\Gamma'(\vec{x}_1), w_A'(\vec{x}_1,a_1)))
  \end{align*}
  Applying $\MsJ$-elimination on $H$ gives a term of type
  \begin{align*}
    (\vec{x}_1:\Gamma_1)(a_1~a_1':A_1(\vec{x}_1))(p_1:\Id_{A_1(\vec{x}_1)}(a_1,a_1')) \vdash
    \Id_{\Gamma_2.A_2.A_2.\Id_{A_2}}(
    w_\Id(w_A,w_\Gamma),
    w_\Id(w_A',w_\Gamma')
    )
  \end{align*}
  which means $w_\Id(w_A,w_\Gamma) \simeq w_\Id(w_A',w_\Gamma')$.
\end{proof}

\begin{lemma}\label{lem:id-id-gamma}
  Fix $w_\Gamma \colon \Gamma_1 \tosimeq \Gamma_2$ and
  $w_A \colon \Gamma_1.A_1 \tosimeq \Gamma_2.A_2$ and
  $w_\Id(w_A,w_\Gamma) \colon \Gamma_1.A_1.A_1.\Id_{A_1} \tosimeq
  \Gamma_2.A_2.A_2.\Id_{A_2}$ as in \Cref{def:id-id}.
  For
  $\gamma = \gamma(w_\Id(w_A,w_\Gamma), w_A^*, w_\Delta; f_1,f_2,\refl) \colon
  \Delta_1.f_1^*A_1,f_1^*A_1.\Id_{f_1^*A_1} \tosimeq
  \Delta_2.f_2^*A_2,f_2^*A_2.\Id_{f_2^*A_2}$ as in \Cref{def:glue} where
  $w_A^* \colon \Gamma_1.A_1.A_1 \tosimeq \Gamma_2.A_2.A_2$ is given by
  $\gamma(w_A,\id,w_A;\pi,\pi,\refl)$
  \begin{equation*}
    \begin{tikzcd}[cramped, sep=small]
      &
      \Delta_2.f_2^*A_2.f_2^*A_2.\Id_{f_2^*A_2} \ar{rr} \ar[two heads]{dd}
      \ar[phantom]{rrdd}[pos=0em]{\lrcorner}
      & & \Gamma_2.A_2.A_2.\Id_{A_2} \\
      \Delta_1.f_1^*A_1.f_1^*A_1.\Id_{f_1^*A_1}
      \ar[ur, "\gamma", "\simeq"']
      \ar[crossing over]{rr}
      \ar[phantom]{rrdd}[pos=0em]{\lrcorner}
      \ar[two heads]{dd}
      & & \Gamma_1.A_1.A_1.\Id_{A_1} \ar[ur, "\simeq"', "w_{\Id}"]  \\
      & \Delta_2 \ar{rr}[near start]{f_2}
      & & \Gamma_2.A_2.A_2 \ar[crossing over, twoheadleftarrow]{uu} \\
      \Delta_1 \ar{rr}[swap]{f_1}
      \ar[ur, "w_\Delta", "\simeq"']
      & & \Gamma_1.A_1.A_1 \ar[crossing over, twoheadleftarrow]{uu}
      \ar[ur, "\simeq"', "w_A^*"]
    \end{tikzcd}
  \end{equation*}
  one has $\gamma = w_\Id(w_A',w_\Delta)$ for
  $w_A' \colon \Delta_1.f_1^*A_1 \tosimeq \Delta_2.f_2^*A_2$ given by
  $w_A' \coloneqq \gamma(w_A,w_\Gamma,w_\Delta;f_1,f_2,\refl)$.
\end{lemma}
\begin{proof}
  This follows from the fact that $w_\Id$ is stable under substitution.
  By definition, $w_\Id(w_A,w_\Gamma)$ is given as
  \begin{align*}
    (\vec{x}_1:\Gamma_1)(a_1~a_1':A_1(\vec{x}_1))(p:\Id_{A_1(\vec{x}_1)})
    \vdash
    w_\Gamma(\vec{x}_1), w_A(\vec{x}_1,a_1), w_A(\vec{x}_1,a_1'), \transport(p)
    :
    \Gamma_2.A_2.A_2.\Id_{A_2}
  \end{align*}
  and $w_A^*$ is given as
  \begin{align*}
    (\vec{x}_1:\Gamma_1)(a_1~a_1':A_1(\vec{x}_1))
    \vdash
    w_\Gamma(\vec{x}_1), w_A(\vec{x}_1,a_1), w_A(\vec{x}_1,a_1')
    :
    \Gamma_2.A_2.A_2
  \end{align*}
  Moreover, $w_A'$ is
  \begin{align*}
    (\vec{y}_1:\Delta_1)(a_1~a_1':A_1(f_1\vec{y}_1))
    \vdash
    w_\Delta(\vec{y}_1), w_A(f_1\vec{y}_1,a_1), w_A(f_1\vec{y}_1,a_1')
    :
    \Delta_2.f_2^*A_2.f_2^*A_2
  \end{align*}
  Therefore, $w_\Id(w_A',w_\Delta)$ is
  \begin{align*}
    (\vec{y}_1:\Delta_1)(a_1~a_1':A_1(f_1\vec{y}_1))(p:\Id_{A_1(f_1\vec{y}_1)})
    \vdash
    w_\Delta(\vec{y}_1), w_A(f_1\vec{y}_1,a_1), w_A(f_1\vec{y}_1,a_1'), \transport(p)
    :
    \Delta_2.f_2^*A_2.f_2^*A_2.\Id_{f_2^*A_2}
  \end{align*}
  which is exactly $\gamma(w_\Id(w_A,w_\Gamma), w_A^*, w_\Delta;f_1,f_2)$.
\end{proof}

\begin{definition}\label{def:id-sigma}
  Let $w_{AB} \colon \Gamma_1.A_1.B_1 \tosimeq \Gamma_2.A_2.B_2$ and
  $w_A \colon \Gamma_1.A_1 \tosimeq \Gamma_2.A_2$ and
  $w_\Gamma \colon \Gamma_1 \tosimeq \Gamma_2$ be such that
  such that
  \begin{equation*}
    \begin{tikzcd}[cramped]
      \Gamma_1.A_1.B_1 \ar{r}{w_{AB}} \ar[two heads]{d} & \Gamma_2.A_2.B_2 \ar[two heads]{d} \\
      \Gamma_1.A_1 \ar{r}{w_A} \ar[two heads]{d} & \Gamma_2.A_2 \ar[two heads]{d} \\
      \Gamma_1 \ar{r}{w_\Gamma} & \Gamma_2
    \end{tikzcd}
  \end{equation*}
  Suppose $w_\Gamma$ is
  $(\vec{x}_1:\Gamma_1) \vdash t_\Gamma(\vec{x}_1) : \Gamma_2$.
  Write $w_A$ as $(\vec{x}_1:\Gamma_1)(a_1:A_1(\vec{x}_1))
  \vdash t_\Gamma(\vec{x}_1), t_A(\vec{x}_1,a_1)
  : \Gamma_2.A_2$ so that $w_{AB}$ is
  \begin{align*}
    (\vec{x}_1:\Gamma_1)(a_1:A_1(\vec{x}_1))(b_1:B_1(a_1,\vec{x}_1))
    \vdash t_\Gamma(\vec{x}_1), t_A(\vec{x}_1, a_1), t_B(\vec{x}_1, a_1,b_1) :
    \Gamma_2.A_2.B_2
  \end{align*}
  Putting $d(\vec{x}_1, a_1, b_1) \coloneqq \pair_2(t_A(\vec{x}_1, a_1),
  t_B(\vec{x}_1,a_1,b_1))$
  gives
  \begin{align*}
    (\vec{x}_1:\Gamma_1).\Sigma(a_1:A_1(\vec{x}_1)).B_1(\vec{x}_1,a_1)
    \vdash
    t_\Gamma(\vec{x}_1), \sigmarec_{d(\vec{x}_1)}
    :
    \Gamma_2.\Sigma(A_2,B_2)
  \end{align*}
  which we take as $w_\Sigma(w_{AB},w_A,w_\Gamma)$.
\end{definition}

\begin{lemma}\label{lem:id-sigma}
  The map
  $w_\Sigma(w_{AB},w_A,w_\Gamma) \colon \Gamma_1.\Sigma(A_1.B_1) \tosimeq
  \Gamma_2.\Sigma(A_2,B_2)$ from \Cref{def:id-sigma} is over $\Gamma$ mapping
  constructors to constructors:
  \begin{equation*}
    \begin{tikzcd}[cramped]
      \Gamma_1.A_1.B_1 \ar{r}{\pair_1} \ar{d}{\simeq}[swap]{w_{AB}}
      & \Gamma_2.\Sigma(A_1,B_1) \ar{d}{w_\Sigma}[swap]{\simeq} \\
      \Gamma_2.A_2.B_2 \ar{r}[swap]{\pair_2} & \Gamma_2.\Sigma(A_2,B_2)
    \end{tikzcd}
  \end{equation*}

  Moreover, if there are
  $w_{AB}' \colon \Gamma_1.A_1.B_1 \tosimeq \Gamma_1.A_2.B_2$ over
  $w_A' \colon \Gamma_2.A_1 \tosimeq \Gamma_2.A_2.B_2$ over
  $w_\Gamma' \colon \Gamma_1 \tosimeq \Gamma_2$ such that
  $w_{AB} \simeq w_{AB}'$ and
  $w_A \simeq w_A'$ and
  $w_\Gamma \simeq w_\Gamma'$
  then
  $w_\Sigma(w_{AB}) \simeq w_\Sigma(w_{AB}')$.
\end{lemma}
\begin{proof}
  Taking the left homotopy inverse of $w_{AB}$ as
  \begin{align*}
    (\vec{x}_2:\Gamma_2)(a_2:A_2(\vec{x}_2))(b_2:B_2(a_2)) \vdash
    \ell_\Gamma(\vec{x}_2), \ell_A(\vec{x}_2,a_2), \ell_B(\vec{x}_2,a_2,b_2) :
    \Gamma_2.A_2.B_2
  \end{align*}
  we have a left inverse
  $\overline{w_\Sigma}^\ell$ given by
  \begin{align*}
    (\vec{x}_2:\Gamma_2).\Sigma(a_2:A_2(\vec{x}_2)).B_2(\vec{x}_2,a_2)
    \vdash
    \ell_\Gamma(\vec{x}_2), \sigmarec_{e(\vec{x}_2)}
    :
    \Gamma_2.A_2.B_2
  \end{align*}
  where
  $e(\vec{x}_2, a_2,b_2) \coloneqq \pair_1(\ell_A(\vec{x}_2,a_2),
  \ell_B(\vec{x}_2,a_2,b_2))$ because
  $(\ell_\Gamma \cdot t_\Gamma)(\vec{x}_2) \simeq \vec{x}_2$ and
  \begin{align*}
    (\sigmarec_e \cdot \sigmarec_d)(\pair_1(a_1,b_1))
    &= \sigmarec_e(\pair_2(t_A(a_1), t_B(a_1)) \\
    &= \pair_1(\ell_A(t_A(a_1), t_B(a_1,b_1)), \ell_B(t_A(a_1),
      t_B(a_1))) \\
    (\sigmarec_e \cdot \sigmarec_d)(\pair_1(a_1,b_1))
    &\simeq \pair_1(a_1,b_1)
  \end{align*}
  This completes the verification of the left inverse property thanks
  to $\Sigma$-induction.
  An identical construction verifies the right inverse property.
  Furthermore, observe that by the computation rule for the $\Sigma$-recursor,
  $w_\Sigma \cdot \pair_1 = \sigmarec_d \cdot \pair_1 = d = \pair_2 \cdot
  w_{AB}$.

  If there is another $w_{AB}' \colon \Gamma_1.A_1.B_1 \tosimeq \Gamma_1.A_2.B_2$
  over $w_A' \colon \Gamma_1.A_1 \tosimeq \Gamma_2.A_2$ over
  $w_\Gamma' \colon \Gamma_1 \tosimeq \Gamma_2$
  given as
  \begin{align*}
    (\vec{x}_1:\Gamma_1)(a_1:A_1(\vec{x}_1))(b_1:B_1(\vec{x}_1,a_1))
    \vdash t_\Gamma'(\vec{x}_1), t_A'(\vec{x}_1,a_1,b_1), t_B'(\vec{x}_1,a_1,b_1)
    :
    \Gamma_2.A_2.B_2
  \end{align*}
  such that there is
  \begin{align*}
    (\vec{x}_1:\Gamma_1)(a_1:A_1)(b_1:B_1(a_1)) \vdash H \colon
    \Id((t_\Gamma(\vec{x}_1), t_A(a_1), t_B(a_1,b_1))
    (t_\Gamma'(\vec{x}_1), t_A'(a_1), t_B'(a_1,b_1)))
  \end{align*}
  Then, one has
  \begin{align*}
    (t_\Gamma(\vec{x}_1), d(\vec{x}_1,a_1,b_1))
    &=
      (t_\Gamma(\vec{x}_1), \pair_2(t_A(\vec{x}_1,a_1),t_B(\vec{x}_1,a_1,b_1))) \\
    &\simeq
      (t_\Gamma'(\vec{x}_1), \pair_2(t_A'(\vec{x}_1,a_1),t_B'(\vec{x}_1,a_1,b_1))) \\
    &= (t_\Gamma'(\vec{x}_1),d'(a_1,b_1))
  \end{align*}
  From this, it follows that
  $w_\Sigma(w_{AB},w_A,w_\Gamma) \simeq w_\Sigma(w_{AB}',w_A',w_\Gamma')$.
\end{proof}

\begin{lemma}\label{lem:id-sigma-gamma}
  For
  $\gamma = \gamma(w_\Sigma(w_{AB},w_A,w_\Gamma), w_\Gamma, w_\Delta;
  f_1,f_2,\refl) \colon \Delta_1.\Sigma(f_1^*A_1,f_1^*B_1) \tosimeq
  \Delta_2.\Sigma(f_1^*A_2,f_2^*B_2)$ as in \Cref{def:glue}
  \begin{equation*}
    \begin{tikzcd}[cramped, sep=small]
      &
      \Delta_2.\Sigma(f_2^*A_2,f_2^*B_2) \ar{rr} \ar[two heads]{dd}
      \ar[phantom]{rrdd}[pos=0em]{\lrcorner}
      & & \Gamma_2.\Sigma(A_2,B_2) \\
      \Delta_1.\Sigma(f_1^*A_1,f_1^*B_1)
      \ar[ur, "\gamma", "\simeq"']
      \ar[crossing over]{rr}
      \ar[phantom]{rrdd}[pos=0em]{\lrcorner}
      \ar[two heads]{dd}
      & & \Gamma_1.\Sigma(A_1,B_1) \ar[ur, "\simeq"', "w_{\Gamma.\Sigma}"]  \\
      & \Delta_2 \ar{rr}[near start]{f_2}
      & & \Gamma_2 \ar[crossing over, twoheadleftarrow]{uu} \\
      \Delta_1 \ar{rr}[swap]{f_1}
      \ar[ur, "w_\Delta", "\simeq"']
      & & \Gamma_1 \ar[crossing over, twoheadleftarrow]{uu}
      \ar[ur, "\simeq"', "w_\Gamma"]
    \end{tikzcd}
  \end{equation*}
  one has $\gamma = w_\Sigma(w_{AB}',w_A',w_\Gamma)$ for
  $w_A' \colon \Delta_1.f_1^*A_1 \tosimeq \Delta_2.f_2^*A_2$ as
  $w_A' \coloneqq \gamma(w_A,w_\Gamma,w_\Delta;f_1,f_2,\refl)$ and
  \begin{align*}
    w_{AB}' \coloneqq \gamma(w_{AB},w_A,w_A';f_1.A_1,f_2.A_2,\refl) \colon
    \Delta_1.f_1^*A_1.f_1^*B_1 \tosimeq \Delta_2.f_2^*A_2.f_2^*B_2
  \end{align*}
\end{lemma}
\begin{proof}
  This follows from the fact that $w_\Sigma$ is stable under substitution.
  By definition, $w_{\Gamma.\Sigma}(w_{AB},w_A,w_\Gamma)$ is given as
  \begin{align*}
    (\vec{x}_1:\Gamma_1).\Sigma(a_1:A_1(\vec{x}_1)).B_1(\vec{x}_1,a_1)
    \vdash
    t_\Gamma(\vec{x}_1), \sigmarec_{d(\vec{x}_1)}
    :
    \Gamma_2.\Sigma(A_2,B_2)
  \end{align*}
  where $d(\vec{x}_1, a_1, b_1) = \pair_2(t_A(\vec{x}_1, a_1),
  t_B(\vec{x}_1,a_1,b_1))$, so $\gamma \colon \Delta.\Sigma(f_1^*A_1,f_2^*B_1)$
  is given as
  \begin{align*}
    (\vec{y}_1:\Delta_1).(s_1:\Sigma(f_1^*A_1,f_1^*B_1))
    \vdash
    w_\Delta(\vec{y_1}),
    \sigmarec_{d(f_1\vec{x}_1)}
    :
    \Delta_2.\Sigma(f_2^*A_2,f_2^*B_2)
  \end{align*}
  which is seen to be equals to $w_\Sigma(w_{AB}',w_A',w_\Gamma)$ by explicit
  computation.
\end{proof}

\begin{definition}[Extensional Kernel]\label{def:w-ett}
  Let $\McW_\ETT$ be the smallest subset of maps of $\McW$ containing all the
  identities such that:
  \begin{itemize}
    \item It contains maps homotopic to
    $\gamma(w_A, w_\Gamma, w_\Delta; f_1, f_2, H)$ of
    \Cref{def:glue} whenever
    $w_A,w_\Gamma, w_\Delta \in \McW_\ETT$ and
    $H \colon w_\Gamma \cdot f_1 \simeq f_2 \cdot t_\Delta$.
    \item It contains maps homotopic to $w_\Pi(w_{AB},w_A,w_\Gamma)$ of
    \Cref{def:id-pi} whenever $w_{AB},w_A,w_\Gamma \in \McW_\ETT$.
    \item It contains maps homotopic to $w_\Id(w_A,w_\Gamma)$ of
    \Cref{def:id-id} whenever $w_A,w_\Gamma \in \McW_\ETT$.
    \item It contains maps homotopic to $w_\Sigma(w_{AB},w_A,w_\Gamma)$ of
    \Cref{def:id-sigma} whenever $w_{AB},w_A,w_\Gamma \in \McW_\ETT$.
  \end{itemize}
\end{definition}

% \begin{lemma}\label{lem:w-ett-unique}
%   %
%   If $w,w' : \Gamma_1 \tosimeq \Gamma_2 \in \McW_\ETT$ then $w \simeq w'$.
%   %
% \end{lemma}
% %
% \begin{proof}
%   %
%   We apply induction on $w \in \McW_\ETT$ and $w \in \McW_\ETT$.

%   Suppose that $w = $
%   %
% \end{proof}

% %
% It is straightforwards to observe \Cref{itm:grading,itm:proj}.
% %
% To establish \Cref{itm:uip}, we need the following series of results.

\begin{lemma}\label{lem:w-ett-free-ty}
  If $w_A \colon \Delta_1.A_1 \tosimeq \Delta_2.A_2 \in \McW_\ETT$ where $A_i$
  for $i=1,2$ are pullbacks of freely adjoined types by cofibrancy (i.e.
  $A_i \neq \Pi(C,D), \Id_C, \Sigma(C,D)$ for any types $C,D$) then there exists
  a unique context $\Theta$ along with a type $A \in \Ty\Theta$ such that $A_i$
  is the pullback of a unique map $f_i \colon \Gamma_i \to \Theta$.
  Moreover,
  $w_A \simeq \gamma(\id,\id,w_\Delta;f_1,f_2,H)$ for some (unique)
  $w_\Delta \colon \Delta_1 \tosimeq \Delta_2 \in \McW_\ETT$ and
  $H \colon f_2 \cdot w_\Delta \simeq f_1$.
  \begin{equation*}
    \begin{tikzcd}[cramped, sep=small]
      & \Delta_2.f_2^*A \ar{rr} \ar[two heads]{dd}
      \ar[phantom]{rrdd}[pos=0em]{\lrcorner}
      & & \Theta.A \\
      \Delta_1.f_1^*A
      \ar[rrru, phantom, "\simeq"{pos=0.4}]
      \ar[dashed]{ur}{w_A=\gamma}[description, inner sep=1]{\simeq}
      \ar[crossing over]{rr}
      \ar[phantom]{rrdd}[pos=0em]{\lrcorner}
      \ar[two heads]{dd}
      & & \Theta.A \ar[]{ur}[description, inner sep=1]{=}  \\
      & \Delta_2 \ar{rr}[near start]{f_2}
      & & \Theta \ar[crossing over, twoheadleftarrow]{uu} \\
      \Delta_1 \ar{rr}[swap]{f_1}
      \ar[rrru, phantom, "\simeq"]
      \ar[]{ur}[description, inner sep=1]{\simeq}
      & & \Theta \ar[crossing over, twoheadleftarrow]{uu}
      \ar[]{ur}[swap]{=}
    \end{tikzcd}
  \end{equation*}
\end{lemma}
\begin{proof}
  We apply induction on $w_A \in \McW_\ETT$.
  By \Cref{prop:cofib-ty-class}, $w_A \not\simeq w_\Pi,w_\Id,w_\Sigma$ of
  \Cref{def:id-pi,def:id-id,def:id-sigma} respectively.
  Suppose $w_A \simeq \gamma(w_A',w_\Gamma,w_\Delta; g_1,g_2,H)$ of
  \Cref{def:glue} for some $w_A' \colon \Gamma_1.A_1' \tosimeq \Gamma_2.A_2'$ over
  $w_\Gamma \colon \Gamma_1 \tosimeq \Gamma_2$ and maps
  $g_i \colon \Delta_i \to \Gamma_i$ where $i=1,2$.
  Because the outer-most constructor of $A_i = g_i^*A_i'$ is not
  $\Pi,\Id,\Sigma$, neither is the outer-most constructor of $A_i'$.
  So induction applies and we may express $w_A'$ as
  $w_A' \simeq \gamma'(\id,\id,w_\Gamma;f_1,f_2,H)$ for some unique
  $f_i \colon \Gamma_i \to \Theta_A$, where $i=1,2$ and
  $H' \colon f_2 \cdot w_\gamma \simeq f_1$.
  \begin{equation*}
    \begin{tikzcd}[cramped, sep=small]
      &
      \Delta_2.g_2^*f_2^*A \ar[rr] \ar[dd, two heads]
      \ar[phantom]{rrdd}[pos=0em]{\lrcorner}
      &
      & \Gamma_2.f_2^*A \ar{rr} \ar[two heads]{dd}
      \ar[phantom]{rrdd}[pos=0em]{\lrcorner}
      & & \Theta.A \\
      \Delta_1.g_1^*f_1^*A \ar[rr, crossing over] \ar[ru, "\simeq"{description}]
      \ar[dd, two heads] \ar[rrru, phantom, "\simeq"{pos=0.4}]
      \ar[phantom]{rrdd}[pos=0em]{\lrcorner}
      &
      &
      \Gamma_1.f_1^*A
      \ar[rrru, phantom, "\simeq"{pos=0.4}]
      \ar[dashed]{ur}{}[description, inner sep=1]{\simeq}
      \ar[crossing over]{rr}
      \ar[phantom]{rrdd}[pos=0em]{\lrcorner}
      & & \Theta.A \ar[]{ur}[description, inner sep=1]{=}  \\
      &
      \Delta_2 \ar[rr, "g_2"{pos=0.3}]
      &
      & \Gamma_2 \ar{rr}[near start]{f_2}
      & & \Theta \ar[crossing over, twoheadleftarrow]{uu} \\
      \Delta_1 \ar[rr, "g_1"'] \ar[ru, "\simeq"{description}]
      \ar[rrru, phantom, "\simeq"{pos=0.4}]
      &
      &
      \Gamma_1 \ar{rr}[swap]{f_1} \ar[uu, twoheadleftarrow, crossing over]
      \ar[rrru, phantom, "\simeq"]
      \ar[]{ur}[description, inner sep=1]{\simeq}
      & & \Theta \ar[crossing over, twoheadleftarrow]{uu}
      \ar[]{ur}[swap]{=}
    \end{tikzcd}
  \end{equation*}
  The result now follows because
  $w_A \simeq \gamma(\id,\id,w_\Gamma;g_1f_1,g_2f_2,H'H)$.
\end{proof}

\begin{lemma}\label{lem:w-ett-canon}
  Up to homotopy, we may write maps in $\McW_\ETT$ canonically:
  \begin{itemize}
    \item If
    $w \colon \Gamma_1.\Pi(A_1,B_1) \tosimeq \Gamma_2.\Pi(A_2,B_2) \in \McW_\ETT$
    then there exists $w_{AB} \colon \Gamma_1.A_1.B_1 \tosimeq \Gamma_2.A_2.B_2$
    over $w_A \colon \Gamma_1.A_1 \tosimeq \Gamma_2.A_2$ over
    $w_\Gamma \colon \Gamma_1 \tosimeq \Gamma_2$ such that
    $w \simeq w_\Pi(w_{AB},w_A,w_\Gamma)$.
    \item If
    $w \colon \Gamma_1.A_1.A_1.\Id_{A_1} \tosimeq \Gamma_2.A_1.A_1.\Id_{A_1} \in
    \McW_\ETT$ then there exists $w_A \colon \Gamma_1.A_1 \tosimeq \Gamma_2.A_2$
    over $w_\Gamma \colon \Gamma_1 \tosimeq \Gamma_2$ such that
    $w \simeq w_\Id(w_A,w_\Gamma)$.
    \item If
    $w \colon \Gamma_1.\Sigma(A_1,B_1) \tosimeq \Gamma_2.\Sigma(A_2,B_2) \in \McW_\ETT$
    then there exists $w_{AB} \colon \Gamma_1.A_1.B_1 \tosimeq \Gamma_2.A_2.B_2$
    over $w_A \colon \Gamma_1.A_1 \tosimeq \Gamma_2.A_2$ over
    $w_\Gamma \colon \Gamma_1 \tosimeq \Gamma_2$ such that
    $w \simeq w_\Sigma(w_{AB},w_A,w_\Gamma)$.
  \end{itemize}
\end{lemma}
\begin{proof}
  In each of the cases, we apply induction on $w \in \McW_\ETT$.
  For illustration, suppose
  $w \colon \Gamma_1.\Pi(A_1,B_1) \tosimeq \Gamma_2.\Pi(A_2,B_2)$.
  Then, by \Cref{prop:cofib-ty-class}, $w \neq w_\Id,w_\Sigma$ of
  \Cref{def:id-id,def:id-sigma} respectively.
  If $w \simeq w_\Sigma(w_{AB},w_A,w_\Gamma)$ as in \Cref{def:id-pi} then we are
  done.
  It remains to consider the case
  $w \simeq \gamma(w',w_\Delta,w_\Gamma;f_1,f_2,H) \colon \Gamma_1.\Pi(A_1,B_1) \tosimeq
  \Gamma_2.\Pi(A_2,B_2)$ as in \Cref{def:glue} for
  $w' \colon \Delta_1.C_1 \tosimeq \Delta_2.C_2$ and
  $f_i \colon \Gamma_i \to \Delta_i$ such that
  $H \colon w_\Delta \cdot f_1 \simeq f_2 \cdot w_\Gamma$.
  By \Cref{prop:cofib-ty-class}, there exists $A_i',B_i'$ such that
  $C_i = \Pi(A_i',B_i')$ so that $A_i = f_i^*A_i'$.
  Furthermore, by $\UIP$, one has that $H \simeq \refl$.
  Therefore, by \Cref{lem:glue},
  $w \simeq \eta(w',w_\Delta,w_\Gamma;f_1,f_2,H) \simeq
  \eta(w',w_\Delta,w_\Gamma;f_1,f_2,\refl)$.
  But induction applies on
  $w' \colon \Delta_1.C_1 = \Delta_1.\Pi(A_1',B_1') \tosimeq
  \Delta_2.\Pi(A_2',B_2') = \Delta_2.C_2$ so that
  $w' \simeq w_\Pi(w_{AB}',w_A',w_\Delta')$.
  At this point, \Cref{lem:id-pi-gamma} to conclude the result.

  The cases for $\Id$ and $\Sigma$ are identical, except instead of using
  \Cref{lem:id-pi-gamma} one uses \Cref{lem:id-id-gamma} and
  \Cref{lem:id-sigma-gamma} respectively.
\end{proof}

\begin{proposition}\label{prop:w-ett-thin-subgpd}
  $\McW_\ETT$ is a wide subcategory of $\McW$ closed under homotopy inverses
  such that if $w,w' \colon \Theta_1 \tosimeq \Theta_2 \in \McW_\ETT$ then
  $w \simeq w'$.
\end{proposition}
\begin{proof}
  To observe it is closed under homotopic inverses, it suffices that each of the
  constructions of $\gamma, w_\Pi,w_\Id,w_\Sigma$ are stable under homotopy inverses.
  For example, if That is, if
  $\overline{w_{AB}},\overline{w_A},\overline{w_\Gamma}$ are respectively
  homotopy inverses of $w_{AB},w_A, w_\Gamma$ then
  $w_\Sigma(\overline{w_{AB}},\overline{w_A},\overline{w_\Gamma})$ is a homotopy
  inverse of $w_\Sigma(w_{AB},w_A,w_\Gamma)$.
  Likewise, the homotopy inverses of $\gamma, w_\Pi,w_\Id$ are given as
  repeating the same construction using the homotopy inverses of their
  respective arguments.

  By \Cref{prop:cofib-ty-class}, there is a well-founded order $<$ on $\ob\bC$ such that
  \begin{align*}
    \Theta.A.B < \Theta.\Pi(A,B), \Theta.\Sigma(A,B)
    &&
       \Theta.A < \Theta.A.A < \Theta.A.A.\Id_A
    &&
       \Theta < \Theta.A
  \end{align*}
  We proceed by well-founded lexicographic induction on
  $(\Theta_1,\Theta_2,\Theta_3)$ to show that if
  $w \colon \Theta_1 \to \Theta_2 \in \McW_\ETT$ and
  $w' \colon \Theta_2 \to \Theta_3 \in \McW_\ETT$ then $w'w \in \McW_\ETT$.
  Again by \Cref{prop:cofib-ty-class}, one may consider the following mutually
  exclusive cases for the domain and codomain of $w$:
  \begin{itemize}
    \item
    \emph{$w \colon
      \Delta_1.A_1 \tosimeq \Delta_2.A_2$ as from \Cref{def:glue} where
      the outermost type-constructor of $A_i$ is not $\Id,\Pi,\Sigma$.}
    Then, by cofibrancy, $w' \colon \Delta_2.A_2 \tosimeq \Delta_3.A_3$ for
    some $A_3$ whose outermost type constructor is also not $\Id,\Pi,\Sigma$.
    Using \Cref{lem:w-ett-free-ty} on $w$ and $w'$ respectively, one may find
    $\Theta_A \in \bC$ and $A \in \Ty\Theta$ such that there are unique maps
    $g_i \colon \Delta_i \to \Theta_A$ where $i=1,2,3$ such that $A_i = f_i^*A$.
    Moreover, $w \simeq \gamma(\id,\id,w_\Delta;f_1,f_2,H)$ for some
    $w_\Delta \colon \Delta_1 \tosimeq \Delta_2$ and
    $H \colon f_2 \cdot w_\Delta \simeq f_1$.
    Likewise,
    $w' \simeq \gamma(\id,\id,w_\Delta';f_2,f_3,H')$ for some
    $w_\Delta' \colon \Delta_2 \tosimeq \Delta_3$ and
    $H \colon f_3 \cdot w_\Delta' \simeq f_2$.
    Therefore, $w'w \simeq \gamma(\id,\id,w_\Delta'w_\Delta;f_1,f_3,H'H)$.
    By induction, $w_\Delta'w_\Delta \in \McW_\ETT$, so $w'w \in \McW_\ETT$.
    \item \emph{$w \colon \Gamma_1.\Pi(A_1,B_1) \tosimeq \Gamma_2.\Pi(A_2,B_2)$,
      or
      $w \colon \Gamma_1.A_1.A_1.\Id_{A_1} \tosimeq \Gamma_2.A_2.A_2.\Id_{A_2}$,
      or $w \colon \Gamma_1.\Sigma(A_1,B_1) \tosimeq \Gamma_2.\Sigma(A_2,B_2)$.}
    These cases are all identical.
    For illustration, suppose
    $w \colon \Gamma_1.\Pi(A_1,B_1) \tosimeq \Gamma_2.\Pi(A_2,B_2)$.
    Then, by \Cref{prop:cofib-ty-class},
    $w' \colon \Gamma_2.\Pi(A_2,B_2) \tosimeq \Gamma_3.\Pi(A_3,B_3)$.
    Using \Cref{lem:w-ett-canon},
    $w \simeq w_\Pi(w_{AB},w_A,w_\Gamma)$ for
    $w_{AB} \colon \Gamma_1.A_1.B_1 \tosimeq \Gamma_2.A_2.B_2$ over
    $w_A \colon \Gamma_1.A_1 \tosimeq \Gamma_2.A_2$ over
    $w_\Gamma \colon \Gamma_1 \tosimeq \Gamma_2$.
    Also by \Cref{lem:w-ett-canon},
    $w' \simeq w_\Pi(w_{AB}',w_A',w_\Gamma')$ for
    $w_{AB}' \colon \Gamma_2.A_2.B_2 \tosimeq \Gamma_3.A_3.B_3$ over
    $w_A' \colon \Gamma_2.A_2 \tosimeq \Gamma_3.A_3$ over
    $w_\Gamma' \colon \Gamma_2 \tosimeq \Gamma_3$.
    Therefore,
    \begin{align*}
      w'w \simeq w_\Pi(w_{AB}'w_{AB}, w_A'w_A, w_\Gamma'w_\Gamma)
    \end{align*}
    By induction, $w_{AB}'w_{AB}, w_A'w_A, w_\Gamma'w_\Gamma \in \McW_\ETT$, so $w'w \in \McW_\ETT$.
  \end{itemize}

  Finally, we proceed to show by well-founded lexicographic induction on
  $(\Theta_1,\Theta_2)$ to show that that if
  $w,w' \colon \Theta_1 \tosimeq \Theta_2 \in \McW_\ETT$ then $w \simeq w'$.
  Once again by cofibrancy, one may consider the following (mutually exclusive)
  cases for the domain and codomain of $w$:
  \begin{itemize}
    \item
    \emph{$w \colon \Delta_1.A_1 \tosimeq \Delta_2.A_2$ as from \Cref{def:glue} where
      the outermost type-constructor of $A_i$ is not $\Id,\Pi,\Sigma$.}
    Then, by \Cref{lem:w-ett-free-ty}, there exists $\Theta_A$ and $A \in \Ty\Theta_A$
    such that there are unique maps $f_i \colon \Delta_i \to \Theta_A$ such that
    $\Delta_i.A_i = \Delta_i.f_i^*\Theta_A$.
    Moreover, $w \simeq \gamma(\id,\id,u_\Delta; f_1,f_2,H)$ for
    $u_\Delta \colon \Delta_1 \tosimeq \Delta_2 \in \McW$ and
    $H \colon f_2 \cdot u_\Delta \simeq f_1$.
    Also by \Cref{lem:w-ett-free-ty},
    $w' \simeq \gamma(\id,\id,u_\Delta'; f_1,f_2,H')$ for
    $u_\Delta' \colon \Delta_1 \tosimeq \Delta_2 \in \McW$ and
    $H' \colon f_2 \cdot u_\Delta \simeq f_1$.
    By induction, $u_\Delta \simeq u_\Delta'$ and so
    $H \simeq H'$ by $\UIP$.
    Hence, by \Cref{lem:glue}, $w \simeq w'$.
    \item \emph{$w \colon \Gamma_1.\Pi(A_1,B_1) \tosimeq \Gamma_2.\Pi(A_2,B_2)$,
      or
      $w \colon \Gamma_1.A_1.A_1.\Id_{A_1} \tosimeq \Gamma_2.A_2.A_2.\Id_{A_2}$,
      or $w \colon \Gamma_1.\Sigma(A_1,B_1) \tosimeq \Gamma_2.\Sigma(A_2,B_2)$.}
    These cases are all identical.
    For illustration, suppose
    $w \colon \Gamma_1.\Pi(A_1,B_1) \tosimeq \Gamma_2.\Pi(A_2,B_2)$.
    Then, by \Cref{lem:w-ett-canon},
    $w \simeq w_\Pi(w_{AB},w_A,w_\Gamma)$ for
    $w_{AB} \colon \Gamma_1.A_1.B_1 \tosimeq \Gamma_2.A_2.B_2$ over
    $w_A \colon \Gamma_1.A_1 \tosimeq \Gamma_2.A_2$ over
    $w_\Gamma \colon \Gamma_1 \tosimeq \Gamma_2$.
    Also by \Cref{lem:w-ett-canon},
    $w' \simeq w_\Pi(w_{AB}',w_A',w_\Gamma')$ for
    $w_{AB}' \colon \Gamma_1.A_1.B_1 \tosimeq \Gamma_2.A_2.B_2$ over
    $w_A' \colon \Gamma_1.A_1 \tosimeq \Gamma_2.A_2$ over
    $w_\Gamma' \colon \Gamma_1 \tosimeq \Gamma_2$.
    By induction, $w_{AB} \simeq w_{AB}'$ and $w_A' \simeq w_A$ and
    $w_\Gamma' \simeq w_\Gamma$, so $w \simeq w'$.
    \qedhere
  \end{itemize}
\end{proof}

\begin{lemma}\label{lem:w-ett-grading-proj} $ $
  \begin{enumerate}
    \item \label{itm:grading} If
    $\Gamma_1 \tosimeq \Gamma_2 \in \McW_\ETT$ then $\Gamma_1$ and
    $\Gamma_2$ are of the same length.
    \item \label{itm:proj} If
    $w_\Gamma \colon \Gamma_1 \tosimeq \Gamma_2 \in \McW_\ETT$ then there is
    $w_{\ft{\Gamma}} \colon \ft{\Gamma_1} \tosimeq \ft{\Gamma_2} \in
    \McW_\ETT$ such that
    $(\pi,\pi) \colon w_\Gamma \to w_{\ft{\Gamma}} \in \bC^\to$.
  \end{enumerate}
\end{lemma}
\begin{proof}
  Easy induction on $w \colon \Gamma_1 \tosimeq \Gamma_2 \in \McW_\ETT$.
\end{proof}

As mentioned previously, our goal is now to construct a category
$\vbr{\bC}_\ETT$ from $\bC$ such that all maps in $\McW_\ETT$ are formally
identified with identities.
To do so, and in view of \Cref{lem:w-ett-grading-proj}\ref{itm:grading}, there
is an equivalence relation $\sim$ on each $\ob_n{\bC}$ such that
$\Gamma_1 \sim \Gamma_2$ exactly when there is
$\Gamma_1 \tosimeq \Gamma_2 \in \McW_\ETT$.
We would like to take as objects of $\vbr{\bC}_\ETT$ equivalence classes of
objects in $\ob\bC$ under the $\sim$ relation.
Hence, we need to define what is a map in $\vbr{\bC}_\ETT$ between two such
equivalence classes:

\begin{definition}\label{def:w-ett-maps}
  Let $\set{\Gamma_i}_{i \in I}$ and $\set{\Delta_j}_{j \in J}$ be
  equivalence classes in $(\ob\bC)/\sim$.
  Define a relation $\approx$ on
  $\coprod_{i,j} \bC(\Gamma_i,\Delta_j)$ such that $f_1 \approx f_2$
  for $f_k \colon \Gamma_{i_k} \to \Delta_{j_k}$ exactly when there is
  some
  $w_\Gamma \colon \Gamma_{i_1} \tosimeq \Gamma_{i_2} \in \McW_\ETT$ and
  $w_\Delta \colon \Delta_{j_1} \tosimeq \Delta_{j_2} \in \McW_\ETT$
  such that
  \begin{equation*}
    \begin{tikzcd}
      \Gamma_{i_1} \ar{r}{f_1} \ar{d}[swap]{\McW_\ETT \ni w_\Gamma}{\simeq} \ar[phantom]{rd}{\simeq} &
      \Delta_{j_1} \ar{d}[swap]{\simeq}{w_\Delta \in \McW_\ETT} \\
      \Gamma_{i_2} \ar{r}[swap]{f_2} &
      \Delta_{j_2}
    \end{tikzcd}
  \end{equation*}
  commutes up to homotopy.
\end{definition}
Because $\McW_\ETT$ is a wide subcategory closed under homotopy
inverses, $\approx$ is an equivalence relation and $w \approx \id$
whenever $w \colon \dom{w} \to \cod{w} \in \McW_\ETT$.
Moreover, by \Cref{lem:w-ett-grading-proj}\ref{itm:proj}, whenever
$\Gamma_1 \tosimeq \Gamma_2 \in \McW_\ETT$ we have that
$\pi_{\Gamma_1} \approx \pi_{\Gamma_2}$.
Denote by $[-]$ equivalence classes $(\ob\bC)/\sim$ and
$(\mor\bC)/\approx$.

Thanks to \Cref{prop:w-ett-thin-subgpd}, we have the
following equivalent characterisation of $\approx$.

\begin{lemma}\label{lem:approx-char}
  Let $\set{\Gamma_i}_{i \in I}$ and $\set{\Delta_j}_{j \in J}$ be
  equivalence classes in $(\ob\bC)/\sim$ and take
  $f_k \colon \Gamma_{i_k} \to \Delta_{j_k}$ for $k=1,2$.
  Suppose $w_\Gamma \colon \Gamma_{i_1} \tosimeq \Gamma_{i_2} \in \McW_\ETT$ and
  $w_\Delta \colon \Delta_{j_1} \tosimeq \Delta_{j_2} \in \McW_\ETT$.
  Then,
  \begin{equation*}
    f_1 \approx f_2 \text{ if and only if }
      \begin{tikzcd}
        \Gamma_{i_1}
        \ar{r}{f_1} \ar{d}[swap]{\McW_\ETT \ni w_\Gamma}{\simeq} \ar[phantom]{rd}{\simeq} &
        \Delta_{j_1}
        \ar{d}[swap]{\simeq}{w_\Delta \in \McW_\ETT} \\
        \Gamma_{i_2}
        \ar{r}[swap]{f_2} &
        \Delta_{j_2}
      \end{tikzcd}
  \end{equation*}
\end{lemma}
\begin{proof}
  The $\Leftarrow$ direction is by definition while the $\Rightarrow$
  direction makes crucial use of
  \Cref{prop:w-ett-thin-subgpd}.
\end{proof}

\begin{lemma}\label{lem:ext-ker}
  Suppose $\bE \in \CxlCat_\ETT$ and
  $F \colon \bC \to \abs{\bE} \in \CxlCat_{\ITT+\UIP}$.
  If $w \colon \Gamma_1 \tosimeq \Gamma_2 \in \McW_\ETT$ then
  $Fw$ is the identity.
\end{lemma}
\begin{proof}
  Apply induction on $w$.
  If $w$ or its homotopy inverse is homotopic to some other map
  which $F$ collapses to the identity then by reflection, $w$ is
  also the identity.
  Likewise, if $w$ is itself the identity or the composition of maps
  which $F$ maps to the identity, then $Fw$ must also be the
  identity.
  So we just consider the following cases:
  \begin{itemize}
    \item \emph{$w \simeq \gamma(w_\Gamma, w_{\ft \Gamma}, w_\Delta; f_1,f_2, H)$ as
      in \Cref{def:glue} where $w_\Gamma,w_{\ft\Gamma}, w_\Delta \in \McW_\ETT$
      and $H \colon w_\Gamma \cdot f_1 \simeq f_2 \cdot w_\Delta$.}
    Inductively assume $w_\gamma,w_{\ft\Gamma},w_\Delta$ are all
    mapped to the identities.
    Then by uniqueness of the map in \Cref{def:glue} combined with equality
    reflection in the extensional setting, again $Fw = \id$.
    \item \emph{$w \simeq w_\Pi(w_{AB},w_A,w_\Gamma)$ or $w_\Id(w_A,w_\Gamma)$ or
      $w_\Sigma(w_{AB},w_A,w_\Gamma)$ as in \Cref{def:id-pi} or \Cref{def:id-id}
      or \Cref{def:id-sigma} respectively, where
      $w_\Gamma,w_A,w_{AB} \in \McW_\ETT$.}
    These cases are similar.
    Consider, for example, the $\Pi$-types case.
    By induction, $Fw_\Gamma, Fw_A, Fw_{AB}$ are all identities.
    And so
    $F\Pi(A_1,B_1) = \Pi(FA_1,FB_1) = \Pi(FA_2,FB_2) =
    F\Pi(A_2,B_2)$.
    Because $F$ preserves weak equivalences and maps into an
    extensional type theory, the result follows by reflection.
    \qedhere
  \end{itemize}
\end{proof}

%%% Local Variables:
%%% mode: latex
%%% TeX-master: "./main.tex"
%%% TeX-engine: xetex
%%% End:

%% file: c-ett.tex
\section{Freely Generated Extensional Type Theory}\label{sec:c-ett}
Having defined the wide subcategory $\McW_\ETT$, we can now construct
the \emph{freely generated extensional type theory} $\vbr{\bC}_\ETT$
obtained by formally collapsing all maps in $\McW_\ETT$ to identities.
In particular:
\begin{itemize}
  \item In \Cref{prop:free-ett-cxl-cat}, we show that $\vbr{\bC}_\ETT$ defined
  in \Cref{def:free-ett} has the structure of a contextual category.
  Hence, $\vbr{\bC}_\ETT$ in \Cref{def:free-ett} corresponds to $\bQ$ in
  \cite[Definition 3.2.14]{hofmann:extensional-constructs} and our
  \Cref{prop:free-ett-cxl-cat} corresponds to \cite[Proposition
  3.2.15]{hofmann:extensional-constructs}.
  \item In \Cref{lem:strict-lifting}, we show that $\bC \to \vbr{\bC}_\ETT$ is a
  trivial fibration.
  This corresponds to \cite[Proposition 3.2.16]{hofmann:extensional-constructs}.
\end{itemize}

The main result of this section is that we verify that $\vbr{\bC}_\ETT$ has all
the structure of a contextual category.

main results in this
section correspond to \cite[Definition 3.2.14]{hofmann:extensional-constructs} and results in this

\begin{definition}[Freely Generated Extensional Type Theory]\label{def:free-ett}
  Define $\vbr{\bC}_\ETT$ as the category:
  \begin{itemize}
    \item \emph{Objects.} $\ob{\vbr{\bC}_\ETT} \coloneqq (\ob\bC)/\sim$.
    \item \emph{Maps.} $\mor{\vbr{\bC}_\ETT} \coloneqq (\mor\bC)/\approx$.
    \item \emph{Identities.} $\id_{[\Gamma]} \coloneqq [\id_\Gamma]$.
    \item \emph{Composition.} $[f] \cdot [g] \coloneqq [fwg]$ where
    $w \colon \cod{g} \tosimeq \dom{f} \in \McW_\ETT$ is chosen arbitrarily.
  \end{itemize}
  We also equip $\vbr{\bC}_\ETT$ with the following data
  required for a contextual category:
  \begin{itemize}
    \item \emph{Grading.} $\ob{\vbr{\bC}_\ETT} \coloneqq
    \coprod_{n \in \bN} \set{[\Gamma] \mid \Gamma \in \ob_n{\bC}}$.
    \item \emph{Empty context.} $1 \coloneqq [1]$.
    \item \emph{Context projections.} For each
    $[\Gamma] \in \ob_{n+1}{\vbr{\bC}_\ETT}$, put
    $\ft{[\Gamma]} \coloneqq [\ft\Gamma]$ and
    $\pi_{[\Gamma]} \coloneqq [\pi_\Gamma] \colon [\Gamma] \to
    [\ft\Gamma]$.
    \item \emph{Substitutions.} For each
    $[\Gamma] \in \ob_{n+1}{\vbr{\bC}_\ETT}$ and
    $[f] \colon [\Delta] \to \ft{[\Gamma]}$, put
    $[f]^*[\Gamma] \coloneqq [w_\Delta^*f^*w_\Gamma^*\Gamma]$
    and $[f].[\Gamma] \coloneqq [(w_\Delta f w_\Gamma).\Gamma]$
    where:
    \begin{itemize}
      \item $\Delta_0$ is any representative of $[\Delta]$.
      \item $f \colon \Delta_0 \to \Theta$ is any representative of $[f]$ for
      $\Theta$ any representative of $\ft[\Gamma] = [\ft\Gamma]$ (this equality
      is by \Cref{lem:free-ett-basic-tt}).
      \item $w_\Gamma \colon \Theta \tosimeq \ft\Gamma$ and
      $w_\Delta \colon \Delta \tosimeq \Delta_0$ are any maps in $\McW_\ETT$.
    \end{itemize}
    That is, we pullback along the map
    \begin{equation*}
      \begin{tikzcd}
        \Delta \ar{r}[swap]{\simeq}{w_\Delta}
        & \Delta_0 \ar{r}{f}
        & \Theta \ar{r}[swap]{\simeq}{w_\Gamma}
        & \ft\Gamma
      \end{tikzcd}
    \end{equation*}
  \end{itemize}
\end{definition}

\begin{proposition}\label{prop:free-ett-cxl-cat}
  The category $\vbr{\bC}_\ETT$ equipped with the data in \Cref{def:free-ett} forms
  a contextual category and the assignment
  $[-] \colon \bC \to \vbr{\bC}_\ETT$ sending objects and maps to
  their equivalence classes is a map of contextual categories.
\end{proposition}
The proof of this result depends on the following series of lemmas.

\begin{lemma}\label{lem:free-ett-cat}
  $\vbr{\bC}_\ETT$ is a category and the assignment
  $[-] \colon \bC \to \vbr{\bC}_\ETT$ of objects and arrows into their
  equivalence classes is a full and surjective-on-objects functor
  $[-] \colon \bC \to \vbr{\bC}_\ETT$ that sends all maps in $\McW_\ETT$ to
  identities.
\end{lemma}
\begin{proof}
  Composition is well-defined because for all
  $[\Gamma \xrightarrow{f} \Delta_1] = [\Gamma' \xrightarrow{f'} \Delta_1']$ and
  $[\Delta_2 \xrightarrow{g} Z] = [\Delta_2' \xrightarrow{g'} Z']$ with
  $[\Delta_1] = [\Delta_2]$ we have the following diagram in $\Ho\bC$, where all
  equivalences are in $\McW_\ETT$ and the middle square commutes by
  \Cref{prop:w-ett-thin-subgpd}.
  \begin{equation*}
    \begin{tikzcd}
      \Gamma \ar{r}{f} \ar{d}[swap]{\simeq} & \Delta_1 \ar{r}{\simeq} \ar{d}[description]{\simeq}
      & \Delta_2 \ar{r}{g} \ar{d}[description]{\simeq} & Z \ar{d}{\simeq} \\
      \Gamma' \ar{r}[swap]{f'} & \Delta_1' \ar{r}[swap]{\simeq} & \Delta_2' \ar{r}[swap]{g'} & Z'
    \end{tikzcd}
  \end{equation*}
  Thus associativity and identity follow from those of $\bC$.
  And $[-] \colon \bC \to \vbr{\bC}_\ETT$ is functorial because
  $[\id_\Gamma] = \id_{\Gamma}$ and
  $[fg] = [f \cdot \id \cdot g] = [f] \cdot [g]$ because
  $\id \in \McW_\ETT$.
  Finally, if $w \in \McW_\ETT$ then $w \approx \id$ and so
  $[w] = [\id] = \id$.
\end{proof}

\begin{lemma}\label{lem:free-ett-basic-tt}
  $\vbr{\bC}_\ETT$ inherits the terminal object, the grading of
  objects and the projection maps from $\bC$ by passing into the
  quotient $[-] \colon \bC \to \vbr{\bC}_\ETT$.
\end{lemma}
\begin{proof}
  By \Cref{lem:w-ett-grading-proj}\ref{itm:grading}, we have
  $\ob{\vbr{\bC}_\ETT} = (\ob\bC)/{\sim} = (\coprod_n \ob_n\bC)/{\sim}
  = \coprod_n \ob_n\bC/{\sim}$, so the grading of objects passes into
  the quotient.
  Because the terminal object $1 \in \bC$ is the unique object in
  $\ob_0\bC$, it follows that $[1]$ is the unique object in
  $\ob_0{\vbr{\bC}_\ETT}$.
  By \Cref{lem:w-ett-grading-proj}\ref{itm:proj}, putting
  $\ft[\Gamma] \coloneqq [\ft\Gamma]$ and
  $\pi_{[\Gamma]} \coloneqq [\pi_\Gamma]$ is well-defined.
  Terminality of $[1]$ follows from \Cref{lem:approx-char} and the
  fact that all objects in $\bC$ are fibrant.
\end{proof}

\begin{lemma}\label{lem:free-ett-subst}
  For each $[\Gamma]$ and $[f] \colon [\Delta] \to \ft[\Gamma]$,
  the definitions of $[f]^*[\Gamma]$ and
  $[f].[\Gamma]$ from \Cref{def:free-ett} is well-defined.
  Moreover, we have in $\vbr{\bC}_\ETT$ that
  \begin{equation*}\label{diag:subst}\tag{\textsc{subst}}
    \begin{tikzcd}[]
      {[f]^*[\Gamma]} \ar{d}[swap]{\pi} \ar{r}{{[f].[\Gamma]}}
      & {[\Gamma]} \ar{d}{\pi} \\
      {[\Delta]} \ar{r}[swap]{{[f]}} & {\ft[\Gamma]}
    \end{tikzcd}
  \end{equation*}
\end{lemma}
The proof will be preceded by the following lemma.
  \begin{lemma}\label{lem:free-ett-subst-id}
    If $w \colon \Theta \tosimeq \ft{\Gamma} \in \McW_\ETT$ then
    $w.\Gamma \colon w^*\Gamma \tosimeq \Gamma \in \McW_\ETT$ and so
    $[w^*\Gamma] = [\Gamma] \in \vbr{\bC}_\ETT$.
  \end{lemma}
  \begin{proof}
    Use \Cref{lem:glue} on
    \begin{equation*}
      \begin{tikzcd}[cramped, sep=small]
        & \Gamma \ar[equal]{rr} \ar[two heads]{dd}
        \ar[phantom]{rrdd}[pos=0em]{\lrcorner}
        & & \Gamma \\
        w^*\Gamma
        \ar[dashed]{ur}{w.\Gamma}[description, inner sep=1]{\simeq}
        \ar[crossing over]{rr}
        \ar[phantom]{rrdd}[pos=0em]{\lrcorner}
        \ar[two heads]{dd}
        & & \Gamma \ar[equal]{ur}[]{}  \\
        & \ft{\Gamma} \ar[equal]{rr}[]{}
        & & \ft{\Gamma} \ar[crossing over, twoheadleftarrow]{uu} \\
        \Theta \ar{rr}[swap]{w} \ar[]{ur}[description, inner sep=1]{w}
        & & \ft{\Gamma} \ar[crossing over, twoheadleftarrow]{uu} \ar[equal]{ur}[swap]{}
      \end{tikzcd}
    \end{equation*}
  \end{proof}
  \begin{proof}[Proof of \Cref{lem:free-ett-subst}]
  We see, by \Cref{lem:free-ett-subst-id}, that the canonical substitution pullbacks along any
  maps in $\McW_\ETT$ are always in $\McW_\ETT$.

  In particular, the above definition of $[f]^*[\Gamma]$ and
  $[f].[\Gamma]$ are independent of the choice of $w_\Gamma$ and
  $w_\Delta$ because the left and right squares below become the
  identities in $\vbr{\bC}_\ETT$:
  \begin{equation*}
    \begin{tikzcd}
      w_\Delta^*f^*w_\Gamma^*\Gamma \ar[two heads]{d}
      \ar{r}{w_\Delta.f^*w_\Gamma^*\Gamma}
      \ar[phantom]{rd}[pos=0]{\lrcorner}
      & f^*w_\Gamma^*\Gamma \ar[two heads]{d}
      \ar{r}{f.w_\Gamma^*\Gamma}
      \ar[phantom]{rd}[pos=0]{\lrcorner}
      & w_\Gamma^*\Gamma \ar[two heads]{d}
      \ar{r}{w_\Gamma.\Gamma}
      \ar[phantom]{rd}[pos=0]{\lrcorner}
      & \Gamma \ar[two heads]{d} \\
      \Delta \ar{r}{\simeq}[swap]{w_\Delta} & \Delta_0 \ar{r}[swap]{f} & \Theta \ar{r}{\simeq}[swap]{w_\Gamma} & \ft\Gamma
    \end{tikzcd}
  \end{equation*}

  We next check the construction is independent of the choice of $f$.
  Let $f' \colon \Delta_0' \to \Theta'$ be another representative of
  $[f] \colon [\Delta] \to \ft[\Gamma]$.
  Take $w_\Gamma' \colon \Theta' \tosimeq \ft{\Gamma}$ and
  $w_\Delta' \colon \Delta \tosimeq \Delta_0'$ any two maps in $\McW_\ETT$.
  By \Cref{lem:approx-char}, we have that
  $w_\Gamma f w_\Delta \simeq w_\Gamma' f' w_\Delta' \colon \Delta \to
  \ft{\Gamma}$.
  \Cref{def:glue} gives
  $\eta \colon (w_\Gamma f w_\Delta)^*\Gamma \tosimeq (w_\Gamma' f'
  w_\Delta')^*\Gamma \in \McW_\ETT$ such that
  $(\eta,\id) \colon (w_\Gamma f w_\Delta).\Gamma \to (w_\Gamma' f'
  w_\Delta').\Gamma \in \bC^\to$.
  This shows that the construction is independent of the choice of $f$
  as well.

  Finally, \Cref{diag:subst} is easily verified by direct computation
  using the property of the connecting map in $\bC$ and functoriality
  of $[-]$.
\end{proof}

\begin{lemma}\label{lem:free-ett-subst-pullback}
  The square \Cref{diag:subst} of \Cref{lem:free-ett-subst} is in fact a
  pullback.
  Hence, $\vbr{\bC}_\ETT$ inherits strictly functorial substitution
  from $\bC$.
\end{lemma}
\begin{proof}
  We check the universal property.
  Pick a representative $f \colon \Delta \to \ft\Gamma$.
  Let there be $\Theta_1 \xrightarrow{g} \Delta_1$ and
  $\Theta_2 \xrightarrow{h} \Gamma_2$ such that there exists
  $w_\Theta \colon \Theta_1 \tosimeq \Theta_2 \in \McW_\ETT$ and
  $w_\Gamma \colon \Gamma_2 \tosimeq \Gamma \in \McW_\ETT$ and
  $w_\Delta \colon \Delta_1 \tosimeq \Delta \in \McW_\ETT$.
  Further assume that
  \begin{equation*}
    \begin{tikzcd}[]
      {[\Theta_1]} \ar{d}[swap]{{[g]}} \ar{r}{{[h]}} & {[\Gamma]} \ar{d}{{[\pi]}} \\
      {[\Delta]} \ar{r}[swap]{{[f]}} & {\ft[\Gamma]}
    \end{tikzcd}
    \in \vbr{\bC}_\ETT
  \end{equation*}
  Then, by definition,
  $[f] [g] = [\Theta_1 \xrightarrow{g} \Delta_1 \xrightarrow{w_\Delta}
  \Delta \xrightarrow{f} \ft\Gamma]$ and
  $[\pi] [h] = [\Theta_2 \xrightarrow{h} \Gamma_2 \xrightarrow{w_\Gamma}
  \Gamma \xrightarrow{\pi} \ft\Gamma]$.
  By the diagram above and \Cref{lem:approx-char},
  $\pi w_\Gamma h w_\Theta \simeq f w_\Delta g$.
  Using the homotopy lifting property, we can find
  $\widetilde{h} \simeq w_\Gamma h w_\Theta \colon \Theta_1 \to \Gamma$ such that
  $\pi \widetilde{h} = f w_\Delta g$.
  Hence, in $\bC$, there is a unique factorisation
  $\Theta_1 \xrightarrow{\ell} f^*\Gamma$ such that
  \begin{equation*}
    \begin{tikzcd}
      \Theta_1 \ar[bend right=10]{dd}[swap]{g} \ar{r}{w_\Theta}
      \ar[bend left=10]{rrdd}[description]{\widetilde{h}}
      \ar[phantom, bend left=35]{rrdd}{\simeq}
      \ar[dashed]{rdd}[description]{\exists!\ell}
      & \Theta_2 \ar[bend left]{rd}{h} \\
              &           & \Gamma_2 \ar{d}{w_\Gamma} \\
     \Delta_1 \ar[bend right]{rd}[swap]{w_\Delta}
     & f^*\Gamma \ar{r}[description]{f.\Gamma} \ar[two heads]{d}
     \ar[phantom]{rd}[pos=0]{\lrcorner}
     & \Gamma \ar[two heads]{d} \\
              & \Delta \ar{r}[swap]{f}    & \ft\Gamma
    \end{tikzcd}
    \in \bC
  \end{equation*}
  Note that we have $[\pi] [\ell] = [\pi \ell] = [w_\Delta g] = [g]$ and
  $[f].[\Gamma] \cdot [\ell] = [f.\Gamma][\ell] = [f.\Gamma \cdot
  \ell] = [\widetilde{h}] = [w_\Gamma h w_\Theta] = [h]$.
  Thus, $[\ell]$ does provide a required factorisation in
  $\vbr{\bC}_\ETT$.

  It remains to check uniqueness.
  Suppose there is another map $\Theta' \xrightarrow{\ell'} Y$ where there
  is $w_\Theta' \colon \Theta_1 \tosimeq \Theta' \in \McW_\ETT$ and
  $w_Y \colon Y \tosimeq f^*\Gamma \in \McW_\ETT$ such that $[\ell']$ is also
  a factorisation where $[\pi] [\ell'] = [g]$ and
  $[f].[\Gamma] \cdot [\ell'] = [h]$.
  Because
  $[\Theta_1 \xrightarrow{g} \Delta_1] = [\pi][\ell'] = [\Theta'
  \xrightarrow{\ell'} Y \xrightarrow{w_Y} f^*\Gamma \xrightarrow{\pi}
  \Delta]$, it follows again by \Cref{lem:approx-char} that
  $\pi w_Y \ell'w_\Theta' \simeq w_\Delta g$.
  Let $\widetilde{\ell'}$ be homotopic to
  $w_Y \ell' w_\Theta'$ such that $\pi \widetilde{\ell'} = w_\Delta g$.
  Then, $[\widetilde{\ell'}] = [\ell']$ .
  Therefore,
  $[\Theta_2 \xrightarrow{h} \Gamma_2] = [f].[\Gamma] \cdot [\ell'] =
  [f.\Gamma][\widetilde{\ell'}] = [\Theta_1 \xrightarrow{\widetilde{\ell'}}
  f^*\Gamma \xrightarrow{f.\Gamma} \Gamma]$.
  Again by \Cref{lem:approx-char}, we see that
  $\widetilde{h} \simeq w_\Gamma h w_\Theta \simeq f.\Gamma \cdot
  \widetilde{\ell'}$.
  Hence, the situation is now:
  \begin{equation*}
    \begin{tikzcd}
      \Theta_1 \ar[bend right=10]{d}[swap]{g} \ar{rd}[description]{\widetilde{\ell'}}
      \ar[phantom, bend left=5]{rrd}{\simeq}
      \ar[bend left]{rrd}{\widetilde{h}} \\
      \Delta_1 \ar[bend right]{rd}[swap]{w_\Delta}
      & f^*\Gamma \ar[two heads]{d} \ar{r}[description]{f.\Gamma}
      \ar[phantom]{rd}[pos=0]{\lrcorner}
      & \Gamma \ar[two heads]{d} \\
               & \Delta \ar{r}[swap]{f}    & \ft\Gamma
    \end{tikzcd}
    \in \bC
  \end{equation*}
  By $\MsJ$-elimination on
  $f.\Gamma \cdot \widetilde{\ell'} \simeq \widetilde{h}$ and then
  using the uniqueness of $\ell$, it follows that
  $\widetilde{\ell'} \simeq \ell$, from which
  $[\ell'] = [\widetilde{\ell'}] = [\ell]$ follows.
\end{proof}

\begin{proof}[Proof of \Cref{prop:free-ett-cxl-cat}]
    Combine \Cref{lem:free-ett-basic-tt,lem:free-ett-subst-pullback}. From their proofs, we see $[-]$ is a map of contextual categories.
\end{proof}

\begin{lemma}\label{lem:strict-lifting}
  The map $\bC \xrightarrow{[-]} \vbr{\bC}_\ETT$ is a trivial fibration as in
  \Cref{def:ctx-triv-fib} for each $\McI^\CxlCat_{\ITT+\UIP}$-cellular
  $\bC \in \CxlCat_{\ITT+\UIP}$.
\end{lemma}
\begin{proof}
  We first show strict type lifting.
  Let $\Gamma \in \bC$ and $[\Delta] \in \vbr{\bC}_\ETT$ be such that
  $[\Gamma] = \ft{[\Delta]} = [\ft \Delta]$.
  Then, there is
  $w_\Gamma \colon \Gamma \tosimeq \ft \Delta \in \McW_\ETT$.
  Hence, by \Cref{lem:free-ett-subst-id}, it follows that
  $w_\Gamma.\Delta \colon w_\Gamma^*\Delta \tosimeq \Delta \in \McW_\ETT$ with
  $\ft{w_\Gamma^*\Delta} = \Gamma$.
  So $w_\Gamma^*\Delta$ is a type over $\Gamma$ such that
  $[w_\Gamma^*\Delta] = [\Delta]$, as required by strict type lifting.

  For strict term lifting, let there be $\Gamma \in \bC$ and a type
  $\Gamma.A \in \bC$ over $\Gamma$.
  Take any section
  $[a] \colon [\Gamma] \to [\Gamma.A] \in \vbr{\bC}_\ETT$ so that
  there is a representative $a \colon \Delta_1 \to \Delta_2$ with
  $w_1 \colon \Delta_1 \tosimeq \Gamma \in \McW_\ETT$ and
  $w_2 \colon \Delta_2 \tosimeq \Gamma.A \in \McW_\ETT$.
  Then, by assumption, $[\pi][a] = [\id]$ and so we have
  $\pi \cdot w_2 \cdot a \cdot w_1 \tosimeq \id$.
  By homotopy lifting, we are done by taking a map $\widehat{a}$
  homotopic to $w_2 \cdot a \cdot w_1$ so that
  $\pi \widehat{a} = \id$ (i.e. a section).
\end{proof}

%%% Local Variables:
%%% mode: latex
%%% TeX-master: "./main.tex"
%%% TeX-engine: xetex
%%% End:

%% file: c-ett-logic.tex
\section{Logical Structure on the Freely Generated Extensional Type
  Theory}\label{sec:c-ett-logic}
We next check that $\vbr{\bC}_\ETT$ inherits the respective logical
structures of $\bC$ by passing into the quotient.
They all proceed in the same manner.
Results in this section correspond to \cite[Proposition
3.2.17]{hofmann:extensional-constructs}.

\begin{definition}\label{def:c-ett-id}
  For each $[\Gamma] \in \vbr{\bC}_\ETT$ and $[A] \in \Ty{[\Gamma]}$,
  choose a representative $\Gamma$ of $[\Gamma]$ and, by
  \Cref{lem:strict-lifting}, a representative $A \in \Ty\Gamma$ of
  $[A]$.
  Define the following:
  \begin{itemize}
    \item \emph{$\Id$-type.} $\Id_{[A]} \coloneqq [\Id_A]$.
    \item \emph{$\refl$-constructor.}
    $\refl_{[A]} \coloneqq [\refl_A]$.
    \item \emph{$\MsJ$-eliminator.} For each
    $[C] \in \Ty{[\Gamma].[A].[A].\Id_{[A]}}$,
    choose, by \Cref{lem:strict-lifting}, a representative
    $C \in \Ty{\Gamma.A.A.\Id_{A}}$.
    Then, for each
    $[d] \colon [\Gamma].[A] \to [\Gamma].[A].[A].\Id_{[A]}.[C]$ where
    $[\pi][d] = [\refl_A]$, choose a representative
    $d \colon \Gamma.A \to \Gamma.A.A.\Id_A.C$ such that
    $\pi \cdot d = \refl_A$.
    Put $\MsJ_{[d]} \coloneqq [\MsJ_{d}]$.
  \end{itemize}
\end{definition}

\begin{lemma}\label{lem:c-ett-id}
  The structure given in \Cref{def:c-ett-id} is well-defined and gives rise to $\Id$-type
  structures in $\vbr{\bC}_\ETT$.
\end{lemma}
\begin{proof}
  Assume $w_{A} \colon \Gamma_1.A_1 \tosimeq \Gamma_2.A_2 \in \McW_\ETT$ so that
  by \Cref{lem:w-ett-grading-proj}\ref{itm:proj}, there exists
  $w_\Gamma \colon \Gamma_1 \tosimeq \Gamma_2 \in \McW_\ETT$ such that
  $w_A \xrightarrow{(\pi,\pi)} w_\Gamma \in \bC^\to$.
  By \Cref{lem:free-ett-subst-id} and stability under substitution, putting
  $A_0 \coloneqq w_\Gamma^*A_2$, we have
  \begin{align*}
    w.A_2.A_2.\Id_{A_2} \colon \Gamma_1.A_0.A_0.\Id_{A_0} \tosimeq
    \Gamma_2.A_2.A_2.\Id_{A_2} \in \McW_\ETT
  \end{align*}
  Also by stability under pullbacks, for the constructor,
  $w_\Gamma^*\refl_2 = \refl_0$.
  So,
  $(w_\Gamma.A_2, w_\Gamma.A_2.A_2.\Id_{A_2}) \colon \refl_0 \to
  \refl_2 \in \bC^\to$ and hence $[\refl_0] = [\refl_2]$.

  From \Cref{lem:free-ett-subst-id}, it also follows that
  $w_\Gamma.A_2 \colon \Gamma_1.A_0 \tosimeq \Gamma_2.A_2 \in \McW_\ETT$ and
  so in combination with
  $w_{AB} \colon \Gamma_1.A_1 \tosimeq \Gamma_2.A_2 \in \McW_\ETT$
  there is
  $w_0 \colon \Gamma_1.A_0 \tosimeq \Gamma_1.A_1 \in \McW_\ETT$.
  This allows for \Cref{lem:id-id} to apply and gives
  \begin{align*}
    w_\Id \colon \Gamma_1.A_1.A_1.\Id_{A_1} \tosimeq
    \Gamma_1.A_0.A_0.\Id_{A_0} \in \McW_\ETT
  \end{align*}
  such that $\refl_0 \xrightarrow{(w_0,w_\Id)} \refl_1$.
  Hence,
  $\Gamma_1.\Id_{A_1} \tosimeq \Gamma_1.\Id_{A_0} \tosimeq
  \Gamma_2.\Id_{A_2} \in \McW_\ETT$ and
  $[\refl_1] = [\refl_0] = [\refl_2]$.
  This shows that $\Id$-types and constructors descend into the
  quotient.

  For the $\MsJ$-eliminator, assume there is a map
  $\Delta_1 \xrightarrow{d} \Delta_2$ such that there is
  $w_{1} \colon \Gamma_1.A_0 \tosimeq \Delta_1 \in \McW_\ETT$ and
  $w_{2} \colon \Delta_2 \tosimeq \Gamma_1.A_0.A_0.\Id_{A_0}.C \in
  \McW_\ETT$.
  Further assume $[\pi] [d] = [\refl_0]$ so that
  $\pi \cdot w_2 \cdot d \cdot w_1 \simeq \refl_0$ by
  \Cref{lem:approx-char}.
  Replace $w_2 \cdot d \cdot w_1$ with homotopic $\widetilde{d}$ so
  that $\pi \widetilde{d} = \refl_0$.
  Then by \Cref{def:c-ett-id}, we have
  $\MsJ_{[d]} \coloneqq [\MsJ_{\widetilde{d}}]$.
  Any other
  $\widetilde{d}' \colon \Gamma_1.A_0 \to
  \Gamma_1.A_0.A_0.\Id_{A_0}.C$ obtained in the above manner by making
  different choices must be also such that $[d] = [\widetilde{d}']$.
  Hence, by \Cref{lem:approx-char}, $\widetilde{d}' \simeq \widetilde{d}$,
  from which $\MsJ_{\widetilde{d}'} \simeq \MsJ_{\widetilde{d}}$
  follows by $\MsJ$-elimination.
  This means
  $[\MsJ_{\widetilde{d}'}] \simeq [\MsJ_{\widetilde{d}}]$,
  so the recursor is well-defined.

  By functoriality of the quotient map, the necessary computation
  rules and substitutivity follows.
  For example, we have
  $\MsJ_{[d]} \cdot \refl_{[A]} = [\MsJ_{\widetilde{d}}] \cdot
  [\refl_0] = [\MsJ_{\widetilde{d}} \cdot \refl_0] =
  [\widetilde{d}] = [d]$.
  This completes the verification that $\vbr{\bC}_\ETT$ inherits
  $\Id$-type structures from $\bC$.
\end{proof}

\begin{definition}\label{def:c-ett-pi}
  For each $[\Gamma] \in \vbr{\bC}_\ETT$ and
  $[A] \in \Ty{[\Gamma].[A]}$ and $[B] \in \Ty{[\Gamma].[A]}$, choose
  a representative of $\Gamma$ of $[\Gamma]$ and, by
  \Cref{lem:strict-lifting}, representatives $A \in \Ty\Gamma$ and
  $B \in \Ty{\Gamma.A}$ of $[A]$ and $[B]$ respectively.
  Define the following:
  \begin{itemize}
    \item \emph{$\Pi$-type.} $\Pi([A],[B]) \coloneqq [\Pi(A,B)]$.
    \item \emph{$\lambda$-abstraction.} For each section
    $[b] \colon [\Gamma].[A] \to [\Gamma].[A].[B]$, by
    \Cref{lem:strict-lifting}, a section
    $b \colon \Gamma.A \to \Gamma.A.B$ that represents $[b]$.
    Then, define $\lambda([b]) \coloneqq [\lambda(b)]$.
    \item \emph{Application.} For pairs of sections
    $[k] \colon [\Gamma] \to [\Gamma].\Pi([A],[B])$ and
    $[a] \colon [\Gamma] \to [\Gamma].[A]$, choose, by
    \Cref{lem:strict-lifting}, representatives
    $k \colon \gamma \to \Gamma.\Pi(A,B)$ and
    $a \colon \gamma \to \Gamma.A$ that are also sections.
    Then, define $\app([k],[a]) \coloneqq [\app(k,a)]$.
    \item \emph{Functional extensionality.}
    $\ext_{[A],[B]} \coloneqq [\ext_{A,B}]$.
  \end{itemize}
\end{definition}

\begin{lemma}\label{lem:c-ett-pi}
  The structure given in \Cref{def:c-ett-pi} is well-defined and gives rise to
  $\Pi_\Ext$-structures in $\vbr{\bC}_\ETT$.
\end{lemma}
\begin{proof}
  Assume
  $w_{AB} \colon \Gamma_1.A_1.B_1 \tosimeq \Gamma_2.A_2.B_2 \in \McW_\ETT$ so
  that by \Cref{lem:w-ett-grading-proj}\ref{itm:proj}, there exist
  $w_A \colon \Gamma_1.A_1 \tosimeq \Gamma_2.A_2 \in \McW_\ETT$ and
  $w_\Gamma \colon \Gamma_1 \tosimeq \Gamma_2 \in \McW_\ETT$ such that
  $w_{AB} \xrightarrow{(\pi,\pi)} w_A \xrightarrow{(\pi,\pi)} w_\Gamma
  \in \bC^\to$.
  By stability under pullback and \Cref{lem:free-ett-subst-id}, by putting
  $A_0 \coloneqq w_\Gamma^*A_2$ and
  $B_0 \coloneqq (w_\Gamma.A_2)^*B_2$, we have
  $w_\Gamma.\Pi(A_2,B_2) \colon w_\Gamma^*(\Gamma_2.\Pi(A_2,B_2)) =
  \Gamma_1.\Pi(A_0, B_0) \tosimeq \Gamma_2.\Pi(A_2,B_2) \in
  \McW_\ETT$.

  From \Cref{lem:free-ett-subst-id}, it also follows that
  $w_\Gamma.A_2.B_2 \colon \Gamma_1.A_0.B_0 \tosimeq \Gamma_2.A_2.B_2 \in
  \McW_\ETT$ and so in combination with
  $w_{AB} \colon \Gamma_1.A_1.B_1 \tosimeq \Gamma_2.A_2.B_2 \in \McW_\ETT$
  there is
  $w_0 \colon \Gamma_1.A_0.B_0 \tosimeq \Gamma_1.A_1.B_1 \in \McW_\ETT$.
  This allows for \Cref{lem:id-pi} to apply and gives
  $w_\Pi \colon \Gamma_1.\Pi(A_1,B_1) \tosimeq \Gamma_1.\Pi(A_0,B_0)
  \in \McW_\ETT$.
  Hence,
  $\Gamma_1.\Pi(A_1,B_1) \tosimeq \Gamma_1.\Pi(A_0,B_0) \tosimeq
  \Gamma_2.\Pi(A_2,B_2) \in \McW_\ETT$.
  This shows that $\Pi$-types descend into the quotient.

  For $\lambda$-abstraction, assume there is a section
  $[b] \colon [\Gamma_1.A_0] \to [\Gamma_1.A_0.B_0]$.  and there are
  two lifts $b,b' \colon \Gamma_1.A_0 \to \Gamma_1.A_0.B_0$.
  Then, $[b] = [b']$, so \Cref{lem:approx-char} states $b \simeq b'$.
  Hence, by $\MsJ$-elimination, $\lambda(b) \simeq \lambda(b')$.
  This shows that $\lambda$ is well-defined.
  The same argument applies for $\app$ and $\ext$.

  It is also routine to check the necessary substitution and
  computation rules.
\end{proof}

\begin{definition}\label{def:c-ett-sigma}
  For each $[\Gamma] \in \vbr{\bC}_\ETT$ and
  $[A] \in \Ty{[\Gamma].[A]}$ and $[B] \in \Ty{[\Gamma].[A]}$, choose
  a representative of $\Gamma$ of $[\Gamma]$ and, by
  \Cref{lem:strict-lifting}, representatives $A \in \Ty\Gamma$ and
  $B \in \Ty{\Gamma.A}$ of $[A]$ and $[B]$ respectively.
  Define the following:
  \begin{itemize}
    \item \emph{$\Sigma$-type.}
    $\Sigma([A],[B]) \coloneqq [\Sigma(A,B)]$.
    \item \emph{$\pair$-constructor.}
    $\pair_{[A],[B]} \coloneqq [\pair_{A,B}]$.
    \item \emph{$\sigmarec$-eliminator.} For each
    $[C] \in \Ty{[\Gamma].[A].[B].\Sigma_{[A],[B]}}$ choose a
    representative $C \in \Ty{\Gamma.A.B.\Sigma_{A,B}}$.
    Then, for each
    $[d] \colon [\Gamma].[A] \to
    [\Gamma].[A].[B].\Sigma_{[A],[B]}.[C]$ where
    $[\pi][d] = \pair_{[A],[B]}$, choose a representative
    $d \colon \Gamma.A \to \Gamma.A.B.\Sigma_{A,B}.C$ such that
    $\pi \cdot d = \pair_{A,B}$.
    Put $\sigmarec_{[d]} \coloneqq [\sigmarec_{d}]$.
  \end{itemize}
\end{definition}

\begin{lemma}\label{lem:c-ett-sigma}
  The structure given in \Cref{def:c-ett-sigma}
  is well-defined and gives rise to $\Sigma$-type structures in
  $\vbr{\bC}_\ETT$.
\end{lemma}
\begin{proof}
  Assume
  $w_{AB} \colon \Gamma_1.A_1.B_1 \tosimeq \Gamma_2.A_2.B_2 \in \McW_\ETT$ so
  that by \Cref{lem:w-ett-grading-proj}\ref{itm:proj}, there exists
  $w_A \colon \Gamma_1.A_1 \tosimeq \Gamma_2.A_2 \in \McW_\ETT$ and
  $w_\Gamma \colon \Gamma_1 \tosimeq \Gamma_2 \in \McW_\ETT$ such that
  $w_{AB} \xrightarrow{(\pi,\pi)} w_A \xrightarrow{(\pi,\pi)} w_\Gamma
  \in \bC^\to$.
  By stability under pullback and \Cref{lem:free-ett-subst-id}, by putting
  $A_0 \coloneqq w_\Gamma^*A_2$ and
  $B_0 \coloneqq (w_\Gamma.A_2)^*B_2$, we have
  $w_\Gamma.\Sigma(A_2,B_2) \colon w_\Gamma^*(\Gamma_2.\Sigma(A_2,B_2)) =
  \Gamma_1.\Sigma(A_0, B_0) \tosimeq \Gamma_2.\Sigma(A_2,B_2) \in
  \McW_\ETT$.
  Also by stability under pullback, for the constructor,
  $w_\Gamma^*\pair_2 = \pair_0$ for
  $\pair_0$ and $\pair_2$ maps in the slices
  over $\Gamma_1$ and $\Gamma_2$ respectively.
  In particular, this means
  $(w_\Gamma.A_2.B_2, w_\Gamma.\Sigma(A_2,B_2)) \colon \pair_0 \to
  \pair_2 \in \bC^\to$ and hence
  $[\pair_0] = [\pair_2]$.

  From \Cref{lem:free-ett-subst-id}, it also follows that
  $w_\Gamma.A_2.B_2 \colon \Gamma_1.A_0.B_0 \tosimeq \Gamma_2.A_2.B_2 \in
  \McW_\ETT$ and so in combination with
  $w_{AB} \colon \Gamma_1.A_1.B_1 \tosimeq \Gamma_2.A_2.B_2 \in \McW_\ETT$
  there is
  $w_0 \colon \Gamma_1.A_0.B_0 \tosimeq \Gamma_1.A_1.B_1 \in \McW_\ETT$.
  This allows for \Cref{lem:id-sigma} to apply and gives
  $w_\Sigma \colon \Gamma_1.\Sigma(A_1,B_1) \tosimeq \Gamma_1.\Sigma(A_0,B_0)
  \in \McW_\ETT$ such that
  $\pair_1 \xrightarrow{(w_0,w_\Sigma)} \pair_0$.
  Hence,
  $\Gamma_1.\Sigma(A_1,B_1) \tosimeq \Gamma_1.\Sigma(A_0,B_0) \tosimeq
  \Gamma_2.\Sigma(A_2,B_2) \in \McW_\ETT$ and
  $[\pair_1] = [\pair_0] = [\pair_2]$.
  This shows that $\Sigma$-types and constructors descend into the
  quotient.

  For the recursor, assume there is a map
  $\Delta_1.X.Y \xrightarrow{d} \Delta_2.S.D$ such that there is
  $w_{XY} \colon \Gamma_1.A_0.B_0 \tosimeq \Delta_1.X.Y \in \McW_\ETT$ and
  $w_{SD} \colon \Delta_2.S.D \tosimeq \Gamma_1.\Sigma(A_0,B_0).C \in
  \McW_\ETT$.
  Further assume $[\pi] \cdot [d] = [\pair_0]$.
  Then, by \Cref{lem:approx-char}, we have
  $\pi \cdot w_{SD} \cdot d \cdot w_{XY} \simeq \pair_0$.
  Adjust $w_{SD} \cdot d \cdot w_{XY}$ to a homotopic $\widetilde{d}$
  so that $\pi \widetilde{d} = \pair_0$.
  Note that in particular
  $[d] = [d_0] = [\widetilde{d}]$.
  Put
  $\sigmarec_{[d]} \coloneqq [\sigmarec_{\widetilde{d}}]$.
  This is well-defined because any other
  $\widetilde{d}' \colon \Gamma_1.A_0.B_0 \to \Gamma_1.\Sigma(A_0,B_0).C$
  obtained in the above manner by making different choices must be
  also such that $[d] = [\widetilde{d}']$.
  Hence, by \Cref{lem:approx-char}, $\widetilde{d}' \simeq \widetilde{d}$,
  from which $\sigmarec_{\widetilde{d}'} \simeq \sigmarec_{\widetilde{d}}$
  follows by $\MsJ$-elimination.
  This means
  $[\sigmarec_{\widetilde{d}'}] \simeq [\sigmarec_{\widetilde{d}}]$,
  so the recursor is well-defined.

  By functoriality of the quotient map, the necessary computation
  rules and substitutivity follows.
  For example, we have
  $\sigmarec_{[d]} \cdot \pair([A],[B]) = [\sigmarec_{\widetilde{d}}]
  \cdot [\pair_0] = [\sigmarec_{\widetilde{d}} \cdot
  \pair_0] = [\widetilde{d}] = [d]$.
  This completes the verification that $\vbr{\bC}_\ETT$ inherits
  $\Sigma$-type structures from $\bC$.
\end{proof}

\begin{definition}\label{def:c-ett-unit}
  Assume $\bC$ admits $\Unit$-type structures.
  For each $[\Gamma] \in \vbr{\bC}_\ETT$, choose a representative of
  $\Gamma$ of $[\Gamma]$ and define the following:
  \begin{itemize}
    \item \emph{$\Unit$-type.}
    $\Unit_{[\Gamma]} \coloneqq [\Unit_\Gamma]$.
    \item \emph{$\star$-constructor.}
    $\star_{[\Gamma]} \coloneqq [\star_{\Gamma}]$.
    \item \emph{$\unitrec$-eliminator.} For each
    $[C] \in \Ty{[\Gamma].\Unit_{[\Gamma]}}$ choose a
    representative $C \in \Ty{\Gamma.\Unit_{\Gamma}}$.
    Then, for each
    $[d] \colon [\Gamma] \to
    [\Gamma].\Unit_{[\Gamma]}.[C]$ where
    $[\pi][d] = \star_{[\Gamma]}$, choose a representative
    $d \colon \Gamma \to \Gamma.\Unit_\Gamma.C$ such that
    $\pi \cdot d = \star_\Gamma$.
    Put $\unitrec_{[d]} \coloneqq [\unitrec_{d}]$.
  \end{itemize}
\end{definition}

\begin{lemma}\label{lem:c-ett-unit-ty}
  If $\bC$ admits $\Unit$-type structures then the structure given in \Cref{def:c-ett-unit}
  is well-defined and gives rise to $\Unit$-type structures in
  $\vbr{\bC}_\ETT$.
\end{lemma}
\begin{proof}
  Assume $w_\Gamma \colon \Gamma_1 \tosimeq \Gamma_2 \in \McW_\ETT$.
  Then, by \Cref{lem:free-ett-subst-id} and substitutivity of
  $\Unit$ types,
  $w_\Gamma^*(\Gamma_2.\Unit_2) = \Gamma_1.\Unit_1 \tosimeq
  \Gamma_2.\Unit_2 \in \McW_\ETT$.
  Also by substitutivity, we have
  $(w_\Gamma, w_\Gamma.\Unit_2) \colon \star_1 \to \star_2 \in \bC^\to$.
  So the definitions of $\Unit_{[\Gamma]} \coloneqq [\Unit_\Gamma]$
  and $\star_{[\Gamma]} \coloneqq [\star_\Gamma]$ is well-defined.

  Now, if $d \colon \Delta_1 \to \Delta_2$ where
  $w_1 \colon \Gamma_1 \tosimeq \Delta_1$ and
  $w_2 \colon \Delta_2 \tosimeq \Gamma_1.C$ in $\McW_\ETT$ is such that
  $[\pi][d] = [\star_1]$ then
  $\pi \cdot w_2 \cdot d \cdot w_1 \simeq \star_1$ by
  \Cref{lem:approx-char}.
  Replace $w_2 \cdot d \cdot w_1$ with homotopic $\widetilde{d}$ so
  that $\pi\widetilde{d} = \star_1$.
  Then, by \Cref{def:c-ett-unit},
  $\unitrec_{[d]} \coloneqq [\unitrec_{\widetilde{d}}]$.
  Repeating the same argument as in \Cref{lem:c-ett-sigma} using
  \Cref{lem:approx-char} and $\MsJ$-elimination, this is well-defined.

  It is routine to check the necessary substitution and computation
  rules.
\end{proof}

Note that the proofs of
\Cref{lem:c-ett-id,lem:c-ett-pi,lem:c-ett-sigma,lem:c-ett-unit} also
show that the quotient map $[-]$ preserves all logical structures.

%%% Local Variables:
%%% mode: latex
%%% TeX-master: "./main.tex"
%%% TeX-engine: xetex
%%% End:

%% file: c-ett-lift.tex
\section{Adjointness and Morita Equivalence}\label{sec:c-ett-lift}
With the explicit description of $\vbr{\bC}_\ETT$ for each cellular
$\bC \in \CxlCat_{\ITT+\UIP}$, we are now ready to prove the main result of this
paper, which is a generalisation of \cite[Theoerm
3.2.5]{hofmann:extensional-constructs}.

\begin{lemma}\label{lem:c-ett-extensional}
  $\vbr{\bC}_\ETT \in \CxlCat_\ETT$.
\end{lemma}
\begin{proof}
  We have already verified that $\vbr{\bC}_\ETT$ is a model of a type theory
  supporting $\Id$-types and extensional $\Pi$-types so it just remains to check
  it models the reflection rule.

  Let there be two sections $[a_i] \colon [\Gamma] \to [\Gamma.A]$ such
  that $[a_1] \simeq [a_2] \in \vbr{\bC}_\ETT$.
  by definition of substitution (\Cref{def:free-ett}) and the $\Id$-type
  (\Cref{def:c-ett-id}), this means there is a factorisation
  $[e] \colon [\Gamma] \to [\Gamma.A.A.\Id_A]$ such that
  \begin{equation*}
    \begin{tikzcd}
      {[\Gamma]} \ar{rd}[swap]{{[e]}} \ar{rr}{{([a_1],[a_2])}} & & {[\Gamma.A.A]} \\
      & {[\Gamma.A.A.\Id_A]} \ar[two heads]{ur}[swap]{{[\pi]}}
    \end{tikzcd}
  \end{equation*}
  First, we find a representative of $([a_1],[a_2])$.
  By strict term lifting, there is a choice of representatives
  $a_i \colon \Gamma \to \Gamma.A$ that are sections.
  And so by well-definedness of substitution pullbacks in
  $\vbr{\bC}_\ETT$ and functoriality of the quotient map,
  $([a_1],[a_2]) = [(a_1,a_2)]$.

  Now, let $e \colon \Gamma_0 \to \Delta$ be any representative of $[e]$ so that
  there are $w_\Gamma \colon \Gamma \tosimeq \Gamma_0 \in \McW_\ETT$ and
  $w_\Delta \colon \Delta \tosimeq \Gamma.A.A.\Id_A \in \McW_\ETT$.
  Then, $[(a_1,a_2)] = ([a_1],[a_2]) = [\pi][e] = [\pi w_\Delta e]$, it
  follows $(a_1,a_2) \simeq \pi \cdot w_\Delta \cdot e \cdot w_\Gamma$.
  By homotopy lifting, there is
  $\widetilde{e} \simeq w_\Delta \cdot e \cdot w_\Gamma$ such that
  $(a_1,a_2) \simeq \pi \widetilde{e}$.
  In particular, $a_1 \simeq a_2 \in \bC$ and so by defintion
  $[a_1] = [a_2] \in \vbr{\bC}_\ETT$.
\end{proof}

\begin{lemma}\label{lem:c-ett-unit}
  If $\bC$ is $\McI^\CxlCat_{\ITT+\UIP}$-cellular then
  $\bC \xrightarrow{[-]} \vbr{\bC}_\ETT$ is the unit of the adjunction
  $\vbr{-}^\circ_\ETT \dashv \abs{-}$ at $\bC$.
\end{lemma}
\begin{proof}
  We check the universal property of the unit.
  Let $F \colon \bC \to \abs{\bE} \in \CxlCat_{\ITT+\UIP}$ where
  $\bE \in \CxlCat_\ETT$.
  The goal is to show that there is a unique factorisation
  $\abs{\vbr{\bC}_\ETT} \xrightarrow{\widetilde{F}} \abs{\bE}$ through
  $\bC \xrightarrow{[-]} \abs{\vbr{\bC}_\ETT}$.
  \begin{equation*}
    \begin{tikzcd}
      \bC \ar[]{r}{[-]} \ar[]{rd}[swap]{F}
      & \abs{\vbr{\bC}_\ETT} \ar[dashed]{d}{\exists!\widetilde{F}} \\
      & \abs{\bE}
    \end{tikzcd}
  \end{equation*}

  If such $\widetilde{F}$ exists then clearly on objects
  $\widetilde{F}[\Gamma] = F\Gamma$ and on maps $\widetilde{F}[f] = Ff$.
  So uniqueness follows automatically after showing $F$ passes into
  the quotient.

  The fact that $F$ passes into the quotient (so that $\widetilde{F}$ exists)
  follows by \Cref{lem:ext-ker}.
  In particular the assignment of objects objects
  $[\Gamma] \mapsto F\Gamma$ is well-defined, because if
  $[\Gamma_1] = [\Gamma_2]$ then there is
  $w \colon \Gamma_1 \tosimeq \Gamma_2 \in \McW_\ETT$, from which it
  follows by \Cref{lem:ext-ker} that $Fw = \id$, so
  $F\Gamma_1 = F\Gamma_2$.
  Moreover, $[f] \mapsto F[f]$ is also well-defined because if
  $f_1,f_2$ are both representatives of $[f]$ then
  $f_1w_1 \simeq w_2f_2$ for $w_1,w_2 \in \McW_\ETT$ and so
  $Ff_1 = Ff_1 \cdot Fw_1 = F(f_1w_1) = F(f_2w_2) = Ff_2 \cdot Fw_2
  = Ff_2$ using \Cref{lem:ext-ker} and reflection in
  $\bE$.
  By functoriality of $F$ and $[-]$, it follows that this definition
  of $\widetilde{F}$ as above is indeed functorial.

  Next, note that $\widetilde{F} \colon \abs{\vbr{\bC}_\ETT} \to \abs{\bE}$
  is a map of contextual categories because it:
  \begin{itemize}
    \item \emph{Preserves the terminal object and grading.}
    This is because $F$ and $[-]$ do.
    For example,
    $\widetilde{F}1 = \widetilde{F}[1] = F1 = 1$.
    \item \emph{Preserves the father object and projection.}
    This is because $F$ does and context projection passes into the
    quotient as in \Cref{lem:w-ett-grading-proj}\ref{itm:proj}.
    So
    $\widetilde{F}(\ft{[\Gamma]}) = \widetilde{F}[\ft\Gamma] = F(\ft
    \Gamma) = \ft{F\Gamma} = \ft{\widetilde{F}[\Gamma]}$ and
    $\widetilde{F}\pi_{[\Gamma]} = \widetilde{F}[\pi_\Gamma] =
    F\pi_\Gamma = \pi_{F\Gamma} = \pi_{\widetilde{F}[\Gamma]}$.
    \item \emph{Is substitutive.}
    Like in the above cases, this is because $[F]$ and $[-]$ are.
  \end{itemize}

  Finally, we note that $\widetilde{F}$ preserves the logical structures.
  For example,
  $\widetilde{F}\Sigma([A],[B]) = \widetilde{F}[\Sigma(A,B)] =
  F\Sigma(A,B) = \Sigma(FA,FB) =
  \Sigma(\widetilde{F}[A],\widetilde{F}[B])$ and similarly for
  constructors and recursors.
  %
  \iffalse
    %
    Furthermore, if $\bC$ admits contractible types and $\bE$ admits
    unit types then $\widetilde{F} \colon \bC \to \abs{\bE}$ maps the
    contractible type in $\bC$ to the unit type (viewed as a
    contractible type) in $\bE$.
    %
    Indeed,
    $\widetilde{F}\MsK_{\Gamma} = F[\MsK_\Gamma] = F\Unit_{[\Gamma]} =
    \Unit_{F[\Gamma]} = \Unit_{\widetilde{F}\Gamma} =
    \MsK_{\widetilde{F}\Gamma}$ and
    $\widetilde{F}\Msk_{\Gamma} = F[\Msk_\Gamma] = F\star_{[\Gamma]} =
    \star_{F[\Gamma]} = \star_{\widetilde{F}\Gamma} =
    \Msk_{\widetilde{F}\Gamma}$.
    %
  \fi

  This shows that the unique factorisation $\widetilde{F}$ as required
  does indeed exist and is given by the above definition.
\end{proof}

\begin{theorem}
  The free-forgetful adjunction
  \begin{tikzcd}[cramped]
    \CxlCat_{\ITT+\UIP}
    \ar[r, {yshift=3}, "\vbr{-}_\ETT^\circ"]
    \ar[r, phantom, "\bot" {font=\tiny}]
    & \CxlCat_{\ETT}
    \ar[l, {yshift=-3}, "\abs{-}"]
  \end{tikzcd}
  satisfies the conditions of \Cref{def:morita} so that $\ITT+\UIP$ and $\ETT$
  are Morita equivalent.
\end{theorem}
\begin{proof}
  By \Cref{lem:c-ett-unit}, we indeed have an adjunction
  $\vbr{-}^\circ_\ETT \dashv \abs{-}$.
  The Quillen adjunction condition of \Cref{def:morita} follows because the
  forgetful functor $\abs{-} \colon \CxlCat_\ETT \to \CxlCat_{\ITT+\UIP}$
  clearly preserves fibrations and acyclic fibrations.
  By \Cref{lem:c-ett-unit} and \Cref{lem:strict-lifting} the weak equivalence
  condition of the unit as per \Cref{def:morita} holds for all
  $\McI^\CxlCat_{\ITT+\UIP}$-cellular models.
  To conclude the unit is a weak equivalence for all cofibrant models, note that
  $\McI^\CxlCat_{\ITT+\UIP}$ generates the cofibrations, so each cofibrant model
  $\bD \in \CxlCat_{\ITT+\UIP}$ has a cellular replacement
  $\bD' \xrightarrow{\simeq} \bD$.
  Because $\vbr{-}^\circ_\ETT \dashv \abs{-}$ is a Quillen adjunction and all
  objects of $\CxlCat_\ETT$ is fibrant by \cite[Proposition
  3.13]{kapulkin-lumsdaine-18}, it follows that
  $\abs{\vbr{\bD'}^\circ_\ETT} \xrightarrow{\simeq} \abs{\vbr{\bD}^\circ_\ETT}$
  is a weak equivalence.
  The required conclusion that the unit $\eta_\bD$ is a weak equivalence then
  follows from 2-out-of-3 and naturality of the unit.
\end{proof}

%%% Local Variables:
%%% mode: latex
%%% TeX-master: "./main.tex"
%%% TeX-engine: xetex
%%% End:

%% file: conclusion.tex
\section{Conclusion}\label{sec:conclusion}
%
%\color{red} %%% ← REMEMBER TO REMOVE THIS COMMAND AFTER EDITS
%
In this paper, we generalized the well-known result of Hofmann
\cite{hofmann:extensional-concepts,hofmann:extensional-constructs} that $\ETT$
is a conservative extension of $\ITT+\UIP$ by showing that these theories are in fact Morita equivalent in the sense of \cite{isaev18}.
%using the left semi-model structure provided from \cite{kapulkin-lumsdaine-18}.
%
This strengthens Hofmann's result, as instead of working solely with the initial model we show his result to be true in all cofibrant extensions (i.e., extensions that do not impose new definitional equalities) thereof.
Such extensions previously had to be handled one at a time, whereas the framework of Morita equivalence allows us to prove the result once and for all.
Our work further clarifies the relation between the notion of Morita equivalence, introduced in \cite{isaev18}, and the notion of conservative extensions of dependent type theories, studied in \cite{hofmann:extensional-concepts,hofmann:extensional-constructs}.

Based on results developed in this paper, there are three natural topics for further exploration:
\begin{enumerate}
  \item \label{itm:ext-1} \emph{Constructive proof of Morita equivalence.}
  Just like Hofmann's original proof \cite{hofmann:extensional-concepts,hofmann:extensional-constructs} and Bocquet's generalization \cite{bocquet:coherence}, our proof relies on the axiom of choice to work with the constructions of \Cref{sec:c-ett,sec:c-ett-logic}.
  Bocquet speculates \cite[\S2.1.(2)]{bocquet:coherence} that some instances of non-constructive principles could be avoided however further investigation into this question is needed.
  \item \label{itm:ext-2} \emph{Modular approach to the construction.}
  In the present paper, we rely on the notion of the left semi-model structure on the category of contextual categories with appropriate logical constructors \cite{kapulkin-lumsdaine-18} to be able to define Morita equivalence.
  However, our constructions are elementary and do not take advantage of the left semi-model structure, nor do they rely on the fibration category structure that each individual contextual category carries \cite{akl15}.
  A potential line of research is therefore to investigate whether our construction could be refactored to make it depend solely on the homotopical structure of the category of models of type theory and/or each model alone.
  This in particular raises the question of the homotopy-theoretic meaning of the ``extensional collapse.''
  Indeed, our construction of $\vbr{\bC}_\ETT$ shares similarities with the construction of the homotopy category of fibration category but the two differ in several ways.
  A further line of investigation would be to study this behavior.
  \item \label{itm:ext-3} \emph{Examples and further extensions of Morita theory.}
  Roughly speaking, the notions of conservativity and Morita equivalence of two
  theories express the fact that the two theories have the same ``logical power''.
  When defining logical structures in type theory, it is common to use
  different presentations for constructions which one expects to give rise to
  equivalent logical theories.
  Isaev \cite{isaev18} shows for instance the the theory with the unit type is equivalent to that with a contractible type (in the sense of homotopy type theory).
  One can further investigate other examples of Morita equivalences between dependent type theories.
  Somewhat independently, the notion of Morita equivalence could be generalized to other type theories where it could serve as a potential framework for phrasing the conservativity results between different cubical type theories and the conservativity of those over homotopy type theory.
\end{enumerate}

%
%\color{black} %%% ← REMEMBER TO REMOVE THIS COMMAND AFTER EDITS

%%% Local Variables:
%%% mode: latex
%%% TeX-master: "./main.tex"
%%% TeX-engine: xetex
%%% End: